\documentclass[11pt]{amsart}

\usepackage{amsfonts,amssymb,amsmath,amsthm,latexsym,graphics,epsfig}
\usepackage{verbatim,enumerate,array,booktabs,color,bigstrut,prettyref,tikz-cd}
\usepackage{multirow}
\usepackage[all]{xy}
\usepackage[backref]{hyperref}
\usepackage[OT2,T1]{fontenc}
\usepackage{ctable}

\usepackage{mathtools}

\newcommand\myiso{\stackrel{\mathclap{\normalfont\mbox{\small $p$}}}{-}}

\newrefformat{eq}{\textup{(\ref{#1})}}
\newrefformat{prty}{\textup{(\ref{#1})}}

\definecolor{mylinkcolor}{rgb}{0.8,0,0}
\definecolor{myurlcolor}{rgb}{0,0,0.8}
\definecolor{mycitecolor}{rgb}{0,0,0.8}
\hypersetup{colorlinks=true,urlcolor=myurlcolor,citecolor=mycitecolor,linkcolor=mylinkcolor,linktoc=page,breaklinks=true}

\DeclareSymbolFont{cyrletters}{OT2}{wncyr}{m}{n}
\DeclareMathSymbol{\Sha}{\mathalpha}{cyrletters}{"58}

\newtheorem{defn}{Definition}[section]
\newtheorem{definition}[defn]{Definition}

\newtheorem{corollary}[defn]{Corollary}
\newtheorem{lemma}[defn]{Lemma}

\newtheorem{thm}[defn]{Theorem}
\newtheorem{theorem}[defn]{Theorem}

\newtheorem{proposition}[defn]{Proposition}

\theoremstyle{definition}

\newtheorem*{ack}{Acknowledgements}
\newtheorem{remark}[defn]{Remark}

\newtheorem{example}[defn]{Example}

\newcommand{\QQ}{\mathbb Q}

\newcommand{\ZZ}{\mathbb Z}
\newcommand{\Z}{\mathbb Z}
\newcommand{\Q}{\mathbb{Q}}

\newcommand{\PP}{\mathbb P}

\newcommand{\Rats}{\Q}

\newcommand{\arrow}{\longrightarrow}

\newcommand{\Gal}{\operatorname{Gal}}

\newcommand{\GQ}{\Gal(\overline{\Rats}/\Rats)}

\newcommand{\GL}{\operatorname{GL}}

\newcommand{\tor}{\mathrm{tors}}

\begin{document}

\title[Isogeny graphs]{Infinite Families of Isogeny-Torsion Graphs}

\author{Garen Chiloyan}
\address{Department of Mathematics, University of Connecticut, Storrs, CT 06269, USA}
\email{garen.chiloyan@uconn.edu}
\urladdr{\url{https://sites.google.com/view/garenmath/home}}

\maketitle

\begin{abstract}
		Let $\mathcal{E}$ be a $\Q$-isogeny class of elliptic curves defined over $\Q$. The isogeny graph associated to $\mathcal{E}$ is a graph which has a vertex for each element of $\mathcal{E}$ and an edge for each $\Q$-isogeny of prime degree that maps one element of $\mathcal{E}$ to another element of $\mathcal{E}$, with the degree recorded as a label of the edge. The isogeny-torsion graph associated to $\mathcal{E}$ is the isogeny graph associated to $\mathcal{E}$ where, in addition, we label each vertex with the abstract group structure of the torsion subgroup over $\Q$ of the corresponding elliptic curve. The main result of the article is a determination of which isogeny-torsion graphs associated to $\Q$-isogeny classes of elliptic curves defined over $\Q$ correspond to infinitely many \textit{j}-invariants.  
	\end{abstract}

\section{Introduction}

Let $E/\QQ$ be an elliptic curve. It is well known that $E$ has a following Weierstrass model of the form
$$y^{2}z + a_{1}xyz + a_{3}yz^{2} = x^{3} + a_{2}x^{2}z + a_{4}xz^{2} + a_{6}z^{3}.$$
The Weierstrass model is usually dehomogenized. Moreover, $E$ has the structure of an abelian group with group identity $\mathcal{O} = [0:1:0]$. By the Mordell--Weil theorem, the set of points on $E$ defined over $\QQ$, denoted $E(\QQ)$ has the structure of a finitely generated abelian group. Thus, the set of points on $E$ defined over $\Q$ of finite order, denoted $E(\QQ)_{\text{tors}}$ is a finite group. By Mazur's theorem, $E(\QQ)_{\texttt{tors}}$ is isomorphic to one of fifteen groups (see Theorem \ref{thm-mazur}). Moreover, these fifteen groups occur infinitely often. Let $E'/\QQ$ be an elliptic curve. An isogeny mapping $E$ to $E'$ is a rational morphism $\phi \colon E \to E'$ such that $\phi$ maps the identity of $E$ to the identity of $E'$. If there is a non-constant isogeny defined over $\QQ$, mapping $E$ to $E'$, we say that $E$ is $\QQ$-isogenous to $E'$. This relation is an equivalence relation and the set of elliptic curves defined over $\QQ$ that are $\QQ$-isogenous to $E$ is called the $\QQ$-isogeny class of $E$.

Non-constant isogenies have finite kernels. We are particularly interested in non-constant isogenies with cyclic kernels. The isogeny graph associated to the $\QQ$-isogeny class of $E$ is a visual description of the $\QQ$-isogeny class of $E$. Denote the $\QQ$-isogeny class of $E$ by $\mathcal{E}$. The isogeny graph associated to $\mathcal{E}$ is a graph which has a vertex for each element of $\mathcal{E}$ and an edge for each $\Q$-isogeny of prime degree that maps one element of $\mathcal{E}$ to another element of $\mathcal{E}$, with the degree recorded as a label of the edge. The isogeny-torsion graph associated to $\mathcal{E}$ is the isogeny graph associated to $\mathcal{E}$ where, in addition, we label each vertex with the abstract group structure of the torsion subgroup over $\Q$ of the corresponding elliptic curve.

\begin{example}\label{T4 example}

There are four elliptic curves in the $\Q$-isogeny class with LMFDB label \texttt{17.a} which we will denote $E$, $E'$, $E''$, and $E'''$. The isogeny graph associated to \texttt{17.a} is on the left and the isogeny-torsion graph associated to \texttt{17.a} is on the right.

\begin{center} $\begin{tikzcd}
     & E'''                                 &      \\
     & E \arrow[u, "2", no head] \arrow[ld, "2"', no head] \arrow[rd, "2", no head] &      \\
E' &                                      & E''
\end{tikzcd} \hspace{10mm} \begin{tikzcd}
     & \Z / 2 \Z                                 &      \\
     & \Z / 2 \Z \times \Z / 2 \Z \arrow[u, "2", no head] \arrow[ld, "2"', no head] \arrow[rd, "2", no head] &      \\
\Z / 4 \Z &                                      & \Z / 4 \Z
\end{tikzcd}$
\end{center}

\end{example}

The classification of isogeny graphs associated to $\Q$-isogeny classes of elliptic curves over $\Q$ was probably known in the 1980's as it follows directly from the classification of non-cuspidal $\Q$-rational points on the modular curve $\operatorname{X}_{0}(N)$ for positive integers $N$. Nonetheless, a proof can be found in section 6 of \cite{gcal-r}.

\begin{thm}
	There are $26$ isomorphism types of isogeny graphs that are associated to elliptic curves defined over $\Q$. More precisely, there are $16$ types of (linear) $L_k$ graphs of $k = 1$-$4$ vertices, $3$ types of (nonlinear two-primary torsion) $T_k$ graphs of $k = 4$, $6$, or $8$ vertices, $6$ types of (rectangular) $R_k$ graphs of $k = 4$ or $6$ vertices, and $1$ (special) $S$ graph.
\end{thm}

In the case of a linear graph of $L_{2}$ or $L_{3}$ type or in the case of a rectangular graph of $R_{4}$ type, the degree of the maximal finite, cyclic $\Q$-isogeny of the isogeny graph is written in parentheses to distinguish it from other isogeny-torsion graphs of the same size and shape, but with different isogeny degree. For example, there are $L_{2}(2)$ graphs; graphs of $L_{2}$ type generated by an isogeny of degree $2$ and there are $L_{2}(3)$ graphs; isogeny graphs of $L_{2}$ type generated by an isogeny of degree $3$. Relying only on the size and shape of isogeny graphs of $L_{2}$ type is not enough to distinguish $L_{2}(2)$ and $L_{2}(3)$ isogeny graphs. The main theorem in \cite{gcal-r} was the classification of isogeny-torsion graphs associated to $\Q$-isogeny classes of elliptic curves over $\Q$.

\begin{thm}[Chiloyan, Lozano-Robledo, \cite{gcal-r}]
There are $52$ isomorphism types of isogeny-torsion graphs that are associated to $\Q$-isogeny classes of elliptic curves defined over $\Q$. In particular, there are $23$ isogeny-torsion graphs of $L_k$ type, $13$ isogeny-torsion graphs of $T_k$ type, $12$ isogeny-torsion graphs of $R_k$ type, and $4$ isogeny-torsion graphs of $S$ type.
\end{thm}

For each of the fifteen torsion subgroups, $G$, there are infinitely many \textit{j}-invariants, that correspond to elliptic curves $E$ defined over $\QQ$ such that $\textit{j}(E) = \textit{j}$ and $E(\QQ)_{\text{tors}} \cong G$. It is natural to ask if there is an analogous result for isogeny-torsion graphs. In other words, for which isogeny-torsion graphs $\mathcal{G}$, do there exist infinitely many \textit{j}-invariants that correspond to elliptic curves $E$ defined over $\QQ$ with $\textit{j}(E) = \textit{j}$ and the isogeny-torsion graph associated to the $\QQ$-isogeny class of $E$ is $\mathcal{G}$?

Let $\textit{j}(\mathcal{G})$ be the set of all \textit{j}-invariants of all elliptic curves $E$ defined over $\Q$ such that the isogeny-torsion graph associated to the $\Q$-isogeny class of $E$ is $\mathcal{G}$. We will say that $\textit{j}(\mathcal{G})$ is the set of \textit{j}-invariants associated to $\mathcal{G}$. The main result of this article is a determination of the isogeny-torsion graphs $\mathcal{G}$ such that $\textit{j}(\mathcal{G})$ is infinite. In the case that $\textit{j}(\mathcal{G})$ is infinite, we will say that $\mathcal{G}$ corresponds to an infinite set of $\textit{j}$-invariants and in the case that $\textit{j}(\mathcal{G})$ is finite, we will say that $\mathcal{G}$ corresponds to a finite set of $\textit{j}$-invariants.

Let $\mathcal{V}_{\mathcal{G}}$ be a vertex of $\mathcal{G}$. Let $\textit{j}(\mathcal{V}_{\mathcal{G}})$ be the set of \textit{j}-invariants of all elliptic curves $E$ defined over $\Q$ such that the isogeny-torsion graph associated to the $\Q$-isogeny class of $E$ is $\mathcal{G}$ and $E$ is represented by the vertex $\mathcal{V}_{\mathcal{G}}$. We will say that $\textit{j}(\mathcal{V}_{\mathcal{G}})$ is the set of \textit{j}-invariants associated to the vertex $\mathcal{V}_{\mathcal{G}}$. By Theorem \ref{thm-kenku}, each isogeny-torsion graph has at most $8$ vertices. Hence, an isogeny-torsion graph corresponds to an infinite set of \textit{j}-invariants if and only if any one of the vertices on the isogeny-torsion graph corresponds to an infinite set of \textit{j}-invariants.

\begin{example}\label{l4 example}

Consider a $\Q$-isogeny class of elliptic curves over $\QQ$ that contains four elliptic curves over $\Q$, $E_{1}$, $E_{2}$, $E_{3}$, and $E_{4}$ such that there is an isogeny $\phi \colon E_{1} \to E_{4}$ defined over $\QQ$, which has a cyclic kernel of order $27$. In other words, the isogeny-torsion graph associated to the $\QQ$-isogeny class is of $L_{4}$ (see the following diagram).

\begin{center}
    \includegraphics[scale=.15]{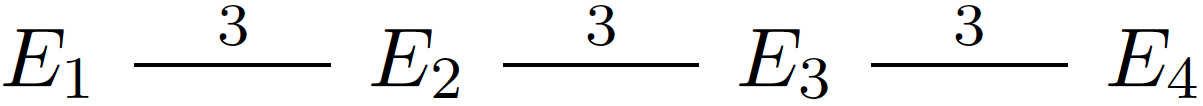}
\end{center}

The \textit{j}-invariants of $E_{1}$ and $E_{4}$ are equal to $-12288000$ and the \textit{j}-invariants of $E_{2}$ and $E_{3}$ are equal to $0$. Hence, there are two \textit{j}-invariants associated to an isogeny-torsion graph of $L_{4}$ type, regardless of torsion configuration. Thus, the two isogeny-torsion graphs of $L_{4}$ type both correspond to a finite set of $\textit{j}$-invariants.
\end{example}

\begin{example}
Let $\mathcal{G}$ be the isogeny-torsion graph in Example \ref{T4 example}.

\begin{center} $\begin{tikzcd}
     & E'''                                 &      \\
     & E \arrow[u, "2", no head] \arrow[ld, "2"', no head] \arrow[rd, "2", no head] &      \\
E' &                                      & E''
\end{tikzcd} \hspace{10mm} \begin{tikzcd}
     & \Z / 2 \Z                                 &      \\
     & \Z / 2 \Z \times \Z / 2 \Z \arrow[u, "2", no head] \arrow[ld, "2"', no head] \arrow[rd, "2", no head] &      \\
\Z / 4 \Z &                                      & \Z / 4 \Z
\end{tikzcd}$
\end{center}

The center vertex of the isogeny-torsion graph represents elliptic curves over $\QQ$ which

\begin{enumerate}
    \item have full two-torsion defined over $\Q$,
    \item do not contain a cyclic $\Q$-rational subgroup of order $4$,
    \item are $\Q$-isogenous to two non-isomorphic elliptic curves over $\Q$ with a point of order $4$ defined over $\Q$.
    \end{enumerate}
    
    Elliptic curves over $\Q$ represented by the center vertex of $\mathcal{G}$ correspond to non-cuspidal $\QQ$-rational points on the modular curve $\operatorname{X}_{24e}$ in the notation of \cite{Rouse}. Moreover, they have \textit{j}-invariant equal to 
    $$J = 2^{8} \cdot \frac{(t^{2}+t+1)^{3} \cdot (t^{2}-t+1)^{3}}{t^{4} \cdot (t^2+1)^{2}}$$
    for some non-zero rational $t$ or $J = 1728$. An argument using Hilbert's Irreducibility Theorem (see Section \ref{HIT} where we will elaborate on this idea further) shows that $\mathcal{G}$ corresponds to an infinite set of \textit{j}-invariants. 

\end{example}

The goal of this article is to prove the following statement:

\begin{thm}\label{main thm}

Let $\mathcal{G}$ be an isogeny-torsion graph associated to the $\Q$-isogeny class of an elliptic curve defined over $\Q$.

Then $\mathcal{G}$ corresponds to a finite set of \textit{j}-invariants if and only if

\begin{itemize}

\item $\mathcal{G}$ is of $L_{2}(p)$ type with isogeny degree $p = 11, 17, 19, 37, 43, 67$, or $163$, or

\item $\mathcal{G}$ is of $L_{4}$ type, or

\item $\mathcal{G}$ is of $R_{4}(pq)$ type with maximal, cyclic, isogeny degree $pq = 14, 15,$ or $21$.

\end{itemize}
\end{thm}

\subsection{Philosophy and structure of the paper}

The main ideas motivating in this paper is to think about elliptic curve theory, not necessarily from the viewpoint of individual elliptic curves over $\QQ$ but $\QQ$-isogeny classes of elliptic curves defined over $\QQ$ and the groups which generate said $\QQ$-isogeny classes.

Let $E/\QQ$ be an elliptic curve. Then $E(\QQ)_{\text{tors}}$ is isomorphic to one of fifteen groups. The main result in \cite{gcal-r} is the classification of the isogeny-torsion graphs associated to $\QQ$-isogeny classes of elliptic curves defined over $\QQ$. Originally, the authors in \cite{gcal-r} wanted to classify the torsion subgroups of a pair of $\QQ$-isogenous elliptic curves defined over $\QQ$, which was extended to classifying the torsion subgroups of all elliptic curves in a $\QQ$-isogeny class, which was extended to the main result of \cite{gcal-r}.

Consider again the isogeny graph of $L_{4}$ type like in Example \ref{l4 example}. Let $\mathcal{E}$ be a $\QQ$-isogeny class such that the isogeny graph associated to $\mathcal{E}$ is of $L_{4}$ type. By the classification of isogeny graphs, the isogeny graph of $L_{4}$ type is the only isogeny graph with an isogeny of degree $27$. Instead of thinking about the four individual elliptic curves in the isogeny graph, this paper considers the collection of the four elliptic curves in $\mathcal{E}$ simultaneously. Actually, the best way to think about $\mathcal{E}$ is that it is a $\QQ$-isogeny class containing an elliptic curve with a cyclic, $\QQ$-rational subgroup $H$ of order $27$. In other words, $\mathcal{E}$ contains an elliptic curve, $E$ corresponding to a non-cuspidal, $\QQ$-rational point on $\operatorname{X}_{0}(27)$; the modular curve generated by $B_{27} = \left\{\begin{bmatrix} \ast & \ast \\ 0 & \ast \end{bmatrix}\right\} \subseteq \operatorname{GL}(2, \ZZ / 27 \ZZ)$. The group $H$ is represented by the left column, $\left\{\begin{bmatrix} \ast \\ 0 \end{bmatrix} \right\}$ of $B_{27}$. In some sense, the group $H$ generates $\mathcal{E}$ because $H$ contains four distinct subgroups, each generating a unique elliptic curve that is $\QQ$-isogenous to $E$.

On the other hand, a similar but ultimately distinct analysis can be done with the the isogeny-torsion graph in Example \ref{T4 example}. Let $E/\QQ$ be an elliptic curve represented by the center vertex in the isogeny-torsion graph in Example \ref{T4 example}. Then 

\begin{itemize}
\item the $\QQ$-isogeny class of $E$ has four elliptic curves,

\item $E(\QQ)_{\text{tors}} \cong \ZZ / 2 \ZZ \times \ZZ / 2 \ZZ$,

\item $E$ is $\QQ$-isogenous to two non-isomorphic elliptic curves over $\QQ$, each with a point of order $4$ defined over $\QQ$
\end{itemize}
and hence, either the \textit{j}-invariant of $E$ is equal to $1728$ or $E$ is non-CM and corresponds to a non-cuspidal, $\QQ$-rational point on the modular curve $\operatorname{X}_{24e}$ using the notation in \cite{Rouse} and \texttt{4.24.0.7} using LMFDB notation.

Let $N$ be a positive integer such that $\operatorname{X}_{0}(N)$ has finitely many non-cuspidal $\QQ$-rational points. Then $N \in S = \{11$, $14, 15, 17, 19, 21, 27, 37, 43, 67, 163\}$ and the genus of $\operatorname{X}_{0}(N)$ is greater than or equal to $2$. The main result of this paper is obvious: Let $\mathcal{E}$ be a $\QQ$-isogeny class of elliptic curves over $\QQ$ and let $\mathcal{G}$ be its associated isogeny graph. Then $\mathcal{G}$ corresponds to a finite set of \textit{j}-invariants if and only if $\mathcal{E}$ contains an elliptic curve over $\QQ$ that corresponds to a non-cuspdial $\QQ$-rational point on $\operatorname{X}_{0}(N)$ for some $N$ in $S$. One direction of the previous statement is proven in Proposition \ref{Finite Graphs Proposition}. The other direction (proving that all other isogeny graphs correspond to infinite sets of \textit{j}-invariants) requires case by case analysis.

In section 2, we will go over the elementary algebraic properties of elliptic curves over $\QQ$ including Mazur's theorem and Kenku's theorem (the classification of non-cuspidal, $\QQ$-rational points on $\operatorname{X}_{0}(N)$). In section 3, we will summarize the work done by Rouse and Zureick-Brown (\cite{Rouse}) in classifying the $2$-adic Galois images attached to non-CM elliptic curves defined over $\QQ$ and the work done by Sutherland and Zywina in classifying the modular curves of prime-power level with infinitely many $\QQ$-rational points (\cite{SZ}). Section 4 covers lemmas from group theory that will be helpful. Section 5 contains the proof of Proposition \ref{Finite Graphs Proposition}, one direction of our main result. Section 7 contains the proof that all isogeny-torsion graphs of $S$ type correspond to infinite sets of \textit{j}-invariants, followed by section 8 which contains the proof that all isogeny-torsion graphs of $T_{k}$ type correspond to infinite sets of \textit{j}-invariants, followed by section 9 which determines which isogeny-torsion graphs of $R_{k}$ type correspond to infinite sets of \textit{j}-invariants, and concluded by section 10 which determines which isogeny-torsion graphs of $L_{k}$ type correspond to infinitely many \textit{j}-invariants.

Most of the proofs follow the same recipe. We determine whether a fixed isogeny-torsion graph $\mathcal{G}$ corresponds to an infinite set of \textit{j}-invariants. This can be done by seeing, for example, if $\mathcal{G}$ contains a unique point defined over $\QQ$. For example, the isogeny-torsion graph $\mathcal{L}_{2}^{1}(7)$ is the isogeny-torsion graph of $L_{2}(7)$ with torsion configuration $([7],[1])$ (see below).

\begin{center} \begin{tikzcd}
\mathbb{Z} / 7 \mathbb{Z} \arrow[r, "7", no head] & \left\{ \mathcal{O} \right\}
\end{tikzcd} \end{center}

Moreover, $\mathcal{L}_{2}^{1}(7)$ is the only isogeny-torsion graph containing a point of order $7$ defined over $\QQ$. As there are infinitely many \textit{j}-invariants corresponding to elliptic curves over $\QQ$ that have a point of order $7$ defined over $\QQ$, $\mathcal{L}_{2}^{1}(7)$ corresponds to an infinite set of \textit{j}-invariants. All isogeny-torsion graphs that are acquired from quadratic twisting an elliptic curve in $\mathcal{G}$ also correspond to an infinite set of \textit{j}-invariants and we use this method many times. For example, we use quadratic twists to prove that the isogeny-torsion graph $\mathcal{L}_{2}^{2}(7)$; the isogeny graph of $L_{2}(7)$ type with torsion configuration $([1],[1])$ corresponds to an infinite set of \textit{j}-invariants (see below).
\begin{center} \begin{tikzcd}
\left\{\mathcal{O}\right\} \arrow[r, "7", no head] & \left\{ \mathcal{O} \right\}
\end{tikzcd} \end{center}
We may simply take an elliptic curve $E/\QQ$ such that $E(\QQ)_{\text{tors}} \cong \ZZ / 7 \ZZ$ and twist it by an appropriate integer and the isogeny-torsion graph of the quadratic twist will be $\mathcal{L}_{2}^{2}(7)$.

We also make use of Hilbert's irreducibility theorem to prove some isogeny-torsion graphs correspond to infinite sets of \textit{j}-invariants. For example, the isogeny-torsion graph $\mathcal{L}_{2}^{1}(5)$ is the ``simplest'' isogeny-torsion graph containing a point of order $5$ defined over $\QQ$ (see below).

\begin{center} \begin{tikzcd}
\mathbb{Z} / 5 \mathbb{Z} \arrow[r, "5", no head] & \left\{ \mathcal{O} \right\}
\end{tikzcd} \end{center}
For example, we use Hilbert's irreducibility theorem to prove that $\mathcal{L}_{2}^{1}(5)$ corresponds to an infinite set of \textit{j}-invariants using the fact that the sets of \textit{j}-invariants corresponding to the isogeny-torsion graphs that properly contain $\mathcal{L}_{2}^{1}(5)$ constitute thin sets in the language of Serre. The isogeny graph of $\mathcal{L}_{1}$ type corresponds to an infinite set of \textit{j}-invariants. The proof of this is the final proof of the paper and again makes use of Hilbert's irreducibility theorem.

\begin{ack}
The author would like to express his gratitude to his advisor \'Alvaro Lozano-Robledo for his patience and many helpful conversations on this topic. The author would also like to thank Harris Daniels for conversations about non-trivial entanglements. The author would also like to thank those, including John Voight, who initially asked which isogeny-torsion graphs correspond to an infinite set of \textit{j}-invariants. The author would like to thank the referee for their many helpful comments.
\end{ack}

\section{Background}

Let $E$ be an elliptic curve defined over $\QQ$. Denote the set of $\QQ$-rational points on $E$ by $E(\QQ)$. Then $E(\QQ)$ has the structure of a finitely generated abelian group. Denote the set of points on $E$ defined over $\QQ$ of finite order by $E(\QQ)_{\text{tors}}$. Then, $E(\QQ)_{\text{tors}}$ is isomorphic to one of $15$ groups.

\begin{thm}[Mazur \cite{mazur1}]\label{thm-mazur}
		Let $E/\Q$ be an elliptic curve. Then
		\[
		E(\Q)_\tor\simeq
		\begin{cases}
		\Z/M\Z &\text{with}\ 1\leq M\leq 10\ \text{or}\ M=12,\ \text{or}\\
		\Z/2\Z\oplus \Z/2N\Z &\text{with}\ 1\leq N\leq 4.
		\end{cases}
		\]
	\end{thm}
Let $N$ be a positive integer. The points on $E$ of order dividing $N$ with coordinates in $\overline{\Q}$ form a finite group, denoted $E[N]$ which is isomorphic to $\Z / N \Z \times \Z / N \Z$. An element of $E[N]$ is called an $N$-torsion point. The group $G_{\Q}:= \operatorname{Gal}(\overline{\Q}/\Q)$ acts on $E[N]$ for all positive integers $N$. From this action, we have the mod-$N$ Galois representation attached to $E$: $\overline{\rho}_{E,N} \colon G_{\Q} \to \operatorname{Aut}(E[N])$.

After identifying $E[N] \cong \Z / N \Z \times \Z / N \Z$ and fixing a set of (two) generators of $E[N]$, we may consider the mod-$N$ Galois representation attached to $E$ as $\overline{\rho}_{E,N} \colon G_{\Q} \to \operatorname{GL}(2,N)$. Let $u$ be an element of $\left(\Z / N \Z\right)^{\times}$. By the properties of the Weil pairing, there exists an element of the image of $\overline{\rho}_{E,N}$ that has determinant $u$. A subgroup $H$ of $\operatorname{GL}(2, \Z / N \Z)$ is said to have \textit{full determinant mod-N} if $\det(H)=(\Z/N\Z)^\times$. Moreover, the image of $\overline{\rho}_{E,N}$ contains an element that represents complex conjugation. If $E$ is non-CM, then by Lemma 2.8 in \cite{SZ}, complex conjugation is represented by $\left(\begin{array}{cc}
    1 & 1 \\
    0 & -1
\end{array}\right)$ or $\left(\begin{array}{cc}
    1 & 0 \\
    0 & -1
\end{array}\right)$. Matrices in $\operatorname{GL}(2, \Z / N \Z)$ that are conjugate to $\left(\begin{array}{cc}
    1 & 1 \\
    0 & -1
\end{array}\right)$ or $\left(\begin{array}{cc}
    1 & 0 \\
    0 & -1
\end{array}\right)$ will be called \textit{representatives of complex conjugation mod-N}.

Let $E$ and $E'$ be elliptic curves defined over $\Q$. An isogeny mapping $E$ to $E'$ is a non-constant rational morphism $\phi \colon E \to E'$ that maps the identity of $E$ to the identity of $E'$. Isogenies are group homomorphisms with kernels of finite order. The degree of an isogeny agrees with the order of its kernel.

\begin{defn}
Let $E/\QQ$ be an elliptic curve. A subgroup $H$ of $E$ of finite order is said to be $\Q$-rational if $\sigma(H) = H$ for all $\sigma \in G_{\Q}$.
\end{defn}

\begin{remark}

Note that for an elliptic curve $E / \QQ$, a group generated by a point $P$ on $E$ defined over $\QQ$ of finite order is certainly a $\QQ$-rational group but in general, the elements of a $\QQ$-rational subgroup of $E$ need not be \textit{fixed} by $G_{\QQ}$. For example, $E[3]$ is a $\QQ$-rational group but certainly, the group $G_{\QQ}$ fixes one or three of the nine points of $E[3]$ by Theorem \ref{thm-mazur}.

\end{remark}

\begin{lemma}[Proposition III.4.12, \cite{Silverman}]\label{Q-rational}

Let $E$ be an elliptic curve over $\Q$ and let $N$ be a positive integer. Then for each cyclic, $\Q$-rational subgroup $H$ of $E$ of order $N$, there is a unique elliptic curve defined over $\Q$ up to isomorphism denoted $E / H$, and an isogeny $\phi_{H} \colon E \to E / H$ with kernel $H$.

\end{lemma}

\begin{remark}

Note that it is only the elliptic curve that is unique (up to isomorphism) but the isogeny is not. For any isogeny $\phi$, the isogeny $-\phi$ has the same domain, codomain, and kernel as $\phi$.

Moreover, for any integer $N$ the isogeny $\phi$ and the isogeny $[N] \circ \phi$ have the same domain and the same codomain. This is why the bijection in Lemma \ref{Q-rational} is with \textit{cyclic}, $\Q$-rational subgroups of an elliptic curve instead of with all $\Q$-rational subgroups of an elliptic curve.

\end{remark}

Let $E / \Q$ be an elliptic curve and let $N$ be a positive integer. Then there is a one to one bijective correspondence between the non-cuspidal, $\Q$-rational points on the modular curve $\operatorname{X}_{0}(N)$ and elliptic curves over $\Q$ up to isomorphism with a cyclic, $\Q$-rational subgroup of order $N$. Work by Fricke, Kenku, Klein, Kubert, Ligozat, Mazur and Ogg, among others contributed to the classification of the non-cuspidal, $\Q$-rational points on $\operatorname{X}_{0}(N)$ (see the summary tables in \cite{lozano0}).

\begin{thm}\label{thm-ratnoncusps} Let $N\geq 2$ be a positive integer such that $\operatorname{X}_0(N)$ has a non-cuspidal, $\Q$-rational point. Then:
		\begin{enumerate}
			\item $2 \leq N \leq 10$, or $N= 12,13, 16,18$ or $25$. In this case $\operatorname{X}_0(N)$ is a curve of genus $0$ and its $\Q$-rational points form an infinite $1$-parameter family, or
			
			\item $N=11,14,15,17,19,21$, or $27$. In this case $\operatorname{X}_0(N)$ is a curve of genus $1$, i.e.,~$\operatorname{X}_0(N)$ is an elliptic curve over $\Q$, but in all cases the Mordell-Weil group $\operatorname{X}_0(N)(\Q)$ is finite, or 
			
			\item $N=37,43,67$ or $163$. In this case $\operatorname{X}_0(N)$ is a curve of genus $\geq 2$ and (by Faltings' theorem) there are only finitely many $\Q$-rational points, which are known explicitly.
		\end{enumerate}
	\end{thm}
	Kenku's theorem is a reformulation of Theorem \ref{thm-ratnoncusps} and gives a clear way to classify the isogeny graphs associated to $\Q$-isogeny classes of elliptic curves over $\Q$. But before we state Kenku's theorem, we need a definition.

\begin{defn}
		Let $E/\Q$ be an elliptic curve. We define $C(E)$ to be the number of finite, cyclic, $\Q$-rational subgroups of $E$ (including the trivial subgroup), and we define $C_p(E)$ similarly to $C(E)$ but only counting cyclic, $\Q$-rational subgroups of order a power of $p$ (like in the definition of $C(E)$, this includes the trivial subgroup), for each prime $p$. 
	\end{defn}

	\begin{thm}[Kenku, \cite{kenku}]\label{thm-kenku} There are at most eight $\Q$-isomorphism classes of elliptic curves in each $\Q$-isogeny class. More concretely, let $E / \Q$ be an elliptic curve. Then $C(E)=\prod_p C_p(E)\leq 8$ and each factor $C_p(E)$ is bounded as follows:
		\begin{center}
			\begin{tabular}{c|ccccccccccccc}
				$p$ & $2$ & $3$ & $5$ & $7$ & $11$ & $13$ & $17$ & $19$ & $37$ & $43$ & $67$ & $163$ & \text{else}\\
				\hline 
				$C_p(E) \leq $ & $8$ & $4$ & $3$ & $2$ & $2$ & $2$ & $2$ & $2$ & $2$ & $2$ & $2$ & $2$ & $1$.
			\end{tabular}
		\end{center}
		Moreover:
		\begin{enumerate}
			\item If $C_{p}(E) = 2$ for a prime $p$ greater than $7$, then $C_{q}(E) = 1$ for all other primes $q$. 
			\item If $C_{7}(E) = 2$, then $C(E) \leq 4$. Moreover, we have $C_{3}(E) = 2$, or $C_{2}(E) = 2$, or $C(E) = 2$.
			\item $C_{5}(E) \leq 3$ and if $C_{5}(E) = 3$, then $C(E) = 3$.
			\item If $C_{5}(E) = 2$, then $C(E) \leq 4$. Moreover, either $C_{3}(E) = 2$, or $C_{2}(E) = 2$, or $C(E) = 2$. 
			\item $C_{3}(E) \leq 4$ and if $C_{3}(E) = 4$, then $C(E) = 4$. 
			\item If $C_{3}(E) = 3,$ then $C(E) \leq 6$. Moreover, $C_{2}(E) = 2$ or $C(E) = 3$.
			\item If $C_{3}(E) = 2$, then $C_{2}(E) \leq 4$.
		\end{enumerate}
	\end{thm}

\begin{remark}\label{2-powers}

For the rest of the paper, if an elliptic curve $E / \Q$ has full two-torsion defined over $\Q$, then we will say that $E[2] = E[2](\Q) = \left\langle P_{2}, Q_{2} \right\rangle$ for some $P_{2}, Q_{2} \in E[2]$. For an integer $M \geq 2$, let $P_{2^{M}} \in E\left[2^{M}\right]$ such that $[2]P_{2^{M}} = P_{2^{M-1}}$ and let $Q_{2^{M}} \in E\left[2^{M}\right]$ such that $[2]Q_{2^{M}} = Q_{2^{M-1}}$.

If an elliptic curve $E / \Q$ has full two-torsion defined over $\Q$ and contains a cyclic, $\Q$-rational subgroup of order $4$, then we may assume without loss of generality, that the cyclic groups $\{\mathcal{O}\}$, $\left\langle P_{2} \right\rangle$, $\left\langle Q_{2} \right\rangle$, $\left\langle P_{2} + Q_{2} \right\rangle$, $\left\langle Q_{4} \right\rangle$, and $\left\langle P_{2} + Q_{4} \right\rangle$ are $\Q$-rational. If an elliptic curve $E / \Q$ has full two-torsion defined over $\Q$ and $C_{2}(E) = C(E) = 8$, then either $P_{4}, P_{4} + Q_{2}, Q_{4},$ and $P_{2}+Q_{4}$ generate distinct $\Q$-rational subgroups of order $4$; or $Q_{8}$ and $P_{2} + Q_{8}$ generate distinct $\Q$-rational subgroups of order $8$. Thus, after a relabeling of the elliptic curves in the $\Q$-isogeny class of $E$, the $\Q$-rational subgroups of $E$ are $\{\mathcal{O}\}$, $\left\langle P_{2} \right\rangle$, $\left\langle Q_{2} \right\rangle$, $\left\langle P_{2} + Q_{2} \right\rangle$, $\left\langle Q_{4} \right\rangle$, $\left\langle P_{2} + Q_{4} \right\rangle$, $\left\langle Q_{8} \right\rangle$, and $\left\langle P_{2} + Q_{8} \right\rangle$. As a consequence, if $E(\Q)_{\text{tors}} \cong \Z / 2 \Z \times \Z / 4 \Z$, then $E(\Q)_{\text{tors}} = \left\langle P_{2}, Q_{4} \right\rangle$ and if $E(\Q)_{\text{tors}} \cong \Z / 2 \Z \times \Z / 8 \Z$, then $E(\Q)_{\text{tors}} = \left\langle P_{2}, Q_{8} \right\rangle$.

\end{remark}

\begin{defn}
Let $\mathcal{E}$ be a $\Q$-isogeny class of elliptic curves defined over $\Q$. The isogeny graph associated to $\mathcal{E}$ is a graph, which has a vertex for each elliptic curve in $\mathcal{E}$ and an edge for each $\Q$-isogeny of prime degree that maps one element of $\mathcal{E}$ to another element of $\mathcal{E}$, with the degree recorded as the label of the edge.
\end{defn}
The classification of isogeny graphs associated to $\Q$-isogeny classes of elliptic curves defined over $\Q$ follows from Theorem \ref{thm-kenku}. A detailed proof is outlined in Section 6 of \cite{gcal-r}.

\begin{theorem}\label{thm-mainisogenygraphs}
	There are $26$ isomorphism types of isogeny graphs associated to $\Q$-isogeny classes of elliptic curves defined over $\Q$. More precisely, there are $16$ types of (linear) $L_k$ graphs, $3$ types of (non-linear, two-primary torsion) $T_k$ graphs, $6$ types of (rectangular) $R_k$ graphs, and $1$ (special) $S$ graph. The degree of the maximal, finite, cyclic $\Q$-isogeny is written in parentheses for the $L_{2}$, $L_{3}$, and $R_{4}$ isogeny graphs.
	\end{theorem}
	
\begin{defn}
Let $\mathcal{E}$ be a $\Q$-isogeny class of elliptic curves defined over $\Q$. The isogeny-torsion graph associated to $\mathcal{E}$ is the isogeny graph associated to $\mathcal{E}$ with the vertices labeled with the abstract group structure of the torsion subgroup over $\Q$ of the corresponding elliptic curve.
\end{defn}

See Example \ref{T4 example} for an example of an isogeny graph and an isogeny-torsion graph associated to the $\Q$-isogeny class of an elliptic curve over $\Q$. In \cite{gcal-r} the authors classified the isogeny-torsion graphs associated to $\Q$-isogeny classes of elliptic curves over $\Q$.

\begin{thm}[\cite{gcal-r}]\label{thm-main2} There are $52$ isomorphism types of isogeny-torsion graphs that are associated to $\Q$-isogeny classes of elliptic curves over $\Q$. In particular, there are $23$ isogeny-torsion graphs of $L_{k}$ type, $13$ isogeny-torsion graphs of $T_{k}$ type, $12$ isogeny-torsion graphs of $R_{k}$ type, and $4$ isogeny-torsion graphs of $S$ type.
\end{thm}

\begin{remark}
		An application of Hilbert's Irreducibility Theorem (see Proposition \ref{HIT proposition} and Remark \ref{Hilbert's Irreducibility Theorem}) can be used to prove that, in terms of density, the wide majority of $\Q$-isogeny classes of elliptic curves over $\Q$ have an associated isogeny graph of $L_{1}$ type. In other words, an elliptic curve over $\Q$ with a finite, cyclic $\QQ$-rational subgroup of order $\geq 2$ is rare.
		\end{remark}

\subsection{Quadratic Twists}

Let $E : y^{2} = x^{3} + Ax + B$ be an elliptic curve defined over $\Q$. Let $d$ be a non-zero, square-free integer. The quadratic twist of $E$ by $d$ is the elliptic curve $E^{(d)} : dy^{2} = x^{3} + Ax + B$. Note that $\textit{j}(E) = \textit{j}(E^{(d)})$.

Let $N$ be a positive integer. Let $\rho_{N} \colon G_{\Q} \to \operatorname{Aut}(E[N])$ and $\rho^{d}_{N} \colon G_{\Q} \to \operatorname{Aut}(E^{d}[N])$ be the mod-$N$ Galois representations attached to $E$ and $E^{(d)}$ respectively. There is a quadratic character $\chi_{d}$ such that $\chi_{d} \circ \rho_{N} = \rho_{N}^{d}$. Let $G_{1}$ and $G_{2}$ be subgroups of $\operatorname{GL}(2, \Z / N \Z)$. We will say that $G_{1}$ is a quadratic twist of $G_{2}$ if $\left\langle G_{1}, \operatorname{-Id} \right\rangle = \left\langle G_{2}, \operatorname{-Id} \right\rangle$.

\subsection{Maps to the \textit{j}-line}

For more information, on this following subsection, see Section 2 in \cite{SZ}.

Let $N$ be a positive integer. Let $\mathcal{F}_{N}$ denote the field of meromorphic functions of the Riemann surface $\operatorname{X}(N)$ whose $q$-expansions have coefficients in $K_{N} := \Q(\zeta_{N})$. For $f \in \mathcal{F}_{N}$ and $\gamma \in \operatorname{SL}(2, \Z)$, let $f \textbar_{\gamma} \in \mathcal{F}$ denote the modular function satisfying $f \textbar_{\gamma}(\tau) = f(\gamma \tau)$. For each $d \in (\Z / N \Z)^{\times}$, let $\sigma_{d}$ be the automorphism of $K_{N}$ such that $\sigma_{d}(\zeta_{N}) = \zeta_{N}^{d}$.

\begin{proposition}
The extension $\mathcal{F}_{N}$ of $\mathcal{F}_{1} = \Q(\textit{j})$ is Galois. There is a unique isomorphism
$$ \theta_{N} \colon \operatorname{GL}(2, \Z / N \Z) / \{\pm I\} \to \operatorname{Gal}(\mathcal{F}_{N} / \Q(\textit{j})) $$
such that the following hold for all $f \in \mathcal{F}_{N}$:

\begin{enumerate}

    \item for $g \in \operatorname{SL}(2, \Z / N \Z)$, we have $\theta_{N}(g)f = f \textbar_{\gamma^{t}}$ where $\gamma$ is any matrix in $\operatorname{SL}(2, \Z)$ that is congruent to $g$ modulo $N$.
    
    \item For $g = \left(\begin{array}{cc}
        1 & 0 \\
        0 & d
    \end{array}\right) \in \operatorname{GL}(2, \Z / N \Z)$, we have $\theta_{N}(g)f = \sigma_{d}(f)$.

\end{enumerate}
\end{proposition}
A subgroup $G$ of $\operatorname{GL}(2, \Z / N \Z)$ containing $\operatorname{-Id}$ and with full determinant mod-$N$ acts on $\mathcal{F}_{N}$ by $g \cdot f = \theta_{N}(g)f$ for $g \in G$ and $f \in \mathcal{F}_{N}$. Let $\mathcal{F}_{N}^{G}$ denote the subfield of $\mathcal{F}_{N}$ fixed by the action of $G$. The modular curve $\operatorname{X}_{G}$ associated to $G$ is the smooth projective curve with function field $\mathcal{F}^{G}_{N}$. The inclusion of fields $\mathcal{F}_{1} = \Q(\textit{j}) \subseteq \mathcal{F}_{N}$ induces a non-constant morphism
$$ \pi_{G} \colon \operatorname{X}_{G} \longrightarrow \operatorname{Spec}\Q[\textit{j}] \cup \{\infty\} = \PP^{1}_{\Q} $$
of degree $[\operatorname{GL}(2, \Z / N \Z) : G]$. If there is an inclusion of groups, $G \subseteq G' \subseteq \operatorname{GL}(2, \Z / N \Z)$, then there is an inclusion of fields $\Q(\textit{j}) \subseteq \mathcal{F}_{N}^{G'} \subseteq \mathcal{F}_{N}^{G}$ which induces a non-constant morphism $g \colon X_{G} \longrightarrow X_{G'}$ of degree $[G' \colon G]$ such that the following diagram commutes.
$$\begin{tikzcd}
X_{G} \arrow[dd, "g"'] \arrow[rrdd, "\pi_{G}"] &  &                \\
                                               &  &                \\
X_{G'} \arrow[rr, "\pi_{G'}"']                 &  & \mathbb{P}^{1}
\end{tikzcd}$$
For an elliptic curve $E/ \Q$ with $\textit{j}(E) \neq 0, 1728$ the group $\overline{\rho}_{E,N}(G_{\Q})$ is conjugate in $\operatorname{GL}(2,\Z / N \Z)$ to a subgroup of $G$ if and only if $\textit{j}(E)$ is an element of
$\pi_{G}(\operatorname{X}_{G}(\Q))$.

\section{Work by Rouse--Zureick-Brown and Sutherland-Zywina}

\begin{thm}[Rouse, Zureick-Brown, \cite{Rouse}]\label{thm-rzb} Let $E$ be an elliptic curve over $\Q$ without complex multiplication. Then, there are exactly $1208$ possibilities for the $2$-adic image $\overline{\rho}_{E,2^\infty}(\GQ)$, up to conjugacy in $\GL(2,\Z_2)$. Moreover, the index of $\rho_{E,2^\infty}(\Gal(\overline{\Q}/\Q))$ in $\GL(2,\Z_2)$ divides $64$ or $96$.
\end{thm}

The authors of \cite{gcal-r} made extensive use of the database in \cite{Rouse} to eliminate non-examples of isogeny-torsion graphs of type $T_{4}$, $T_{6}$, and $T_{8}$. They also used the parametrization of the modular curve $\operatorname{X}_{24e}$ in the database compiled in \cite{Rouse} in a step to eliminate two ``elusive'' non-examples of isogeny-torsion graphs of $S$ type (see Section 13 in \cite{gcal-r}). In a reversal of styles, the author of \textit{this} paper used one-parameter families of curves in the database from \cite{Rouse} to prove each of the isogeny-torsion graphs of $T_{4}$, $T_{6}$, and $T_{8}$ type correspond to infinite sets of \textit{j}-invariants.

\begin{thm}[Sutherland, Zywina, \cite{SZ}]
Up to conjugacy, there are $248$ open subgroups $G$ of $\operatorname{GL}(2,\widehat{\Z})$ of prime-power level
containing $\operatorname{-Id}$, a representative of complex conjugation mod-$N$, and having full determinant mod-$N$ for which $\operatorname{X}_{G}$ has infinitely many non-cuspidal $\Q$-rational points. Of these $248$ groups, there are $220$ of genus $0$ and $28$ of genus $1$.
\end{thm}

The work in \cite{SZ} classified the modular curves of prime-power level with infinitely many $\Q$-rational points. Additionally, it gave a formulization of Hilbert's Irreducibility Theorem that proved key in classifying which isogeny-torsion graphs correspond to infinitely many \textit{j}-invariants.

\subsection{Hilbert's Irreducibility Theorem} \label{HIT}

\begin{theorem}[Hilbert]

Let $f_{1}(X_{1}, \ldots, X_{r}, Y_{1}, \ldots, Y_{s}), \ldots, f_{n}(X_{1}, \ldots, X_{r}, Y_{1}, \ldots, Y_{s})$ be irreducible polynomials in the ring $\Q(X_{1}, \ldots, X_{r})[Y_{1}, \ldots, Y_{s}]$. Then there exists an $r$-tuple of rational numbers $(a_{1}, \ldots, a_{r})$ such that $f_{1}(a_{1}, \ldots, a_{r}, Y_{1}, \ldots, Y_{s}), \ldots, f_{n}(a_{1}, \ldots, a_{r}, Y_{1}, \ldots, Y_{s})$ are irreducible in the ring $\Q[Y_{1}, \ldots, Y_{s}]$.

\end{theorem}

In Section $7$ of \cite{SZ} there is the following statement:

Let $G$ be a subgroup of prime-power level such that the modular curve, $\operatorname{X}_{G}$ generated by $G$ is a curve of genus $0$. Define the set
$$\mathcal{S}_{G} := \bigcup_{G'} \pi_{G',G}(X_{G'}(\Q))$$
where $G'$ varies over the proper subgroups of $G$ that are conjugate to one of the groups that define a modular curve of prime-power level with infinitely many non-cuspidal, $\Q$-rational points from \cite{SZ} and $\pi_{G',G} \colon \operatorname{X}_{G'} \to \operatorname{X}_{G}$ is the natural morphism induced by the inclusion $G' \subseteq G$.

Then $\operatorname{X}_{G} \cong \PP^{1}$ and $\mathcal{S}_{G}$ is a \textit{thin} subset (see Chapter 3 of \cite{MR2363329} for the definition and properties of \textit{thin} sets) of $\operatorname{X}_{G}(\Q)$ in the language of Serre. The field $\Q$ is Hilbertian and in particular, $\PP^{1}(\Q) \cong \operatorname{X}_{G}(\Q)$ is not thin; this implies that $\operatorname{X}_{G}(\Q) \setminus S_{G}$ cannot be thin and is infinite.

\begin{proposition} \label{HIT proposition}
Let $N$ be a positive integer. Let $G$ be a subgroup of $\operatorname{GL}(2, \ZZ / N \ZZ)$ such that
\begin{enumerate}
    \item $G$ contains $\operatorname{-Id}$,
    \item $G$ contains a representative of complex conjugation mod-$N$,
    \item $G$ has full determinant modulo $N$.
\end{enumerate}
If the modular curve $\operatorname{X}_{G}$ defined by $G$ is a genus $0$ curve that contains infinitely many non-cuspidal, $\Q$-rational points, then there are infinitely many \textit{j}-invariants corresponding to elliptic curves over $\Q$ such that the image of the mod-$N$ Galois representation is conjugate to $G$ itself (not a proper subgroup of $G$).
\end{proposition}

\begin{remark}\label{Hilbert's Irreducibility Theorem}
The proof of Proposition \ref{HIT proposition} follows directly from the discussion in Section $7$ of \cite{SZ}. In the case that $\operatorname{X}_{G}$ is genus $0$ and contains infinitely many non-cuspidal $\Q$-rational points, the set $\operatorname{X}_{G}(\QQ) \setminus S_{G}$ is infinite. Hence, the set of non-cuspidal, $\Q$-rational points on $\operatorname{X}_{G}$ that are not on the modular curves defined by any proper subgroup of $G$ is infinite.
\end{remark}

\section{Some Lemmas}

\begin{lemma}[\cite{gcal-r}, Lemma 5.5a]\label{lem-necessity-for-point-rationality} Let $E / \Q$ and $E' / \Q$ be elliptic curves. Let $\phi \colon E \to E'$ be an isogeny such that the kernel of $\phi$ is a finite, cyclic, $\Q$-rational group, $H$. Then, for an arbitrary $P \in E$, the point $\phi(P) \in E'$ is defined over $\Q$ if and only if $\sigma(P) - P \in H$ for all $\sigma \in G_{\Q}$.
\end{lemma}

\begin{lemma}[\cite{gcal-r}, Lemma 5.6]\label{lem-2torspt-all-have-2torspt} Let $E/\Q$ be an elliptic curve with a point of order $2$ defined over $\Q$. Then, every elliptic curve over $\Q$ that is $\Q$-isogenous to $E$ also has a point of order $2$ defined over the rationals.
\end{lemma}

\begin{lemma}[\cite{gcal-r}, Lemma 5.11]\label{lem-Maximality-Of-Rational-2-Power-Groups} Let $E / \Q$ and $E' / \Q$ be elliptic curves and let $P$ be a point of $E$ of order $2^{M}$ with $M \geq 1$. Suppose $P$ generates a $\Q$-rational group and the two cyclic groups of order $2^{M+1}$ that contain $P$ are not $\Q$-rational. Let $\phi \colon E \to E'$ be an isogeny with kernel $\langle P \rangle$. Then, $E'(\Q)_{\text{tors}}$ is cyclic.
\end{lemma}

\begin{lemma}[\cite{gcal-r}, Lemma 5.16]\label{lem-subsequent-rational-pts}
		Let $E / \Q$ be an elliptic curve and let $p$ be an odd prime. Suppose $E$ has a point $P_1$ of order $p$ defined over $\Q$, and suppose $C_{p}(E) = m+1 > 1$ such that $\langle P_{m} \rangle$ is a cyclic $\Q$-rational group of order $p^m$ containing $P_{1}$. For $0 \leq j \leq m$, let $P_{j} \in \langle P_{m} \rangle$ such that $P_0=\mathcal{O}$ and $[p]P_{j} = P_{j-1}$, and  let $\phi_{j} \colon E \to E_j$ be an isogeny with kernel $\langle P_{j} \rangle$. 
		\begin{enumerate}
			\item If $1\leq j \leq m-1$, then $E_j$ has a point of order $p$ defined over $\Q$. Further, for $\phi_{{{m-1}}}\colon E \to E_{m-1}$, the curve $E_{m-1}$ has a point of order $p$ defined over $\Q$ but no points of order $p^{2}$ defined over $\Q$. 
			\item If $j=m$, the curve $E_m$ has no points of order $p$ defined over $\Q$.
		\end{enumerate}
	\end{lemma}

\begin{definition}
Let $N$ be a positive integer. The subgroup of $\operatorname{GL}(2, \Z / N \Z)$ consisting of all matrices of the form $\left(\begin{array}{cc}
    1 & b \\
    0 & d
\end{array}\right)$ will be denoted $\mathfrak{B}_{N}$.
\end{definition}

\begin{lemma}\label{Rational Points}

Let $p$ be an odd prime and $u$ a generator of $\left(\Z / p \Z\right)^{\times}$. Then up to conjugation, the only proper subgroup of $\mathfrak{B}_{p}$ which has an element with determinant $u$ is the subgroup of diagonal matrices of $\mathfrak{B}_{p}$.

\end{lemma}

\begin{proof}

Let $H$ be a subgroup of $\mathfrak{B}_{p}$ of full determinant mod-$p$. Then $H$ contains an element of the form $h = \left(\begin{array}{cc}
    1 & x \\
    0 & u
\end{array}\right)$ for some $x \in \Z / p \Z$. The matrix $h$ has order $p-1$ because $h$ is conjugate to $d = \left(\begin{array}{cc}
    1 & 0 \\
    0 & u
\end{array}\right)$ by the matrix $\left(\begin{array}{cc}
    1 & \frac{x}{1-u} \\
    0 & 1
\end{array}\right)$. Notice that $d$ generates the subgroup of $\mathfrak{B}_{p}$ of diagonal matrices. The group $\mathfrak{B}_{p}$ is of order $(p-1)p$ and $\left\langle h \right\rangle$ is a subgroup of $\mathfrak{B}_{p}$ of index $p$. As $p$ is prime, $\left\langle h \right\rangle$ is a maximal subgroup of $\mathcal{B}_{p}$. The group $H$ contains $\left\langle h \right\rangle$, and either $H = \left\langle h \right\rangle$ and hence, is conjugate to $\left\langle d \right\rangle$ or $H$ properly contains $\left\langle h \right\rangle$ and hence, $H = \mathfrak{B}_{p}$.
\end{proof}

\begin{lemma}\label{Unique Index 2 Subgroup}
Let $p$ be an odd prime. Then $\mathfrak{B}_{p}$ and the subgroup of diagonal matrices of $\mathcal{B}_{p}$ both have a single subgroup of index $2$.
\end{lemma}

\begin{proof}

Let $u$ be a generator of $\left(\ZZ / p \ZZ\right)^{\times}$. Let $d = \left(\begin{array}{cc}
    1 & 0 \\
    0 & u
\end{array}\right)$ and let $t = \left(\begin{array}{cc}
    1 & 1 \\
    0 & 1
\end{array}\right)$. Then $\mathcal{B}_{p} = \left\langle d, t\right\rangle$. Let $H$ be a subgroup of $\mathfrak{B}_{p}$ of index $2$. Then $H$ is a normal subgroup of $\mathcal{B}_{p}$ and all squares of $\mathfrak{B}_{p}$ are elements of $H$. Thus, $t^{2}$ is an element of $H$. The matrix $t$ has order $p$ which is an odd prime and so, the group generated by $t$ is equal to the group generated by $t^{2}$. Hence, $t$ is an element of $H$. Thus, $H = \left\langle d^{2}, t \right\rangle$.

Note that the group of diagonal matrices of $\mathcal{B}_{p}$ is generated by $d$. As $\left\langle d \right\rangle$ is a cyclic group, it has a single subgroup of index $2$.
\end{proof}

\begin{corollary}\label{Corollary Unique Index 2 Subgroup}
Let $p$ be an odd prime and let $E$ be an elliptic curve defined over $\Q$ that has a point of order $p$ defined over $\Q$. Then when $p \equiv 1 \bmod 4$, $\QQ(\sqrt{p})$ is the only quadratic subfield of $\QQ(E[p])$ and when $p \equiv 3 \bmod 4$, $\QQ(\sqrt{-p})$ is the only quadratic subfield of $\QQ(E[p])$.
\end{corollary}

\begin{proof}
The $p$-division field $E[p]$ contains the field generated by the $p$-th roots of unity, $\QQ(\zeta_{p})$. Moreover, $\QQ\left(\zeta_{p}\right)$ contains $\sqrt{(-1)^{(p-1)/2}\cdot p}$. By the fundamental theorem of Galois theory, subfields of $\QQ(E[p])$ of degree $2$ correspond to subgroups of the image of the mod-$p$ Galois representation attached to $E$ of index $2$. By Lemma \ref{Rational Points} the image of the mod-$p$ Galois representation attached to $E$ is conjugate to $\mathcal{B}_{p}$ or the group of diagonal matrices of $\mathcal{B}_{p}$. By Lemma \ref{Unique Index 2 Subgroup}, $\mathcal{B}_{p}$ and the group of diagonal matrices of $\mathcal{B}_{p}$ both contain a single subgroup of index $2$.
\end{proof}

\begin{lemma}\label{Three Subgroups of Index 2}
Let $p$ be an odd prime. Then the group $\left\langle \mathcal{B}_{p}, \operatorname{-Id} \right\rangle$ contains three subgroups of index $2$.
\end{lemma}

\begin{proof}

Let $G = \left\langle \mathcal{B}_{p}, \operatorname{-Id} \right\rangle$. As $\operatorname{-Id}$ is not an element of $\mathcal{B}_{p}$, $\mathcal{B}_{p}$ is an index-$2$ subgroup of $G$ and hence, the order of $G$ is equal to $2(p-1)p$. Let $u$ be a generator of $\left(\ZZ / p \ZZ\right)^{\times}$. Let $d = \left(\begin{array}{cc}
    1 & 0 \\
    0 & u
\end{array}\right)$ and $t = \left(\begin{array}{cc}
    1 & 1 \\
    0 & 1
\end{array}\right)$. Then $G = \left\langle d, t, \operatorname{-Id} \right\rangle$. Let $H_{1} = \mathcal{B}_{p} = \left\langle t, d \right\rangle$, let $H_{2} = \left\langle t, -d \right\rangle$, and let $H_{3} = \left\langle t, d^{2}, \operatorname{-Id} \right\rangle$. First we will prove that the orders of $H_{1}$, $H_{2}$, and $H_{3}$ are all equal to $(p-1)p$. It is clear to see that the order of $H_{1} = \mathcal{B}_{p}$ is equal to $(p-1)p$. We will now prove that the order of $H_{2}$ is equal to $(p-1)p$. Note that $\left\langle t \right\rangle$ is a normal subgroup of $H_{2}$ and intersects $\left\langle -d \right\rangle$ trivially. Thus, $H_{2}$ is isomorphic to the product $\left\langle t \right\rangle \cdot \left\langle -d \right\rangle$. The group $\left\langle -d \right\rangle$ contains $\left\langle d^{2} \right\rangle$ as a subgroup of index $2$ and hence, the order of $\left\langle -d \right\rangle$ is equal to $p-1$. Now we will prove that the order of $H_{3}$ is equal to $(p-1)p$. Clearly, $\operatorname{-Id}$ is not contained in $H_{3}' = \left\langle t, d^{2} \right\rangle$. Hence, $H_{3}'$ is a subgroup of $H_{3}$ of index $2$. Using a similar analysis as we did for the order of $H_{2}$, we see that the order of $H_{3}'$ is equal to $(p-1)p/2$. Hence, the order of $H_{3}$ is equal to $(p-1)p$.

Let $H$ be a subgroup of $G$ of index $2$ (order $(p-1)p$). Then $H$ is a normal subgroup of $G$. As $H$ is a subgroup of $G$ of index $2$, it contains all squares of $G$ and thus, $d^{2}$ is an element of $H$. The order of $t$ is equal to $p$, an odd prime. Hence, $\left\langle t \right\rangle = \left\langle t^{2} \right\rangle$. Thus, $t$ is an element of $H$. If $H$ contains $d$, then $H = H_{1}$. Suppose that $H$ does not contain $d$. If $H$ contains $\operatorname{-Id}$, then $H = H_{3}$. If $H$ does not contain $\operatorname{-Id}$ and $H$ does not contain $d$, then it contains their product $-d$ and hence, $H = H_{2}$.
\end{proof}

\begin{corollary}\label{Corollary Three Subgroups of Index 2}
Let $p$ be an odd prime and let $E$ be an elliptic curve defined over $\Q$ such that the image of the mod-$p$ Galois representation attached to $E$ is conjugate to $\left\langle \mathcal{B}_{p}, \operatorname{-Id} \right\rangle$. Then there are three quadratic subfields of $\QQ(E[p])$. Moreover, in the case that $p \equiv 3 \mod 4$, two of these quadratic subfields are totally imaginary and one of them is totally real.
\end{corollary}

\begin{proof}

By the fundamental theorem of Galois theory, subfields of $\QQ(E[p])$ of degree $2$ correspond to subgroups of the image of the mod-$p$ Galois representation attached to $E$ of index $2$. The three subgroups of the image of the mod-$p$ Galois representation attached to $E$ of index $2$ are the groups $H_{1}, H_{2},$ or $H_{3}$ in the proof of Lemma \ref{Three Subgroups of Index 2}. Hence, there are three quadratic subfields of $\QQ(E[p])$. Let $p \equiv 3 \bmod 4$. Note that $\QQ\left(\sqrt{-p}\right)$ is a totally imaginary quadratic subfield of $\QQ(E[p])$. This is enough to see that two of the three quadratic subfields of $\QQ(E[p])$ are totally imaginary.

\end{proof}

\begin{lemma}\label{Split Cartan}
Let $p = 3$ or $5$ and let $E / \QQ$ be an elliptic curve with a point of order $p$ defined over $\Q$. If $E$ contains two distinct $\Q$-rational subgroups of order $p$, then the image of the mod-$p$ Galois representation attached to $E$ is conjugate to
$$D = \left\langle \left(\begin{array}{cc}
    1 & 0 \\
    0 & z
\end{array}\right) : z \in \left(\ZZ / p \ZZ\right)^{\times} \right\rangle.$$
\end{lemma}

\begin{proof}
Let $p$ and $E$ be as in the hypothesis. Let $P$ be a point of $E$ of order $p$ defined over $\Q$ and let $Q$ be a point of $E$ of order $p$, not defined over $\Q$, that generates a $\Q$-rational group. Then $E[p] = \left\langle P, Q \right\rangle$ and with this basis, the image of the mod-$p$ Galois representation attached to $E$ is a subgroup of $D$.

Noting that the image of the mod-$p$ Galois representation attached to $E$ must have full determinant mod-$p$, we can conclude that the image of the mod-$p$ Galois representation attached to $E$ is equal to $D$.
\end{proof}

\begin{lemma} \label{Quadratic Twisting Odd Graphs}
Let $p = 3, 5,$ or $7$. Let $E / \Q$ be an elliptic curve that contains an element $P$ that generates a $\Q$-rational subgroup of order $p$. Suppose the image of the mod-$p$ Galois representation attached to $E$ contains $\operatorname{-Id}$. Then neither $E$ nor $E / \langle P \rangle$ have a point of order $p$ defined over $\Q$.
\end{lemma}

\begin{proof}
Let $Q$ be a point on $E$ of order $p$. The matrix $\operatorname{-Id}$ is in the center of $\operatorname{GL}(2, \Z / p \Z)$ so a subgroup $H$ of $\operatorname{GL}(2, \Z / p \Z)$ contains an element conjugate to $\operatorname{-Id}$ if and only if $H$ contains $\operatorname{-Id}$. Thus, the image of the mod-$p$ Galois representation attached to $E$ contains $\operatorname{-Id}$ and there exists a Galois automorphism $\tau \in G_{\Q}$ such that $\tau(Q) = -Q$. As $Q$ has odd order, $Q \neq -Q$. Hence, $Q$ is not defined over $\Q$.

Now let $\phi \colon E \to E / \left\langle P \right\rangle$ be an isogeny with kernel $\left\langle P \right\rangle$. Let $R$ be a point on $E$ of order $p$. Let $\tau \in G_{\Q}$ such that $\tau(R) = -R$. If $\phi(R)$ is defined over $\Q$, then $\tau(R)-R = [-2]R \in \left\langle P \right\rangle$ by Lemma \ref{lem-necessity-for-point-rationality}. If this is the case, then $[2]\phi(R) = \phi([2]R) = \mathcal{O}$ and so, $\phi(R)$ has order equal to $1$ or $2$, contradicting the fact that $\phi(R)$ has order $p$.

\end{proof}

\begin{corollary}\label{Corollary Quadratic Twisting Odd Graphs}

Let $p = 3, 5,$ or $7$. Let $E / \Q$ be an elliptic curve such that the image of the mod-$p$ Galois representation attached to $E$ contains $\operatorname{-Id}$. Then none of the elliptic curves over $\Q$ in the $\Q$-isogeny class of $E$ have a point of order $p$ defined over $\Q$.

\end{corollary}

\begin{proof}
By Lemma \ref{Quadratic Twisting Odd Graphs}, at least two elliptic curves over $\QQ$ in the $\QQ$-isogeny class of $E$ do not have points of order $p$ defined over $\QQ$. By the classification of isogeny-torsion graphs, no elliptic curve over $\QQ$ in the $\QQ$-isogeny class of $E$ has a point of order $p$ defined over $\QQ$ (see the tables in Section 2 in \cite{gcal-r}).

\end{proof}

\begin{lemma} \label{No Points of Order 4}
Let $E / \Q$ be an elliptic curve with full two-torsion defined over $\Q$. Let $M$ be the positive integer such that $2^{M}$ is the order of the largest finite, cyclic, $\QQ$-rational subgroup of $E$ of $2$-power order. Suppose the image of the mod-$2^{M+1}$ Galois representation attached to $E$ contains $\operatorname{-Id}$. Then no elliptic curve over $\QQ$ that is $\QQ$-isogenous to $E$ has a point of order $4$ defined over $\Q$.

\end{lemma}

\begin{proof}

The elliptic curve $E$ has full two-torsion defined over $\Q$ and so $E$ has a $2$-isogeny and $C_{2}(E) \geq 4$. As $2^{M} \geq 2$, we have $2^{M+1} \geq 4$, and the image of the mod-$4$ Galois representation attached to $E$ contains $\operatorname{-Id}$. Let $R$ be a point on $E$ of order $4$. There is a Galois automorphism $\tau \in G_{\QQ}$ such that $\tau(R) = -R$. As $R$ is a point of order $4$, $\tau(R) \neq R$ and so, $R$ is not defined over $\QQ$. We break up the rest of the proof into three cases, depending on $C_{2}(E)$.

\begin{itemize}
    \item $C_{2}(E) = 4$.
    
    Suppose that $C(E) = 4$. Then $E(\Q)_{\text{tors}} \cong \Z / 2 \Z \times \Z / 2 \Z$ and the cyclic, $\QQ$-rational subgroups of $E$ are the groups $\left\langle P_{2} \right\rangle$, $\left\langle Q_{2} \right\rangle$, $\left\langle P_{2} + Q_{2} \right\rangle$, and the trivial group. Let $A$ be any element of $E$ of order $2$ and $\phi \colon E \to E / \left\langle A \right\rangle$ be an isogeny with kernel generated by $A$. We claim that $E / \left\langle A \right\rangle (\Q)_{\text{tors}} \cong \Z / 2 \Z$.
    
    Indeed, $E / \left\langle A \right\rangle(\Q)_{\text{tors}}$ is cyclic by Lemma \ref{lem-Maximality-Of-Rational-2-Power-Groups}. Let $A' \in E$ such that $[2]A' = A$. Let $B$ be a point of $E$ of order $2$ not equal to $A$ and let $B' \in E$ such that $[2]B' = B$. By our hypothesis, $E$ has full two-torsion defined over $\Q$ and so, $B$ is defined over $\Q$, and $\phi(B)$ is the point of $E / \left\langle A \right\rangle$ of order $2$ defined over $\Q$ by Lemma \ref{lem-necessity-for-point-rationality}. The two cyclic groups of order $4$ containing $\phi(B)$ are $\left\langle \phi(B') \right\rangle$ and $\left\langle \phi(A' + B') \right\rangle$. Let $C$ be $B'$ or $A'+B'$. Note that $C$ is an element of $E$ of order $4$. By the fact that the image of the mod-$4$ Galois representation attached to $E$ contains $\operatorname{-Id}$, there is a Galois automorphism $\tau \in G_{\Q}$ such that $\tau(C) = -C$. Hence, $\tau(C) - C = [-2]C = B$ or $A+B$. Both $B$ and $A+B$ are elements of $E$ that are not contained in $\left\langle A \right\rangle$. Thus, $\phi(C)$ is not defined over $\Q$ by Lemma \ref{lem-necessity-for-point-rationality}. Hence, all groups $E / \left\langle P_{2} \right\rangle(\Q)_{\text{tors}}, E / \left\langle Q_{2} \right\rangle(\Q)_{\text{tors}},$ and $E / \left\langle P_{2} + Q_{2} \right\rangle(\Q)_{\text{tors}}$ are isomorphic to $\Z / 2 \Z$.
    
    If we suppose that $C(E) = 8$, then $C_{3}(E) = 2$ and the isogeny graph associated to the $\QQ$-isogeny class of $E$ is of type $S$. In this case, the finite, cyclic, $\QQ$-rational subgroups of $E$ are the ones generated by $\mathcal{O}$, $P_{2}$, $Q_{2}$, $P_{2}+Q_{2}$, $D_{3}$, $D_{3}+P_{2}$, $D_{3}+Q_{2}$, and $D_{3}+P_{2}+Q_{2}$ where $D_{3}$ is a point on $E$ of order $3$ that generates a $\QQ$-rational group. As none of $E/\left\langle P_{2} \right\rangle$, $E / \left\langle Q_{2} \right\rangle$, $E/\left\langle P_{2}+Q_{2} \right\rangle$ have points of order $4$ defined over $\QQ$, neither do $E/\left\langle D_{3}+P_{2} \right\rangle$, $E/\left\langle D_{3}+Q_{2} \right\rangle$, or $E/\left\langle D_{3}+P_{2}+Q_{2} \right\rangle$ as passing through a $3$-isogeny defined over $\QQ$ preserves points of order $4$ defined over $\QQ$.
    
    \item $C_{2}(E) = 6$.
    
    We may assume without loss of generality that the finite, cyclic, $\QQ$-rational subgroups of $E$ are the groups $\left\langle P_{2} \right\rangle$, $\left\langle Q_{2} \right\rangle$, $\left\langle P_{2} + Q_{2} \right\rangle$, $\left\langle Q_{4} \right\rangle$, and $\left\langle P_{2} + Q_{4} \right\rangle$, and the trivial group by Remark \ref{2-powers}. Thus, the largest finite, cyclic, $\QQ$-rational subgroup of $E$ is of order $4$ and so, by the hypothesis, the image of the mod-$8$ Galois representation attached to $E$ contains $\operatorname{-Id}$. Using a similar proof in the case of $C_{2}(E) = 4$, we can conclude that the groups $E / \left\langle P_{2} \right\rangle(\Q)_{\text{tors}}$ and $E / \left\langle P_{2} + Q_{2} \right\rangle(\Q)_{\text{tors}}$ are of order $2$.
    
    Let $A = Q_{4}$ or $P_{2} + Q_{4}$. We claim that $E / \left\langle A \right\rangle(\Q)_{\text{tors}} \cong \Z / 2 \Z$. Let $\phi \colon E \to E / \left\langle A \right\rangle$ be an isogeny with kernel $\left\langle A \right\rangle$. By Lemma \ref{lem-Maximality-Of-Rational-2-Power-Groups}, the group $E / \left\langle A \right\rangle (\Q)_{\text{tors}}$ is cyclic. Let $B$ be a point on $E$ of order $2$ not equal to $Q_{2}$. The point of $E / \left\langle A \right\rangle$ of order $2$ defined over $\Q$ is $\phi(B)$ by Lemma \ref{lem-necessity-for-point-rationality}. Let $A'$ be a point on $E$ such that $[2]A' = A$ and let $B'$ be a point on $E$ such that $[2]B' = B$. The two cyclic groups of order $4$ that contain $\phi(B)$ are $\left\langle \phi(B') \right\rangle$ and $\left\langle \phi(A'+B') \right\rangle$. Note that $B'$ and $A'+B'$ are points on $E$ of order $4$ and $8$ respectively. As the image of the mod-$8$ Galois representation attached to $E$ contains $\operatorname{-Id}$, there is a Galois automorphism $\tau \in G_{\Q}$ such that $\tau(B') = -B'$ and $\tau(A') = -A'$. Hence, $\tau(B') - B' = [-2]B' = B \notin \left\langle A \right\rangle$ and $\tau(A'+B') - (A'+B') = [-2](A'+B') = -A+B \notin \left\langle A \right\rangle$. Thus, both $\phi(B')$ and $\phi(A'+B')$ are not defined over $\Q$ by Lemma \ref{lem-necessity-for-point-rationality} and hence, $E / \left\langle Q_{4} \right\rangle(\Q)_{\text{tors}}$ and $E / \left\langle P_{2} + Q_{4} \right\rangle(\Q)_{\text{tors}}$ are groups of order $2$. By the classification of isogeny-torsion graphs of $T_{6}$ type, (see Table \ref{T_{6} Graphs}) we have proven that no elliptic curve over $\Q$ in the $\Q$-isogeny class of $E$ has a point of order $4$ defined over $\Q$.
    
    \item $C_{2}(E) = 8$.
    
    We may assume without loss of generality that the finite, cyclic, $\QQ$-rational subgroups of $E$ are the groups $\left\langle P_{2} \right\rangle$, $\left\langle Q_{2} \right\rangle$, $\left\langle P_{2} + Q_{2} \right\rangle$, $\left\langle Q_{4} \right\rangle$, $\left\langle P_{2} + Q_{4} \right\rangle$, $\left\langle Q_{8} \right\rangle$, $\left\langle P_{2} + Q_{8} \right\rangle$, and the trivial group. By our hypothesis, the image of the mod-$16$ Galois representation attached to $E$ contains $\operatorname{-Id}$. Using a similar proof in the case of $C_{2}(E) = 4$, we can conclude that the groups $E / \left\langle P_{2} \right\rangle(\Q)_{\text{tors}}$ and $E / \left\langle P_{2} + Q_{2} \right\rangle(\Q)_{\text{tors}}$ are of order $2$. Using a similar argument as in the case of $C_{2}(E) = 6$ with $A$ replaced with $Q_{8}$ or $P_{2} + Q_{8}$ we can prove that $E / \left\langle Q_{8} \right\rangle(\Q)_{\text{tors}}$ and $E / \left\langle P_{2} + Q_{8} \right\rangle(\Q)_{\text{tors}}$ are groups of order $2$. Making note of the classification of isogeny-torsion graphs of $T_{8}$ type (see Table \ref{T_{8} Graphs}), we can conclude that no elliptic curve over $\Q$ in the $\Q$-isogeny class has a point of order $4$ defined over $\Q$.
\end{itemize}
\end{proof}

\begin{remark}

Lemma \ref{Quadratic Twisting Odd Graphs} and Lemma \ref{No Points of Order 4} show that

\begin{itemize}

\item For any elliptic curve $E / \Q$ with full two-torsion defined over $\Q$, there is a quadratic twist that eliminates all the points of order $4$ defined over $\Q$ from all the elliptic curves over $\Q$ in the $\Q$-isogeny class of $E$.

\item For an elliptic curve $E / \Q$ that has a cyclic, $\Q$-rational subgroup of even order, there is a quadratic twist that eliminates all the points defined over $\Q$ of order $3$ and $5$ but not points of order $4$ of the elliptic curves over $\Q$ in the $\Q$-isogeny class of $E$.

\item For an elliptic curve $E / \Q$ that has a non-trivial, cyclic, $\Q$-rational subgroup of odd order, there is a quadratic twist that eliminates all the points of order $\geq 3$ defined over $\Q$ from the elliptic curves over $\Q$ in the $\Q$-isogeny class of $E$.

\end{itemize}
\end{remark}

\section{Isogeny-torsion graphs corresponding to finite sets of \textit{j}-invariants}

In this section, we determine the isogeny-torsion graphs that correspond to sets of finitely many \textit{j}-invariants (without having to mention torsion configuration). This classification follows directly from Theorem \ref{thm-ratnoncusps}. Each elliptic curve over $\Q$ represented by a vertex in the isogeny-torsion graphs in Table \ref{Finite graphs} corresponds to a non-cuspidal, $\Q$-rational point on a modular curve $\operatorname{X}_{0}(N)$ of genus $\geq 1$ for some positive integer $N$ which has finitely many non-cuspidal, $\Q$-rational points (see Theorem \ref{thm-ratnoncusps}). The \textit{j}-invariants in Table \ref{Finite graphs} came from the work in \cite{lozano0}.

\begin{proposition}\label{Finite Graphs Proposition}

Let $\mathcal{G}$ be an isogeny-torsion graph of one of the following types (regardless of torsion configuration)

\begin{enumerate}
    \item $L_{2}(p)$ where $p = 11, 17, 19, 37, 43, 67,$ or $163$,
    \item $L_{4}$,
    \item $R_{4}(pq)$ where $pq = 14, 15,$ or $21$.
\end{enumerate}
Then $\mathcal{G}$ corresponds to a finite set of \textit{j}-invariants.
\end{proposition}

\begin{center}
\begin{table}[h!]
\renewcommand{\arraystretch}{1.3}
\scalebox{0.7}{
    \begin{tabular}{ |c|c|c|c| }
    \hline
         Graph & Isogeny Graph & Torsion & \textit{j} \\
         \hline
        \multirow{7}*{\includegraphics[width=40mm]{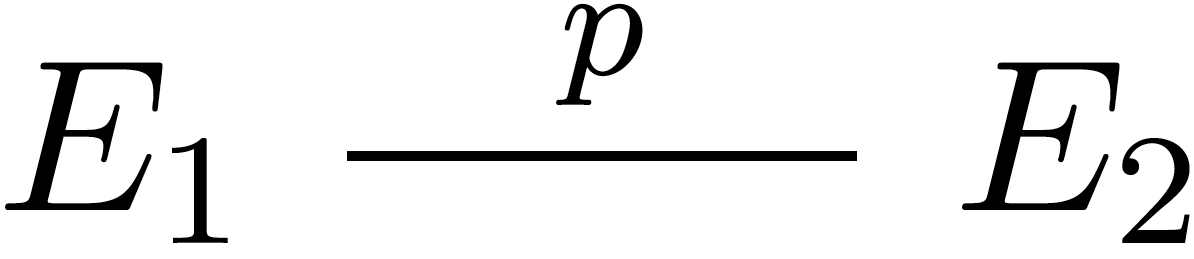}} & \multirow{3}*{$L_{2}(11)$} & \multirow{3}*{$([1],[1])$} & $-11 \cdot 131^{3}$\\
        \cline{4-4}
        & & & $-2^{15}$\\
        \cline{4-4}
        & & & $-11^{2}$ \\
        \cline{2-4}
        & \multirow{2}*{$L_{2}(17)$} & \multirow{2}*{$([1],[1])$} & $-17^{2} \cdot 101^{3}/2$ \\
        \cline{4-4}
        & & & $-17 \cdot 373^{3} / 2^{17}$ \\
        \cline{2-4}
        & $L_{2}(19)$ & $([1],[1])$ & $-2^{15} \cdot 3^{3}$ \\
        \cline{2-4}
        & \multirow{2}*{$L_{2}(37)$} & \multirow{2}*{$([1],[1])$} & $-7 \cdot 11^{3}$ \\
        \cline{4-4}
        & & & $-7 \cdot 137^{3} \cdot 2083^{3}$ \\
        \cline{2-4}
        & $L_{2}(43)$ & $([1],[1])$ & $-2^{18} \cdot 3^{3} \cdot 5^{3}$ \\
        \cline{2-4}
        & $L_{2}(67)$ & $([1],[1])$ & $-2^{15} \cdot 3^{3} \cdot 5^{3} \cdot 11^{3}$ \\
        \cline{2-4}
        & $L_{2}(163)$ & $([1],[1])$ & $-2^{18} \cdot 3^{3} \cdot 5^{3} \cdot 23^{3} \cdot 29^{3}$ \\
        \hline
        
        \multirow{2}*{\includegraphics[height=8mm]{L4.png}} & \multirow{2}*{$L_{4}$} & $([3],[3],[3],[1])$ & $-2^{15} \cdot 3 \cdot 5^{3}$ \\
        \cline{4-4}
        & & $([1],[1],[1],[1])$ & $0$ \\
        \hline
        
        \multirow{10}*{\includegraphics[width=40mm]{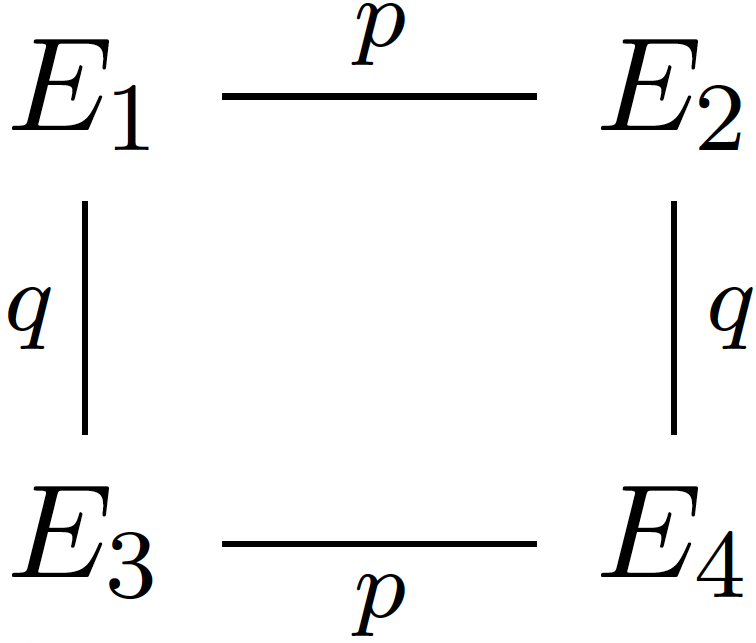}} & \multirow{2}*{$R_{4}(14)$} & \multirow{2}*{$([2],[2],[2],[2])$} & $-3^{3} \cdot 5^{3}$ \\
        \cline{4-4}
        & & & $3^{3} \cdot 5^{3} \cdot 17^{3}$\\
        \cline{2-4}
        
        & \multirow{4}*{$R_{4}(15)$} & $([5],[5],[1],[1])$ & $-5^{2} / 2$ \\
        \cline{4-4}
        & & $([3],[3],[1],[1])$ & $-5^{2} \cdot 241^{3} / 2^{3}$ \\
        \cline{4-4}
        & & $([1],[1],[1],[1])$ & $-5 \cdot 29^{3} / 2^{5}$ \\
        \cline{4-4}
        & & & $5 \cdot 211^{3} / 2^{15}$ \\
        \cline{2-4}
        & \multirow{4}*{$R_{4}(21)$} & & $-3^{2} \cdot 5^{6} / 2^{3}$ \\
        \cline{4-4}
        & & $([3],[3],[1],[1])$ & $3^{3} \cdot 5^{3} /2$ \\
        \cline{4-4}
        & & $([1],[1],[1],[1])$ & $-3^{2} \cdot 5^{3} \cdot 101^{3} / 2^{21}$ \\
        \cline{4-4}
        & & & $-3^{3} \cdot 5^{3} \cdot 383^{3} / 2^{7}$ \\
        \hline
\end{tabular}}
	\caption{Isogeny-torsion graphs corresponding to finitely many \textit{j}-invariants}
 	\label{Finite graphs}
\end{table}
\end{center}

\subsection{Isogeny-Torsion Graphs of $L_{2}(p)$ Type corresponding to finite sets of \textit{j}-invariants}
Let $E / \Q$ be an elliptic curve and let $p$ equal $11, 17, 19, 37, 43, 67,$ or $163$. If $E$ has a $\Q$-rational subgroup of order $p$, then $C(E) = C_{p}(E) = 2$ by Theorem \ref{thm-kenku}. In other words, the $\Q$-isogeny class of $E$ contains two elliptic curves over $\Q$ and both of those elliptic curves have a $\Q$-rational subgroup of order $p$ and no other non-trivial, cyclic, $\Q$-rational subgroups. As a consequence, both elliptic curves have trivial rational torsion. By Theorem \ref{thm-ratnoncusps}, there are finitely many \textit{j}-invariants corresponding to such elliptic curves.

\subsection{Isogeny-Torsion Graphs of $L_{4}$ Type}

Let $E / \Q$ be an elliptic curve such that either $E$ has a cyclic, $\Q$-rational subgroup of order $27$ or $E$ has a cyclic, $\Q$-rational subgroup of order $9$ \textit{and} an independent $\Q$-rational subgroup of order $3$. In both cases, the isogeny-torsion graph attached to the $\Q$-isogeny class of $E$ is of $L_{4}$ type. In the first case, the \textit{j}-invariant of $E$ is $-12288000$ and in the second case, the \textit{j}-invariant is $0$ (both of which are CM \textit{j}-invariants).

{\bf NB} Not \textit{every} elliptic curve of \textit{j}-invariant $0$ fits into an isogeny graph of $L_{4}$ type; only the elliptic curves with Weierstrass model $y^{2} = x^{3} + 16t^{3}$ for some non-zero, square-free, rational $t$.

\subsection{Isogeny-Torsion graphs Of $R_{4}(pq)$ type corresponding to finite sets of \textit{j}-invariants}

Let $E / \Q$ be an elliptic curve and let $pq = 14$, $15$, or $21$. Suppose $E$ has a cyclic, $\Q$-rational subgroup of order $pq$. By Theorem \ref{thm-kenku}, there are four elliptic curves over $\Q$ in the $\Q$-isogeny class of $E$ and moreover, \textit{every} elliptic curve over $\Q$ in the $\Q$-isogeny class has a cyclic, $\Q$-rational subgroup of order $pq$. There are finitely many \textit{j}-invariants corresponding to elliptic curves over $\Q$ with a cyclic, $\Q$-rational subgroup of order $14$, $15$, or $21$ by Theorem \ref{thm-ratnoncusps}. This concludes the proof of Proposition \ref{Finite Graphs Proposition}.

\section{Isogeny-torsion graphs corresponding to infinite sets of \textit{j}-invariants}
\subsection{Methodology Of The Remaining Proofs}
Table \ref{Finite graphs} lists the \textit{j}-invariants corresponding to the $15$ isogeny-torsion graphs corresponding to finite sets of \textit{j}-invariants. It remains to prove that each of the other $37$ isogeny-torsion graphs correspond to infinite sets of \textit{j}-invariants.

The following proofs are of a mostly group theoretic flavor. To prove that an isogeny-torsion graph $\mathcal{G}$ corresponds to an infinite set of \textit{j}-invariants, we will focus on proving a nicely chosen vertex on $\mathcal{G}$ corresponds to an infinite set of \textit{j}-invariants. To do this, we establish the necessary algebraic properties for an elliptic curve in a $\Q$-isogeny class to be represented by the ideal vertex of $\mathcal{G}$. By algebraic properties, we also mean the data afforded by the images of the mod-$N$ Galois representation for some prime powers $N$.

Once we have the necessary mod-$N$ data corresponding to the ideal vertex on $\mathcal{G}$, we make note if the modular curve that parametrizes such an elliptic curve has infinitely many non-cuspidal, $\Q$-rational points. Here is where the results of Rouse--Zureick-Brown \cite{Rouse} and Sutherland--Zywina \cite{SZ} come in to determine some of the isogeny-torsion graphs that correspond to infinite sets of \textit{j}-invariants. At times, the results of \cite{Rouse} and \cite{SZ} are not enough to conclude that $\mathcal{G}$ corresponds to an infinite set of \textit{j}-invariants. But other methods like the ones described in Section \ref{HIT} or Theorem \ref{thm-ratnoncusps} can be used if the results of \cite{Rouse} and \cite{SZ} are insufficient. Once we have that a vertex of an isogeny-torsion graph with some torsion configuration corresponds to an infinite set of \textit{j}-invariants, we take quadratic twists. Taking quadratic twists of all elliptic curves over $\Q$ in a $\Q$-isogeny class does not change the isogeny graph but it very likely changes the isogeny-torsion graph. Usually, a well chosen quadratic twist will toggle between the possible torsion configurations of an isogeny graph. From this fact, it should be noted that one \textit{j}-invariant may correspond to more than one isogeny-torsion graph. As the isogeny-torsion graph with one type of torsion configuration corresponds to an infinite set of \textit{j}-invariants, the isogeny-torsion graph with torsion configuration that comes from quadratic twisting the elliptic curve represented by the ideal vertex also corresponds to an infinite set of \textit{j}-invariants. Finally, at times, it will be necessary to exclude $0$, $1728$, and the \textit{j}-invariants corresponding to the isogeny graphs from Table \ref{Finite graphs}. As the complement of an infinite set with a finite subset is still an infinite set, no issues arise from this exclusion.

The philosophy of this article has been a group theoretic approach to solving questions regarding elliptic curves. Theorems \ref{thm-ratnoncusps} and \ref{thm-kenku}, and work in \cite{Rouse} and \cite{SZ} surely make use of the arithmetic and geometric properties of elliptic curves defined over $\Q$. But aside from that, our proofs appeal only to Galois theory.

\section{Isogeny-Torsion Graphs of $S$ Type}\label{S Graphs Section}

In this section we prove that each of the isogeny-torsion graphs of $S$ type correspond to infinite sets of \textit{j}-invariants.

\begin{proposition} \label{S Graphs Proposition}
Let $\mathcal{G}$ be one of the four isogeny-torsion graphs of $S$ type (regardless of torsion configuration). Then $\mathcal{G}$ corresponds to an infinite set of \textit{j}-invariants.
\end{proposition}

Once we prove that the isogeny-torsion graph corresponding to an elliptic curve over $\Q$ with rational $12$-torsion corresponds to an infinite set of \textit{j}-invariants, we can use quadratic twists to prove the other three isogeny-torsion graphs of $S$ type correspond to infinite set of \textit{j}-invariants. Thus, let us first talk about quadratic twists of elliptic curves over $\Q$ with rational $12$-torsion.

\begin{table}[h!]
 	\renewcommand{\arraystretch}{1.6}
	\begin{tabular} { |c|c|c|c| }
		\hline
		
		Isogeny Graph & Type & Isomorphism Types & Label\\
		\hline
		\multirow{4}*{\includegraphics[scale=0.25]{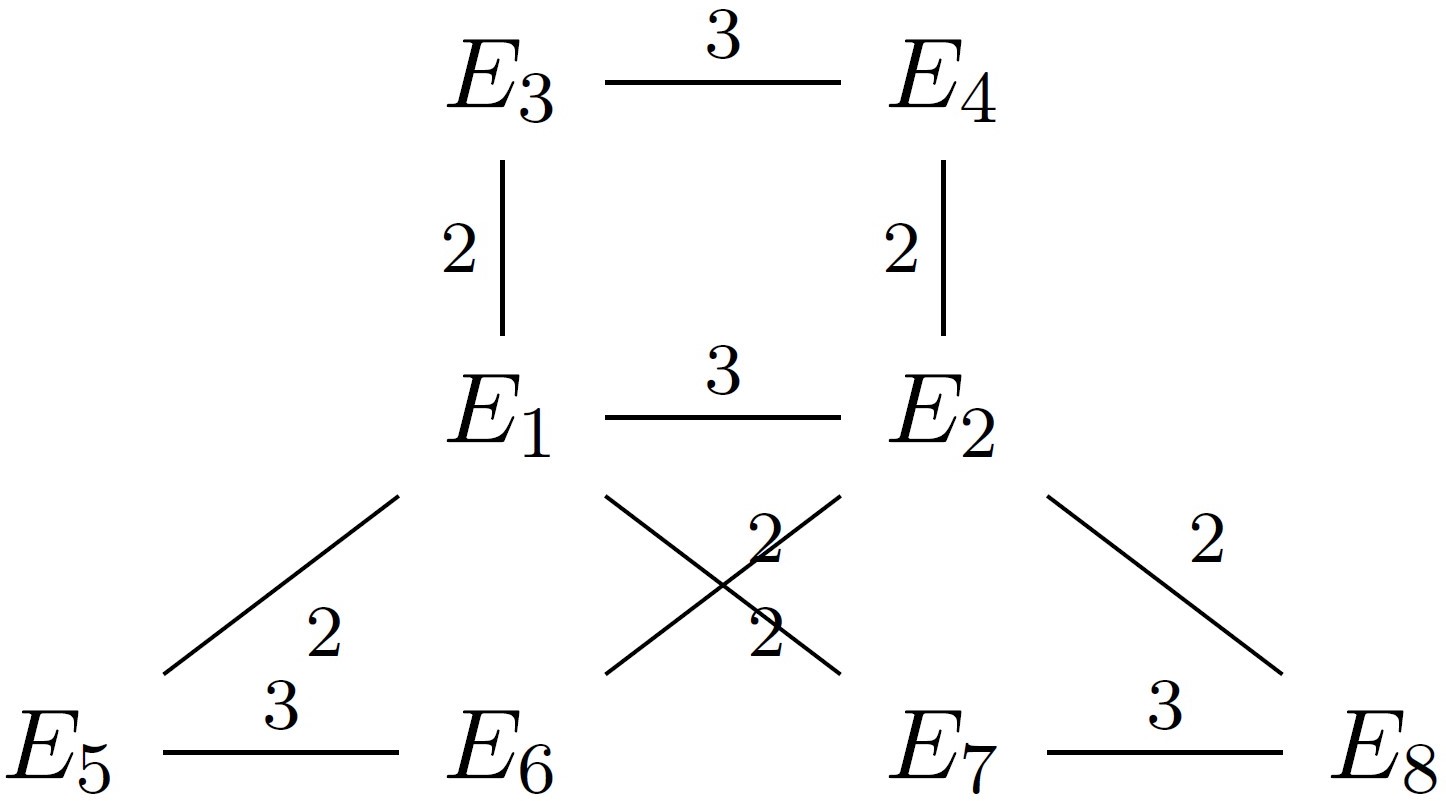}}& \multirow{4}*{$S$} &
		([2,6],[2,2],[12],[4],[6],[2],[6],[2]) & $\mathcal{S}^{1}$ \\
		\cline{3-4}
		& & ([2,6],[2,2],[6],[2],[6],[2],[6],[2]) & $\mathcal{S}^{2}$ \\
		\cline{3-4}
		& & ([2,2],[2,2],[4],[2],[4],[2],[2],[2]) & $\mathcal{S}^{3}$ \\
		\cline{3-4}
		& & ([2,2],[2,2],[2],[2],[2],[2],[2],[2]) & $\mathcal{S}^{4}$ \\
		\hline
	\end{tabular}
	\caption{Isogeny-Torsion Graphs of S type}
    \label{S Graphs}
	\end{table}

\subsection{Quadratic Twists of Elliptic Curves with Rational 12-Torsion}

To start, we note that there are infinitely many \textit{j}-invariants corresponding to elliptic curves over $\QQ$ with a point of order $12$ defined over $\QQ$. For example, let $E_{a,b}(t) : y^{2} + (1-a)xy - by = x^{3}-bx^{2}$ such that
\begin{center}
    $a = \frac{t(1-2t)(3t^{2}-3t+1)}{(t-1)^{3}}$ and $b = -a \frac{2t^{2}-2t+1}{t-1}$.
\end{center}
Let $t \in \QQ$ such that $E_{a,b}(t)$ is an elliptic curve. Then $E_{a,b}(t)(\Q)_{\text{tors}} \cong \Z / 12 \Z$ (see Appendix E of \cite{alrbook}). There is only a single isogeny-torsion graph containing a point of order $12$, namely, $\mathcal{S}^{1}$ (see Table \ref{S Graphs}). Hence, $\mathcal{S}^{1}$ corresponds to an infinite set of \textit{j}-invariants. Let $\operatorname{X}$ be the modular curve that parametrizes elliptic curves defined over $\QQ$ with a point of order $12$ defined over $\QQ$. In other words, $\operatorname{X}$ is generated by $H = \begin{bmatrix} 1 & \ast \\ 0 & \ast \end{bmatrix} \subseteq \operatorname{GL}(2, \ZZ / 12 \ZZ)$. The genus of $\operatorname{X}$ is equal to $0$ and $\operatorname{X}$ has infinitely many points defined over $\QQ$. Reducing $H$ modulo $4$, we get the group $\mathcal{B}_{4}$ which is the group of all elements of $\operatorname{GL}(2, \ZZ / 4 \ZZ)$ of the form $\begin{bmatrix} 1 & b \\ 0 & d \end{bmatrix}$ and reducing $H$ modulo $3$, we get the group $\mathcal{B}_{3}$ which is the group of all elements of $\operatorname{GL}(2, \ZZ / 3 \ZZ)$ of the form $\begin{bmatrix} 1 & b \\ 0 & d \end{bmatrix}$. Let $K_{8}$ be the full lift of $\mathcal{B}_{4}$ to level $8$. By Hilbert's irreducibility theorem, there are infinitely many \textit{j}-invariants corresponding to elliptic curves $E/\QQ$ such that the image of the mod-$24$ Galois representation attached to $E$ is conjugate to $H = K_{8} \times \mathcal{B}_{3}$. In such a case, $\QQ(E[8]) \bigcap \QQ(E[3]) = \QQ$. For more on non-trivial entanglement of division fields, see work in \cite{Entanglement2} and \cite{Entanglement1}. In the case that $\overline{\rho}_{E,12}(G_{\QQ})$ is conjugate to $H$, the quadratic subfield of $\QQ(E[3])$ is $\QQ(\sqrt{-3})$ by Corollary \ref{Corollary Unique Index 2 Subgroup}.

Now we will prove that the isogeny-torsion graphs $\mathcal{S}^{2}$, $\mathcal{S}^{3}$, and $\mathcal{S}^{4}$ each correspond to infinite sets of \textit{j}-invariants.

\begin{itemize}
    \item $\mathcal{S}^{2}$
    
    Let $E/\QQ$ be an elliptic curve such that $E$ has a point of order $12$ defined over $\QQ$ and the image of the mod-$24$ Galois representation attached to $E$ is conjugate to $H$. Let $E^{(-3)}$ be the quadratic twist of $E$ by $-3$. Twisting by $-3$ shifts the points of order $3$ to the other side of the graph. Hence, four of the elliptic curves over $\QQ$ in the isogeny-torsion graph have points of order $3$ defined over $\QQ$. As $\sqrt{-3}$ is not in $\QQ(E[8])$, we have that the image of the mod-$8$ Galois representation attached to $E^{(-3)}$ contains $\operatorname{-Id}$. Let $E'/\QQ$ be the ellipic curve that is $2$-isogenous to $E$. Then $E'$ has full two-torsion defined over $\QQ$ and the image of the mod-$4$ Galois representation attached to $E'$ contains $\operatorname{-Id}$. By Lemma \ref{No Points of Order 4}, no elliptic curve over $\QQ$ in the $\QQ$-isogeny class of $E'$ has a point of order $4$ defined over $\QQ$. Thus, the isogeny-torsion graph associated to the $\QQ$-isogeny class of $E^{(-3)}$ is $\mathcal{S}^{2}$.
    
    \item $\mathcal{S}^{3}$
    
    Let $E/\QQ$ be an elliptic curve such that $E$ has a point of order $12$ defined over $\QQ$ and the image of the mod-$24$ Galois representation attached to $E$ is conjugate to $H$. Let $E^{(-1)}$ be the quadratic twist of $E$ by $-1$. Twisting by $-1$ does not change the position of the points of order $4$ on the isogeny-torsion graph. As $\sqrt{-1}$ is not in $\QQ(E[-3])$, we have that the image of the mod-$3$ Galois representation attached to $E^{(-1)}$ contains $\operatorname{-Id}$. By Corollary \ref{Corollary Quadratic Twisting Odd Graphs}, no elliptic curve in the $\QQ$-isogeny class of $E^{(-1)}$ contains a point of order $3$ defined over $\QQ$. Thus, the isogeny-torsion graph associated to the $\QQ$-isogeny class of $E^{(-1)}$ is $\mathcal{S}^{3}$
    
    \item $\mathcal{S}^{4}$
    
    Let $E/\QQ$ be an elliptic curve such that $E$ has a point of order $12$ defined over $\QQ$ and the image of the mod-$24$ Galois representation attached to $E$ is conjugate to $H$. Let $q$ be a prime number such that $\sqrt{q}$ is not contained in $\QQ(E[8])$ and $\sqrt{q}$ is not contained in $\QQ(E[3])$. Let $E^{(q)}$ be the quadratic twist of $E$ by $q$. Then the image of the mod-$8$ Galois representation attached to $E^{(q)}$ contains $\operatorname{-Id}$ and the image of the mod-$3$ Galois representation attached to $E^{(q)}$ contains $\operatorname{-Id}$. Thus, no elliptic curve that is $\QQ$-isogenous to $E^{(q)}$ contains a point of order $3$ or $4$ defined over $\QQ$. Hence, the isogeny-torsion graph associated to the $\QQ$-isogeny class of $E^{(q)}$ is $\mathcal{S}^{4}$.
    
\end{itemize}

All isogeny-torsion graphs of type $S$ are simply twists of $\mathcal{S}^{1}$. As $\mathcal{S}^{1}$ corresponds to an infinite set of \textit{j}-invariants, so do $\mathcal{S}^{2}$, $\mathcal{S}^{3}$, and $\mathcal{S}^{4}$. This proves Proposition \ref{S Graphs Proposition}.

\section{Isogeny-Torsion Graphs of $T_{k}$ Type}

\subsection{Isogeny-Torsion Graphs of $T_{4}$ Type}

In this subsection, we prove that each of the three isogeny-torsion graphs of $T_{4}$ type correspond to infinite sets of \textit{j}-invariants.

\begin{proposition}\label{T4 Graphs Proposition}
Let $\mathcal{G}$ be any one of the three isogeny-torsion graphs of $T_{4}$ type (regardless of torsion configuration). Then $\mathcal{G}$ corresponds to an infinite set of \textit{j}-invariants.
\end{proposition}

We will prove this proposition case by case. The methodology is finding the image of the mod-$4$ Galois representation attached to the elliptic curve in the $\Q$-isogeny class with full two-torsion defined over $\Q$ for each of the three isogeny-torsion graphs of $T_{4}$ type. Using the RZB database, we show that each of these subgroups of $\operatorname{GL}(2, \Z / 4 \Z)$ define modular curves with infinitely many non-cuspidal, $\Q$-rational points. Examples of three such modular curves from the RZB database, SZ database, and LMFDB are provided when possible.

 \begin{center}
 \begin{table}[h!]
 	\renewcommand{\arraystretch}{2.5}
 	\begin{tabular}{ |c|c|c|c|c| }
 		\hline
 		Graph Type & Type & Isomorphism Types & Label & RZB, LMFDB, SZ  \\
 		\hline
 		
 		\multirow{3}*{\includegraphics[width=30mm]{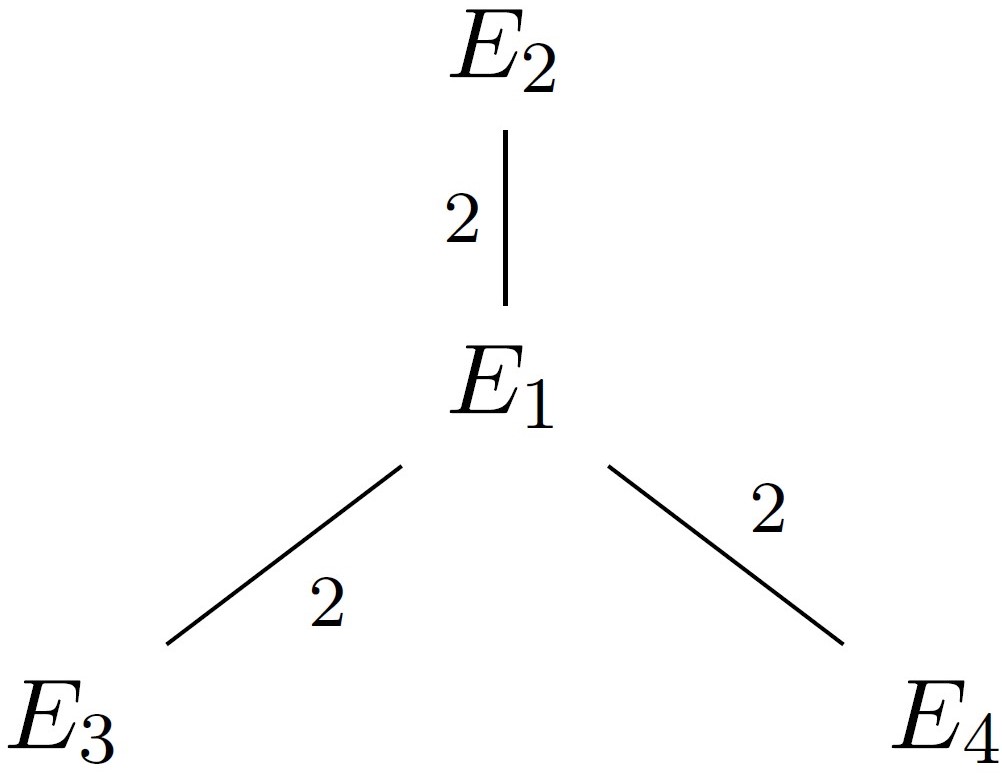}}& \multirow{3}*{$T_{4}$} & $([2,2], [4], [4], [2])$ & $\mathcal{T}_{4}^{1}$ & $\operatorname{H}_{24e}$, \texttt{4.24.0.7}, --- \\
 		\cline{3-5}
 		& & $([2,2], [4], [2], [2])$ & $\mathcal{T}_{4}^{2}$ & $\operatorname{H}_{24d}$, \texttt{4.24.0.4}, --- \\
 		\cline{3-5}
 		& & $([2,2], [2], [2], [2])$ & $\mathcal{T}_{4}^{3}$ & $\operatorname{H}_{24}$, \texttt{4.12.0.4}, $4E^{0}-4b$ \\
 		\hline
 		\end{tabular}
 		\caption{Isogeny-Torsion Graphs of $T_{4}$ Type}
        \label{T_{4} Graphs}
 		\end{table}
 		\end{center}

\begin{itemize}
    \item $\mathcal{T}_{4}^{1}$

Let $E/\Q$ be an elliptic curve with full two-torsion defined over $\Q$. Suppose the image of the mod-$4$ Galois representation attached to $E$ is conjugate to
$$H_{24e} = \left\{ \operatorname{Id}, \left(\begin{array}{cc}
    3 & 0 \\
    0 & 1
\end{array}\right), \left(\begin{array}{cc}
    1 & 2 \\
    2 & 3
\end{array}\right), \left(\begin{array}{cc}
    3 & 2 \\
    2 & 3
\end{array}\right) \right\}.$$
We claim that the isogeny-torsion graph associated to the $\Q$-isogeny class of $E$ is $\mathcal{T}_{4}^{1}$.

A Magma computation reveals that $H_{24e}$ is not conjugate to a subgroup of $\operatorname{GL}(2, \Z / 4 \Z)$ consisting of upper-triangular matrices. Hence $E$ does not have a cyclic, $\Q$-rational subgroup of order $4$. In the course of the proof in Section 13 of \cite{gcal-r} to eliminate the two ``elusive'' non-examples of isogeny-torsion graphs of $S$ type, it was shown that if the image of the mod-$4$ Galois representation attached to $E$ is conjugate to a subgroup of $\left\langle H_{24e}, \operatorname{-Id} \right\rangle$, then $E$ does not contain a $\Q$-rational subgroup of order $3$. Thus, the isogeny-torsion graph associated to the $\Q$-isogeny class of $E$ is of type $T_{4}$.

We must prove the isogeny-torsion graph associated to the $\Q$-isogeny class of $E$ is $\mathcal{T}_{4}^{1}$. To do this, we will show that $E$ is $\Q$-isogenous to two distinct elliptic curves over $\Q$ with rational $4$-torsion (see Table \ref{T_{4} Graphs}). Let us pick a basis $\{P_{4}, Q_{4}\}$ of $E[4]$ such that the image of the mod-$4$ Galois representation attached to $E$ is $H_{24e}$. Let $\phi_{1} \colon E \to E \left\langle P_{2} + Q_{2} \right\rangle$ and $\phi_{2} \colon E \to E \left\langle P_{2} \right\rangle$ be isogenies with kernel generated by $P_{2} + Q_{2}$ and $P_{2}$ respectively. Note that for each $\sigma \in G_{\Q}$, $\sigma(Q_{4}) = Q_{4}$ or $P_{2} + [3]Q_{4}$. Hence $\sigma(Q_{4}) - Q_{4} \in \left\langle P_{2} + Q_{2} \right\rangle$ for all $\sigma \in G_{\Q}$. By Lemma \ref{lem-necessity-for-point-rationality}, $\phi_{1}(Q_{4})$ is a point of order $4$ defined over $\Q$. Also, adding the two column vectors in each element of $H_{24e}$, we see that for each $\sigma \in G_{\Q}$, $\sigma(P_{4} + Q_{4}) = P_{4} + Q_{4}$ or $[3]P_{4} + Q_{4}$. Hence, $\sigma(P_{4}+Q_{4}) - (P_{4}+Q_{4}) \in \left\langle P_{2} \right\rangle$ for all $\sigma \in G_{\Q}$. By Lemma \ref{lem-necessity-for-point-rationality}, $\phi_{2}(P_{4}+Q_{4})$ is a point of order $4$ defined over $\Q$. This is enough to prove that the isogeny-torsion graph associated to the $\Q$-isogeny class of $E$ is $\mathcal{T}_{4}^{1}$.

Elliptic curves over $\Q$ whose transpose of the image of the mod-$4$ Galois representation is conjugate to a subgroup of $H_{24e}$ are in one-to-one correspondence with non-cuspidal, $\Q$-rational points on the modular curve $\operatorname{X}_{24e}$ from the list compiled in \cite{Rouse}. The modular curve $\operatorname{X}_{24e}$ is a genus $0$ curve and has infinitely many non-cuspidal, $\Q$-rational points. By Proposition \ref{HIT proposition}, there are infinitely many \textit{j}-invariants corresponding to elliptic curves over $\Q$ such that the image of the mod-$4$ Galois representation is conjugate to $H_{24e}$ (not a proper subgroup of $H_{24e}$). Thus $\mathcal{T}_{4}^{1}$ corresponds to an infinite list of \textit{j}-invariants.

\item $\mathcal{T}_{4}^{2}$

Let $E / \Q$ be an elliptic curve over $\Q$ such that the image of the mod-$4$ Galois representation attached to $E$ is conjugate to
$$\operatorname{H}_{24d} = \left\{I, \left(\begin{array}{cc}
    1 & 0 \\
    0 & 3
\end{array}\right), \left(\begin{array}{cc}
    1 & 2 \\
    2 & 3
\end{array}\right), \left(\begin{array}{cc}
    1 & 2 \\
    2 & 1
\end{array}\right)\right\}.$$

In the course of the proof in Section 13 of \cite{gcal-r} to eliminate the two ``elusive'' non-examples of isogeny-torsion graphs of $S$ type, it was shown that if the image of the mod-$4$ Galois representation attached to $E$ is conjugate to a subgroup of $\left\langle \operatorname{H}_{24d}, \operatorname{-Id} \right\rangle = \left\langle \operatorname{H}_{24e}, \operatorname{-Id} \right\rangle$, then $E$ does not contain a $\Q$-rational subgroup of order $3$. A Magma computation reveals that $\operatorname{H}_{24d}$ is not conjugate to a subgroup of $\operatorname{GL}(2, \Z / 4 \Z)$ consisting of upper-triangular matrices. Hence, $E$ does not have a cyclic, $\Q$-rational subgroup of order $4$. Thus, the isogeny-torsion graph associated to the $\Q$-isogeny class of $E$ is of type $T_{4}$. It remains to prove that $E$ is $\Q$-isogenous to one and only one elliptic curve over $\Q$ with rational $4$-torsion.

Let $\{P_{4}, Q_{4}\}$ be a basis of $E[4]$ such that the image of the mod-$4$ Galois representation attached to $E$ is $F_{2}$. Let $\phi_{1} \colon E \to E / \left\langle P_{2} \right\rangle$, $\phi_{2} \colon E \to E / \left\langle Q_{2} \right\rangle$, and $\phi_{3} \colon E \to E / \left\langle P_{2} + Q_{2} \right\rangle$ be isogenies with kernels generated by $P_{2}, Q_{2},$ and $P_{2}+Q_{2}$ respectively. Note that for each $\sigma \in G_{\Q}$, $\sigma(P_{4}) = P_{4}$ or $P_{4} + Q_{2}$. Hence, $\sigma(P_{4}) - P_{4} \in \left\langle Q_{2} \right\rangle$ for all $\sigma \in G_{\Q}$. By Lemma \ref{lem-necessity-for-point-rationality}, $\phi_{2}(P_{4})$ is defined over $\Q$ and thus, $E / \left\langle Q_{2} \right\rangle(\Q)_{\text{tors}} \cong \Z / 4 \Z$. By Lemma \ref{lem-necessity-for-point-rationality}, $\phi_{1}(Q_{2})$ and $\phi_{3}(Q_{2})$ are points of order $2$ defined over $\Q$. The point $\phi_{1}(Q_{2})$ lives in two cyclic subgroups of order $4$, namely, $\left\langle \phi_{1}(Q_{4}) \right\rangle$ and $\left\langle \phi_{1}(P_{4}+Q_{4}) \right\rangle$. Let $\sigma_{1}$ be a Galois automorphism that maps $Q_{4}$ to $[3]Q_{4}$. Then $\sigma_{1}(Q_{4}) - Q_{4} = [3]Q_{4} - Q_{4} = Q_{2} \notin \left\langle P_{2} \right\rangle$. Hence, $\phi_{1}(Q_{4})$ is not defined over $\Q$. Let $\sigma_{3}$ be a Galois automorphism that maps $P_{4}$ to $P_{4}+Q_{2}$ and $Q_{4}$ to $P_{2} + Q_{4}$. Then $\sigma_{3}(P_{4}+Q_{4}) - (P_{4}+Q_{4}) = P_{2} + Q_{2} \notin \left\langle P_{4} \right\rangle$. Hence, $\phi_{1}(P_{4}+Q_{4})$ is not defined over $\Q$ and thus, $E / \left\langle P_{4} \right\rangle(\Q)_{\text{tors}} \cong \Z / 2 \Z$. The point $\phi_{3}(Q_{2})$ lives in two cyclic, $\Q$-rational groups of order $4$, namely $\left\langle \phi_{3}(Q_{4}) \right\rangle$ and $\left\langle \phi_{3}(P_{4}) \right\rangle$. Note that $\sigma_{1}(Q_{4}) - Q_{4} = Q_{2} \notin \left\langle P_{2} + Q_{2} \right\rangle$ and $\sigma_{3}(P_{4}) - P_{4} = P_{4} + Q_{2} - P_{4} = Q_{2} \notin \left\langle P_{2} + Q_{2} \right\rangle$. Hence, $E / \left\langle P_{2} + Q_{2} \right\rangle(\Q)_{\text{tors}} \cong \Z / 2 \Z$.

Elliptic curves over $\Q$ such that the transpose of the image of the mod-$4$ Galois representation is conjugate to $\operatorname{H}_{24d}$ correspond to non-cuspidal, $\Q$-rational points on the modular curve $\operatorname{X}_{24d}$ in \cite{Rouse} which is a genus $0$ curve with infinitely many non-cuspidal, $\Q$-rational points. Using \ref{HIT proposition}, there are infinitely many \textit{j}-invariants corresponding to elliptic curves over $\Q$ such that the image of the mod-$4$ Galois representation is conjugate to $\operatorname{H}_{24d}$ (not a proper subgroup of $\operatorname{H}_{24d}$). Hence, $\mathcal{T}_{4}^{2}$ corresponds to an infinite set of \textit{j}-invariants.

\item $\mathcal{T}_{4}^{3}$

Let $E / \Q$ be an elliptic curve such that the image of the mod-$4$ Galois representation attached to $E$ is conjugate to $\operatorname{H}_{24} = \left\langle \operatorname{H}_{24e}, \operatorname{-Id} \right\rangle = \left\langle \operatorname{H}_{24d}, \operatorname{-Id} \right\rangle$. Note that again, as proven in \cite{gcal-r}, that $E$ does not have a $\Q$-rational subgroup of order $3$. Also, $E$ has full two-torsion defined over $\Q$ and does not have a cyclic, $\Q$-rational subgroup of order $4$. Thus, the isogeny-torsion graph associated to the $\Q$-isogeny class of $E$ is of $T_{4}$ type.

By Lemma \ref{No Points of Order 4}, no elliptic curve over $\Q$ that is $\Q$-isogenous to $E$ has a point of order $4$ defined over $\Q$. Hence, the isogeny-torsion graph associated to $E$ is $\mathcal{T}_{4}^{3}$. Note that elliptic curves over $\Q$ such that the image of the mod-$4$ Galois representation is conjugate to a subgroup of $\operatorname{H}_{24}$ correspond to non-cuspidal, $\Q$-rational points on the modular curve defined by the group with label $4E^{0}-4b$ in \cite{SZ}. Elliptic curves over $\Q$ such that the image of the mod-$4$ Galois representation is conjugate to a subgroup of the transpose of $\operatorname{H}_{24}$ correspond to non-cuspidal, $\Q$-rational points on the modular curve $\operatorname{X}_{24}$ in \cite{Rouse}. Both of these modular curves are genus $0$ curves with infinitely many non-cuspidal, $\Q$-rational points. Applying Proposition \ref{HIT proposition}, we can conclude that there are infinitely many \textit{j}-invariants corresponding to elliptic curves over $\QQ$ such that the image of the mod-$4$ Galois representation is conjugate to $\operatorname{H}_{24}$. Thus, $\mathcal{T}_{4}^{3}$ corresponds to an infinite set of \textit{j}-invariants.
\end{itemize}

This concludes the proof of Proposition \ref{T4 Graphs Proposition}.

{\bf N.B.} This is not the \textit{only} way to generate isogeny-torsion graphs of $T_{4}$ type nor prove that they correspond to infinite sets of \textit{j}-invariants. It is possible for example, to have an elliptic curve over $\Q$ with Galois image mod-$4$ conjugate to the full inverse image of the trivial subgroup of $\operatorname{GL}(2, \Z / 2 \Z)$ via the mod-$4$ reduction map from $\operatorname{GL}(2, \Z / 4 \Z)$ and a surjective mod-$3$ Galois representation. The isogeny-torsion graph associated to a $\Q$-isogeny class containing such an elliptic curve is $\mathcal{T}_{4}^{3}$.

		\subsection{Isogeny-Torsion Graphs of $T_{6}$ Type}
		
		In this subsection, we prove that each of the four isogeny-torsion graphs of $T_{6}$ type correspond to infinite sets of \textit{j}-invariants.
		
		\begin{proposition}\label{T6 Graphs Proposition}
		Let $\mathcal{G}$ be an isogeny-torsion graph of type $T_{6}$ (regardless of torsion configuration). Then $\mathcal{G}$ corresponds to an infinite set of \textit{j}-invariants.
		\end{proposition}
		
		We will prove this proposition case by case. The methodology is finding the image of the mod-$8$ Galois representation attached one of the two elliptic curves in the $\Q$-isogeny class with full two-torsion defined over $\Q$ for each of the four isogeny-torsion graphs of $T_{6}$ type. Using the RZB database, we show that each of these subgroups of $\operatorname{GL}(2, \Z / 8 \Z)$ define modular curves with infinitely many non-cuspidal, $\Q$-rational points. Examples of four such modular curves from the RZB database, LMFDB, and SZ database are provided when possible.
		
		\begin{center}
		 \begin{table}[h!]
 	\renewcommand{\arraystretch}{1.3}
 	\begin{tabular}{ |c|c|c|c|c| }
 		\hline
 		Isogeny Graph & Type & Isomorphism Types & Label & RZB, LMFDB, SZ  \\
 		\hline

		\multirow{4}*{\includegraphics[scale=0.21]{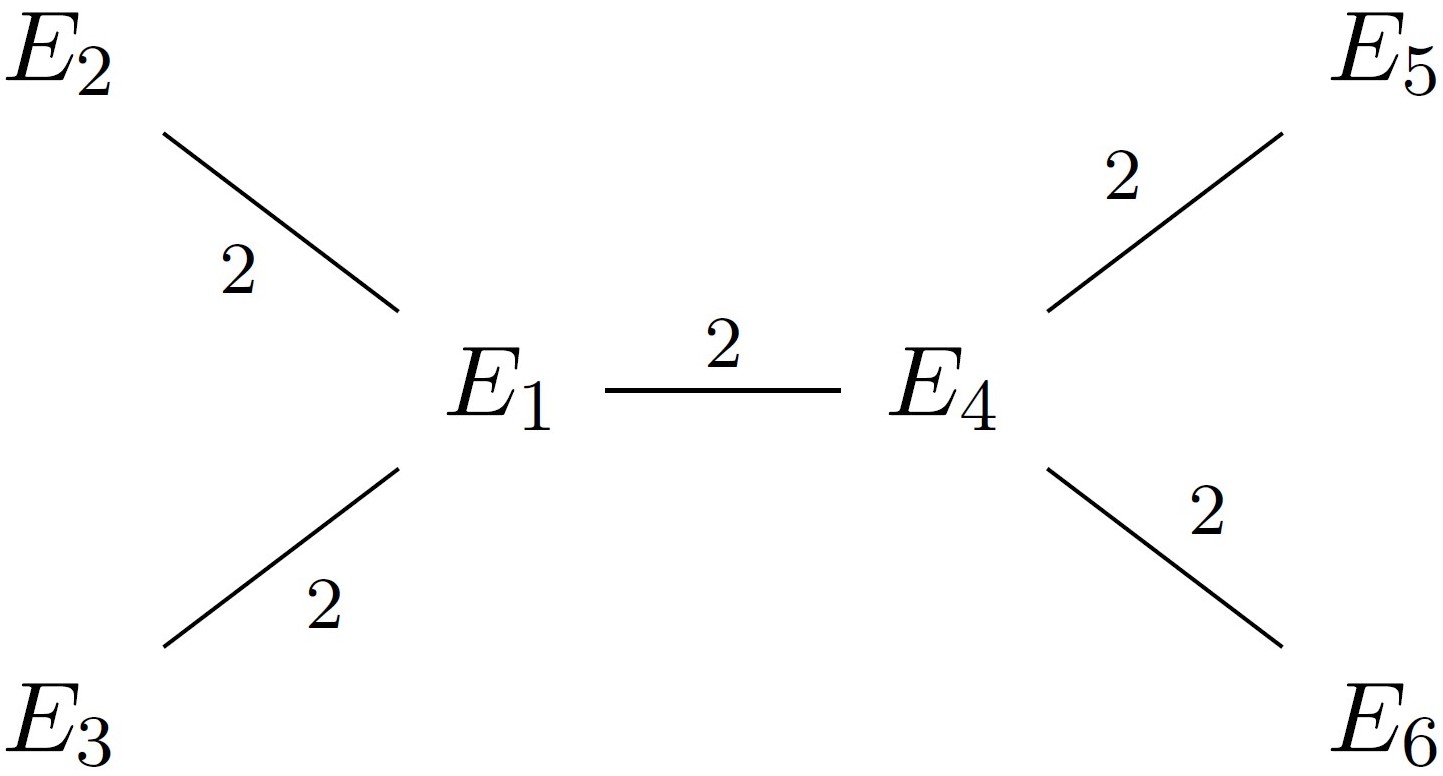}}& \multirow{4}*{$T_{6}$} & $([2,4],[8],[4],[2,2],[2],[2])$ & $\mathcal{T}^{1}_{6}$ & $\operatorname{H}_{98e}$, \texttt{8.48.0.24}, --- \\
 		\cline{3-5}
 		& & $([2,4],[4],[4],[2,2],[2],[2])$ & $\mathcal{T}^{2}_{6}$ & $\operatorname{H}_{98h}$, \texttt{8.48.0.32}, --- \\
 		\cline{3-5}
 		& & $([2,2],[4],[2],[2,2],[2],[2])$ & $\mathcal{T}^{3}_{6}$ & $\operatorname{H}_{98o}$, \texttt{8.48.0.66}, --- \\
 		\cline{3-5}
 		& & $([2,2],[2],[2],[2,2],[2],[2])$ & $\mathcal{T}^{4}_{6}$ & $\operatorname{H}_{98}$, \texttt{8.24.0.10}, $8J^{0}-8c$ \\
 		\hline
 		
 		\end{tabular}
 		\caption{Isogeny-Torsion Graphs of $T_{6}$ Type}
        \label{T_{6} Graphs}
 		\end{table}
 		\end{center}
        
		Let $E / \Q$ be an elliptic curve such that $E(\Q)_{\text{tors}} = \left\langle P_{2}, Q_{4} \right\rangle \cong \Z / 2 \Z \times \Z / 4 \Z$. By Theorem \ref{thm-kenku}, $E$ does not contain a non-trivial, cyclic, $\Q$-rational subgroup of odd order. Suppose the isogeny-torsion graph associated to $E$ is of type $T_{6}$. Then $E$ is represented in the isogeny-torsion graph in Table \ref{T_{6} Graphs} by $E_{1}$. The isogeny-torsion graph associated to the $\Q$-isogeny class of $E$ is $\mathcal{T}_{6}^{1}$ if $E$ is $\Q$-isogenous to an elliptic curve over $\Q$ with a point of order $8$ defined over $\Q$ and the isogeny-torsion graph associated to the $\Q$-isogeny class of $E$ is $\mathcal{T}_{6}^{2}$ otherwise.
		
		\begin{itemize}
		
		\item $\mathcal{T}_{6}^{1}$
		
		We may assume $E / \left\langle P_{2} \right\rangle(\Q)_{\text{tors}} \cong \Z / 8 \Z$. Let $\phi \colon E \to E / \left\langle P_{2} \right\rangle$ be an isogeny with kernel $\left\langle P_{2} \right\rangle$. By Lemma \ref{lem-necessity-for-point-rationality}, $\phi(Q_{4})$ is a point of order $4$ defined over $\Q$. The cyclic, $\Q$-rational subgroups of $E / \left\langle P_{2} \right\rangle$ of order $8$ that contain $\phi(Q_{4})$ are $\left\langle \phi(Q_{8}) \right\rangle$ and $\left\langle \phi(P_{4} + Q_{8}) \right\rangle$. Let us say that $\phi(Q_{8})$ is defined over $\Q$. Then $\sigma(Q_{8}) - Q_{8} \in \left\langle P_{2} \right\rangle$ for all $\sigma \in G_{\Q}$ by Lemma \ref{lem-necessity-for-point-rationality}. For each Galois automorphism $\sigma \in G_{\Q}$, there exist integers $a, b, c \in \Z$ (depending on $\sigma$) such that $\sigma(P_{8}) = [2a+1]P_{8} + [2b]Q_{8}$ and $\sigma(Q_{8}) = [c]P_{2} + Q_{8}$. Thus, the image of the mod-$8$ Galois representation attached to $E$ is conjugate to a subgroup of
		$$H_{98e} = \left\langle \left(\begin{array}{cc}
		    3 & 0 \\
		    0 & 1
		\end{array}\right), \left(\begin{array}{cc}
		    5 & 0 \\
		    0 & 1
		\end{array}\right), \left(\begin{array}{cc}
		    1 & 0 \\
		    2 & 1
		\end{array}\right), \left(\begin{array}{cc}
		    1 & 4 \\
		    0 & 1
		\end{array}\right) \right\rangle.$$
		
		Elliptic curves over $\Q$ whose transpose of the image of the mod-$8$ Galois representation is conjugate to a subgroup of $H_{98e}$ correspond to non-cuspidal, $\Q$-rational points on the modular curve in the list compiled in \cite{Rouse} labeled $\operatorname{X}_{98e}$ which is a genus $0$ modular curve with infinitely many non-cuspidal, $\Q$-rational points. By Proposition \ref{HIT proposition}, there are infinitely many \textit{j}-invariants corresponding to elliptic curves over $\Q$ such that the image of the mod-$8$ Galois representation is conjugate to $H_{98e}$ itself (not a proper subgroup of $H_{98e}$). The image of the mod-$8$ Galois representation being conjugate to $H_{98e}$ guarantees that the isogeny-torsion graph associated to the $\Q$-isogeny class of $E$ is precisely $\mathcal{T}_{6}^{1}$. Thus, $\mathcal{T}_{6}^{1}$ corresponds to an infinite set of \textit{j}-invariants.
		
		\item $\mathcal{T}_{6}^{2}$
		
		Let $E / \Q$ be an elliptic curve such that the image of the mod-$8$ Galois representation associated to $E$ is conjugate to
		
		$$H_{98h} = \left\langle \left(\begin{array}{cc}
		    5 & 0 \\
		    0 & 1
		\end{array}\right), \left(\begin{array}{cc}
		    1 & 0 \\
		    2 & 1
		\end{array}\right), \left(\begin{array}{cc}
		    1 & 0 \\
		    0 & 7
		\end{array}\right), \left(\begin{array}{cc}
		    1 & 4 \\
		    0 & 1
		\end{array}\right) \right\rangle.$$
		Note that $H_{98h}$ is a quadratic twist of the group $H_{98e}$ (multiply the first generator of $H_{98e}$ by $\operatorname{-Id}$ to get $H_{98h}$). The image of the mod-$8$ Galois representation attached to $E$ is generated by $\sigma_{1}, \sigma_{2}, \sigma_{3}, \sigma_{4} \in G_{\Q}$ such that $\sigma_{i}$ is represented by the i-th generator in $H_{98h}$. For example, $\sigma_{1}(P_{8}) = [5]P_{8}$ and fixes $Q_{8}$. We claim that $E / \left\langle Q_{4} \right\rangle(\Q)_{\text{tors}} \cong E / \left\langle P_{2} + Q_{4} \right\rangle(\Q)_{\text{tors}} \cong \Z / 4 \Z$ and that would be enough to conclude that the isogeny-torsion graph associated to the $\Q$-isogeny class of $E$ is $\mathcal{T}_{6}^{2}$ (see Table \ref{T_{6} Graphs}).
		
		To generalize, let $A$ equal $Q_{4}$ or $P_{2} + Q_{4}$. We will prove that $E / \left\langle A \right\rangle(\Q)_{\text{tors}} \cong \Z / 4 \Z$. Let $\phi \colon E \to E / \left\langle A \right\rangle$ be an isogeny with kernel $\left\langle A \right\rangle$. Note that $\sigma_{i}(P_{4}) - P_{4} = \mathcal{O}$ or $\sigma_{i}(P_{4}) - P_{4} = Q_{2}$ for all $i = 1,2,3,4$ and thus $\sigma(P_{4}) - P_{4} \in \left\langle Q_{2} \right\rangle$ for all $\sigma \in G_{\Q}$. By Lemma \ref{lem-necessity-for-point-rationality}, $\phi(P_{4})$ is a point of order $4$ defined over $\Q$. The group $E / \left\langle A \right\rangle(\Q)_{\text{tors}}$ is cyclic of order $4$ or $8$ by Lemma \ref{lem-2torspt-all-have-2torspt} and Lemma \ref{lem-Maximality-Of-Rational-2-Power-Groups}. We need to prove that it is \textit{not} of order $8$.
		
		Let $B \in E$ such that $B = Q_{8}$ if $A = Q_{4}$ and $B = P_{4} + Q_{8}$ if $A = P_{2} + Q_{4}$. Note that $[2]B = A$. By Lemma \ref{lem-Maximality-Of-Rational-2-Power-Groups}, $\phi(B)$ is a point of order $2$ that is not defined over $\Q$. The cyclic groups of order $8$ that contain $\phi(P_{4})$ are $\left\langle \phi(P_{8}) \right\rangle$ and $\left\langle \phi(P_{8} + B) \right\rangle$. Note that $\sigma_{1}(P_{8}) - P_{8} = [4]P_{8} = P_{2} \notin \left\langle A \right\rangle$. Also, $\sigma_{1}$ fixes $B$, so we have that $\sigma_{1}(P_{8} + B) - (P_{8} + B) = P_{2} \notin \left\langle A \right\rangle$. Thus, both $\phi(P_{8})$ and $\phi(P_{8} + B)$ are not defined over $\Q$ by Lemma \ref{lem-necessity-for-point-rationality}. Thus, $E / \left\langle Q_{4} \right\rangle(\Q)_{\text{tors}} \cong E / \left\langle P_{2} + Q_{4} \right\rangle(\Q)_{\text{tors}} \cong \Z / 4 \Z$ and so, the isogeny-torsion graph associated to the $\Q$-isogeny class associated to $E$ is $\mathcal{T}_{6}^{2}$.
		
		Elliptic curves over $\Q$ such that the transpose of the image of the mod-$8$ Galois representation is conjugate to a subgroup of $H_{98h}$ correspond to non-cuspidal, $\Q$-rational points on the modular curve that appears in the list compiled in \cite{Rouse} with label $\operatorname{X}_{98h}$ which is a genus $0$ curve with infinitely many non-cuspidal, $\Q$-rational points. By Proposition \ref{HIT proposition}, there are infinitely many \textit{j}-invariants corresponding to elliptic curves over $\Q$ such that the image of the mod-$8$ Galois representation is conjugate to $H_{98h}$ itself (not a proper subgroup of $H_{98h}$).  The image of the mod-$8$ Galois representation being conjugate to $H_{98h}$ guarantees that the isogeny-torsion graph associated to the $\Q$-isogeny class of $E$ is precisely $\mathcal{T}_{6}^{2}$. Thus, $\mathcal{T}_{6}^{2}$ corresponds to an infinite set of \textit{j}-invariants.
		\end{itemize}
		
		Now let $E / \Q$ be an elliptic curve such that $E(\Q)_{\text{tors}} \cong \Z / 2 \Z \times \Z / 2 \Z$ and the isogeny-torsion graph associated to $E$ is of type $T_{6}$ but this time, $E$ is not $\Q$-isogenous to an elliptic curve over $\Q$ with torsion subgroup of order $8$. Again, we may assume without loss of generality that $E$ is represented in the isogeny-torsion graph in Table \ref{T_{6} Graphs} by $E_{1}$. If $E$ is $\Q$-isogenous to an elliptic curve over $\Q$ with a point of order $4$ defined over $\Q$, then the isogeny-torsion graph associated to the $\Q$-isogeny class of $E$ is $\mathcal{T}_{6}^{3}$. If $E$ is \textit{not} $\Q$-isogenous to an elliptic curve over $\Q$ with a point of order $4$ defined over $\Q$, then the isogeny-torsion graph associated to the $\Q$-isogeny class of $E$ is $\mathcal{T}_{6}^{4}$.
		
		\begin{itemize}
		\item $\mathcal{T}_{6}^{3}$
	    	
		Let $E / \Q$ be an elliptic curve such that the image of the mod-$8$ Galois representation attached to $E$ is conjugate to
		$$H_{98o} = \left\langle \left(\begin{array}{cc}
		    3 & 0 \\
		    0 & 1
		\end{array}\right), \left(\begin{array}{cc}
		    5 & 0 \\
		    0 & 1
		\end{array}\right), \left(\begin{array}{cc}
		    1 & 0 \\
		    2 & 1
		\end{array}\right), \left(\begin{array}{cc}
		    7 & 4 \\
		    0 & 7
		\end{array}\right) \right\rangle.$$
		Note that $H_{3}$ is a quadratic twist of $H_{98e}$ (multiply the fourth generator of $H$ by $\operatorname{-Id}$). The image of the mod-$8$ Galois representation attached to $E$ is generated by $\sigma_{1}, \sigma_{2}, \sigma_{3}, \sigma_{4} \in G_{\Q}$ such that $\sigma_{1}(P_{8}) = [3]P_{8}$, etc. We claim that the isogeny-torsion graph associated to the $\Q$-isogeny class of $E$ is $\mathcal{T}_{6}^{3}$.
		
		Let $\phi \colon E \to E / \left\langle P_{2} + Q_{4} \right\rangle$ be an isogeny with kernel $\left\langle P_{2} + Q_{4} \right\rangle$. Note that the group $E / \left\langle P_{2} + Q_{4} \right\rangle(\Q)_{\text{tors}}$ is cyclic of even order by Lemma \ref{lem-Maximality-Of-Rational-2-Power-Groups} and Lemma \ref{lem-2torspt-all-have-2torspt}. We claim that $\phi(Q_{8})$ is a point of order $4$ defined over $\Q$. Note that $\sigma_{4}(Q_{8}) - Q_{8} = P_{2} + [7]Q_{8} - Q_{8} = P_{2} + [3]Q_{4} = [3](P_{2} + Q_{4}) \in \left\langle P_{2} + Q_{4} \right\rangle$. The other generators fix $Q_{8}$ and hence, $\phi(Q_{8})$ is a point of order $4$ defined over $\Q$ by Lemma \ref{lem-necessity-for-point-rationality} and so $E / \left\langle P_{2}+Q_{4} \right\rangle(\Q)_{\text{tors}}$ is cyclic of order $4$ or $8$.
		
		Now let $\phi' \colon E \to E / \left\langle Q_{4} \right\rangle$ be an isogeny with kernel $\left\langle Q_{4} \right\rangle$. The group $E / \left\langle Q_{4} \right\rangle(\Q)_{\text{tors}}$ is cyclic of even order by Lemma \ref{lem-2torspt-all-have-2torspt} and Lemma \ref{lem-Maximality-Of-Rational-2-Power-Groups}. The point $\phi'(P_{2})$ is of order $2$ and is defined over $\Q$ by Lemma \ref{lem-necessity-for-point-rationality}. The cyclic subgroups of $E / \left\langle Q_{4} \right\rangle$ of order $4$ that contain $\phi'(P_{2})$ are $\left\langle \phi'(P_{4}) \right\rangle$ and $\left\langle \phi'(P_{4} + Q_{8}) \right\rangle$. We need to prove both of these groups of order $4$ are not generated by points defined over $\Q$.
		
		First we prove that $\phi'(P_{4})$ is not defined over $\Q$. Observe that $\sigma_{1}(P_{8}) = [3]P_{8}$ and hence, $\sigma_{1}(P_{4}) = [3]P_{4}$. From this, we can say $\sigma_{1}(P_{4}) - P_{4} = [3]P_{4} - P_{4} = P_{2} \notin \left\langle Q_{4} \right\rangle$. Thus, $\phi'(P_{4})$ is not defined over $\Q$ by Lemma \ref{lem-necessity-for-point-rationality}. Note that $\sigma_{1}$ fixes $Q_{8}$. Thus, $\sigma_{1}(P_{4} + Q_{8}) - (P_{4} + Q_{8}) = P_{2} \notin \left\langle Q_{4} \right\rangle$. From Lemma \ref{lem-necessity-for-point-rationality}, we can conclude that $\phi'(P_{4} + Q_{8})$ is not defined over $\Q$. Thus, $E / \left\langle P_{2} + Q_{4} \right\rangle(\Q)_{\text{tors}} \cong \Z / 2 \Z$ and this forces $E / \left\langle Q_{2} \right\rangle(\Q)_{\text{tors}} \cong \Z / 4 \Z$ and for the isogeny-torsion graph associated to the $\Q$-isogeny class of $E$ to be $\mathcal{T}_{6}^{3}$ (see Table \ref{T_{6} Graphs}).
		
		Elliptic curves over $\Q$ whose transpose of the image of the mod-$8$ Galois representation is conjugate to a subgroup of $H_{3}$ correspond to non-cuspidal, $\Q$-rational points on the modular curve that appears in the list compiled in \cite{Rouse} with label $\operatorname{X}_{98o}$ which is a genus $0$ curve with infinitely many non-cuspidal, $\Q$-rational points. By Proposition \ref{HIT proposition}, there are infinitely many \textit{j}-invariants corresponding to elliptic curves over $\Q$ such that the image of the mod-$8$ Galois representation is conjugate to $H_{3}$ (not a proper subgroup of $H_{3}$). The image of the mod-$8$ Galois representation being conjugate to $H_{3}$ guarantees that the isogeny-torsion graph attached to the $\Q$-isogeny class of $E$ is $\mathcal{T}_{6}^{3}$. Thus $\mathcal{T}_{6}^{3}$ corresponds to an infinite set of \textit{j}-invariants.
		
		\item $\mathcal{T}_{6}^{4}$
		
		Let $E / \Q$ be an elliptic curve such that the image of the mod-$8$ Galois representation attached to $E$ is conjugate to $H_{98} = \left\langle H_{98e}, \operatorname{-Id} \right\rangle$. Then by Lemma \ref{No Points of Order 4}, no elliptic curve over $\Q$ in the $\Q$-isogeny class of $E$ has a point of order $4$ defined over $\Q$. Hence, the isogeny-torsion graph associated to the $\Q$-isogeny class of $E$ is $\mathcal{T}_{6}^{4}$.

		The group $H_{98}$ is conjugate to the group that appears in the list compiled in \cite{SZ} with label $8J^{0}-8c$ which defines a modular curve of genus $0$ with infinitely many non-cuspidal, $\Q$-rational points. Elliptic curves over $\Q$ such that the transpose of the image of the mod-$8$ Galois representation is conjugate to a subgroup of $H_{98}$ correspond to non-cuspidal, $\Q$-rational points on the modular curve in \cite{Rouse} with label $\operatorname{X}_{98}$ which is a modular curve of genus $0$ with infinitely many non-cuspidal, $\Q$-rational points. By Proposition \ref{HIT proposition} there are infinitely many \textit{j}-invariants corresponding to elliptic curves over $\Q$ such that the transpose of the image of the mod-$8$ Galois representation is conjugate to $H_{98}$ (not a proper subgroup of $H_{98}$). The transpose of the image of the mod-$8$ Galois representation being conjugate to $H_{98}$ guarantees that the isogeny-torsion graph attached to the $\Q$-isogeny class of $E$ is $\mathcal{T}_{6}^{4}$. Thus, $\mathcal{T}_{6}^{4}$ corresponds to an infinite set of \textit{j}-invariants.
		\end{itemize}
		
		This concludes the proof of Proposition \ref{T6 Graphs Proposition}.

		\subsection{Isogeny-Torsion Graphs of $T_{8}$ Type}
		In this subsection, we prove that each of the six isogeny-torsion graphs of $T_{8}$ type correspond to infinite sets of \textit{j}-invariants.
        
        \begin{proposition}\label{T8 Graphs Proposition}
        Let $\mathcal{G}$ be an isogeny-torsion graph of type $T_{8}$ (regardless of torsion configuration). Then $\mathcal{G}$ corresponds to an infinite set of \textit{j}-invariants.
        \end{proposition}
        
        We will prove this proposition case by case. The methodology to proving $\mathcal{T}_{8}^{1}$ and $\mathcal{T}_{8}^{2}$ correspond to infinite sets of \textit{j}-invariants is finding the image of the mod-$8$ Galois representation attached to one of the three elliptic curves in the $\Q$-isogeny class with full two-torsion defined over $\Q$. Using the RZB database, we show that each of these subgroups of $\operatorname{GL}(2, \Z / 8 \Z)$ define modular curves with infinitely many non-cuspidal, $\Q$-rational points.
        
        For the isogeny-torsion graphs $\mathcal{T}_{8}^{3}, \mathcal{T}_{8}^{4}, \mathcal{T}_{8}^{5},$ and $\mathcal{T}_{8}^{6}$, it is most convenient to compute the image of the mod-$16$ Galois representation attached to one of the elliptic curves over $\Q$ with full two-torsion defined over $\Q$. Then, again, we will use the RZB database to show that each of these subgroups of $\operatorname{GL}(2, \Z / 16 \Z)$ define modular curves with infinitely many non-cuspidal, $\Q$-rational points. Examples of six such modular curves from the RZB database, LMFDB, and SZ database are provided when possible.

		\begin{center}
		\begin{table}[h!]
 	\renewcommand{\arraystretch}{1.3}
 	\scalebox{0.85}{
 	\begin{tabular}{ |c|c|c|c|c| }
 		\hline
 		Isogeny Graph & Type & Isomorphism Types & Label & RZB, LMFDB, SZ \\
 		\hline

		\multirow{6}*{\includegraphics[scale=0.32]{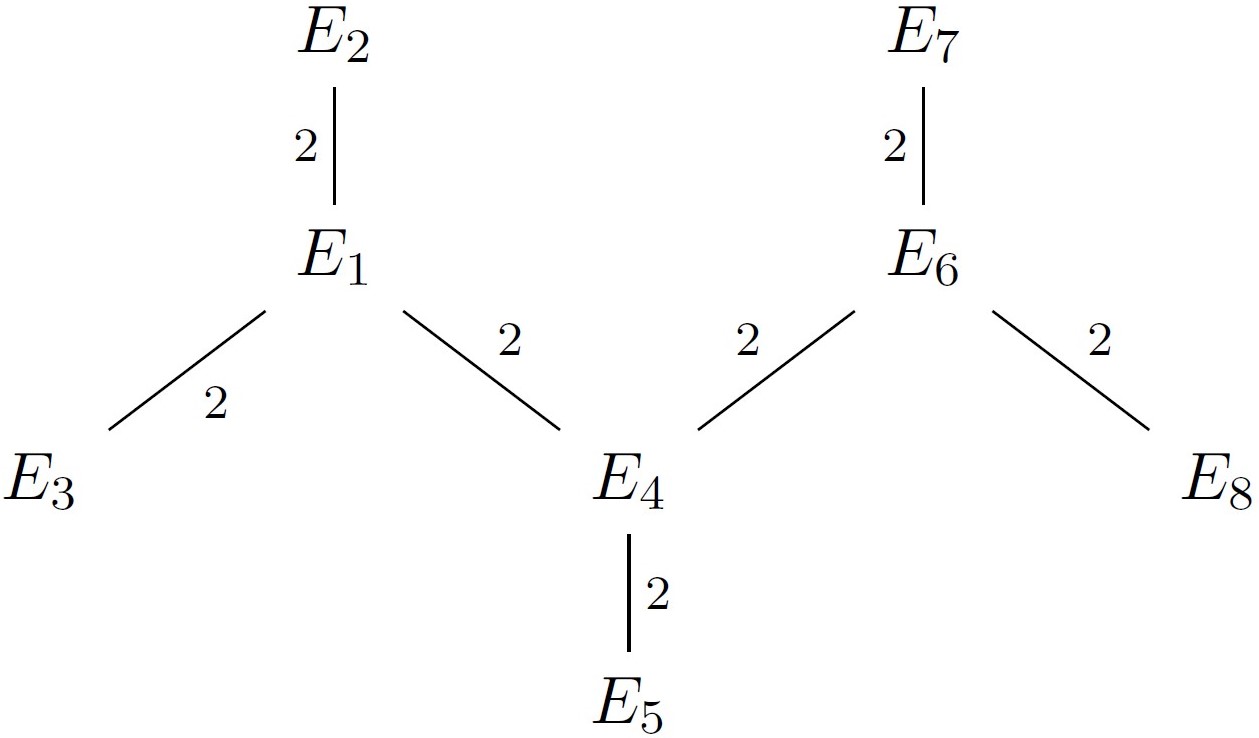}}& \multirow{6}*{$T_{8}$} & $([2,8],[8],[8],[2,4],[4],[2,2],[2],[2])$ & $\mathcal{T}_{8}^{1}$ & $\operatorname{H}_{193n}$, \texttt{8.96.0.40}, --- \\
 		\cline{3-5}
 		& & $([2,4],[8],[4],[2,4],[4],[2,2],[2],[2])$ & $\mathcal{T}_{8}^{2}$ & $\operatorname{H}_{194l}$, \texttt{8.96.0.39}, --- \\
 		\cline{3-5}
 		& & $([2,4],[4],[4],[2,4],[8],[2,2],[2],[2])$ & $\mathcal{T}_{8}^{3}$ & $\operatorname{H}_{215c}$, \texttt{16.96.0.7}, --- \\
 		\cline{3-5}
 		& & $([2,4],[4],[4],[2,4],[4],[2,2],[2],[2])$ & $\mathcal{T}_{8}^{4}$ & $\operatorname{H}_{215l}$, \texttt{16.96.0.9}, --- \\
 		\cline{3-5}
 		& & $([2,2],[4],[2],[2,2],[2],[2,2],[2],[2])$ & $\mathcal{T}_{8}^{5}$ & $\operatorname{H}_{215k}$, \texttt{16.96.0.61}, --- \\
 		\cline{3-5}
 		& & $([2,2],[2],[2],[2,2],[2],[2,2],[2],[2])$ & $\mathcal{T}_{8}^{6}$ & $\operatorname{H}_{215}$, \texttt{16.48.0.3}, $16G^{0}-16b$ \\
 		
 		\hline
 	\end{tabular}}
 	\caption{$T_{8}$ Type Isogeny Graphs}
    \label{T_{8} Graphs}
 	\end{table}
	\end{center}
	
		Let $E / \Q$ be an elliptic curve such that the isogeny-torsion graph associated to the $\Q$-isogeny class of $E$ is of $T_{8}$ type. By Theorem \ref{thm-kenku}, $E$ does not contain a $\Q$-rational subgroup of odd prime order. There are six possibilities for the torsion configuration of the isogeny-torsion graph associated to the $\Q$-isogeny class of $E$ (see Table \ref{T_{8} Graphs}).
		
		\begin{itemize}
		    \item $\mathcal{T}_{8}^{1}$
		
		It is clear to see that the isogeny-torsion graph associated to a $\Q$-isogeny class of elliptic curves over $\Q$ is $\mathcal{T}_{8}^{1}$ if and only if there is an elliptic curve over $\Q$ in the $\Q$-isogeny class with a torsion subgroup of order $16$ (see Table \ref{T_{8} Graphs}). Let $E / \Q$ be an elliptic curve such that $E(\Q)_{\text{tors}} \cong \Z / 2 \Z \times \Z / 8 \Z$. The image of the mod-$8$ Galois representation attached to $E$ is conjugate to a subgroup of
		$$H_{193n} = \left\langle \left(\begin{array}{cc}
		    3 & 0 \\
		    0 & 1
		\end{array}\right), \left(\begin{array}{cc}
		    5 & 0 \\
		    0 & 1
		\end{array}\right), \left(\begin{array}{cc}
		    1 & 0 \\
		    2 & 1
		\end{array}\right)  \right\rangle.$$
		Elliptic curves over $\Q$ such that the transpose of the image of the mod-$8$ Galois representation is conjugate to a subgroup of $H_{193n}$ are parametrized by the modular curve that appears in the list compiled in \cite{Rouse} with label $\operatorname{X}_{193n}$ which is a genus $0$ curve with infinitely many non-cuspidal, $\Q$-rational points. Thus, $\mathcal{T}_{8}^{1}$ corresponds to an infinite set of \textit{j}-invariants.
		
		\item $\mathcal{T}_{8}^{2}$
		
		Let $E / \Q$ be an elliptic curve such that the isogeny-torsion graph associated to the $\Q$-isogeny class of $E$ is $\mathcal{T}_{8}^{2}$. We may assume without loss of generality that $E(\Q)_{\text{tors}} \cong \Z / 2 \Z \times \Z / 4 \Z$ and $E / \left\langle P_{2} \right\rangle(\Q)_{\text{tors}} \cong \Z / 8 \Z$.
		
		Let $\phi \colon E \to E / \left\langle P_{2} \right\rangle$ be an isogeny with kernel $\left\langle P_{2} \right\rangle$. Clearly, we have $\phi(Q_{4})$ is a point of order $4$ defined over $\Q$ by Lemma \ref{lem-necessity-for-point-rationality}. The cyclic subgroups of $E / \left\langle P_{2} \right\rangle$ of order $8$ that contain $\phi(Q_{4})$ are $\left\langle \phi(Q_{8}) \right\rangle$ and $\left\langle \phi(P_{4} + Q_{8}) \right\rangle$. The group $\left\langle Q_{8} \right\rangle$ is $\Q$-rational. If $\phi(Q_{8})$ is defined over $\Q$, then $\sigma(Q_{8}) - Q_{8} \in \left\langle P_{2} \right\rangle$ for all $\sigma \in G_{\Q}$ by Lemma \ref{lem-necessity-for-point-rationality}. This would force $Q_{8}$ to be defined over $\Q$, a contradiction. Hence, $\phi(P_{4} + Q_{8})$ is defined over $\Q$ and so, $\sigma(P_{4} + Q_{8}) - (P_{4} + Q_{8}) \in \left\langle P_{2} \right\rangle$ for all $\sigma \in G_{\Q}$ by Lemma \ref{lem-necessity-for-point-rationality}. In this case, the image of the mod-$8$ Galois representation attached to $E$ is conjugate to a subgroup of
		$$H_{194l} = \left\langle \left(\begin{array}{cc}
		    3 & 0 \\
		    0 & 1
		\end{array}\right), \left(\begin{array}{cc}
		    5 & 0 \\
		    0 & 1
		\end{array}\right), \left(\begin{array}{cc}
		    1 & 0 \\
		    4 & 1
		\end{array}\right), \left(\begin{array}{cc}
		    1 & 0 \\
		    2 & 5
		\end{array}\right) \right\rangle.$$
		Elliptic curves over $\Q$ such that the transpose of the image of the mod-$8$ Galois representation is conjugate to a subgroup of $H_{194l}$ correspond to non-cuspidal, $\Q$-rational points on the modular curve that appears in the list compiled in \cite{Rouse} with label $\operatorname{X}_{194l}$ which is a genus $0$ curve with infinitely many non-cuspidal, $\Q$-rational points. By Proposition \ref{HIT proposition}, there are infinitely many \textit{j}-invariants corresponding to elliptic curves over $\Q$ such that the image of the mod-$8$ Galois representation is conjugate to $H_{194l}$ itself (not a proper subgroup of $H_{194l}$). The image of the mod-$8$ Galois representation being conjugate to $H_{194l}$ guarantees that the isogeny-torsion graph attached to the $\Q$-isogeny class of $E$ is $\mathcal{T}_{8}^{2}$. Thus, $\mathcal{T}_{8}^{2}$ corresponds to an infinite set of \textit{j}-invariants.
		
		\item $\mathcal{T}_{8}^{3}$
		
		Let $E / \Q$ be an elliptic curve such that the isogeny-torsion graph associated to the $\Q$-isogeny class of $E$ is $\mathcal{T}_{8}^{3}$. Then we may assume without loss of generality that $E(\Q)_{\text{tors}} \cong \Z / 2 \Z \times \Z / 4 \Z$ and that $E / \left\langle P_{2}+Q_{4} \right\rangle(\Q)_{\text{tors}} \cong \Z / 8 \Z$. Let $\phi \colon E \to E / \left\langle P_{2} + Q_{4} \right\rangle$ be an isogeny with kernel $\left\langle P_{2} + Q_{4} \right\rangle$. By Lemma \ref{lem-necessity-for-point-rationality}, $\phi(Q_{8})$ is a point of order $4$ defined over $\Q$. The cyclic subgroups of $E / \left\langle P_{2} + Q_{4} \right\rangle$ of order $8$ that contain $\phi(Q_{8})$ are $\left\langle \phi(Q_{16}) \right\rangle$ and $\left\langle \phi(P_{4} + Q_{16}) \right\rangle$. Let us say $\phi(Q_{16})$ is defined over $\Q$, so $\sigma(Q_{16}) - Q_{16} \in \left\langle P_{2} + Q_{4} \right\rangle$ for all $\sigma \in G_{\Q}$ by Lemma \ref{lem-necessity-for-point-rationality}. Hence, the image of the mod-$16$ Galois representation attached to $E$ is conjugate to a subgroup of
		$$H_{215c} = \left\langle \left(\begin{array}{cc}
		    3 & 0 \\
		    0 & 1
		\end{array}\right), \left(\begin{array}{cc}
		    7 & 0 \\
		    0 & 1
		\end{array}\right), \left(\begin{array}{cc}
		    1 & 0 \\
		    2 & 1
		\end{array}\right), \left(\begin{array}{cc}
		    1 & 8 \\
		    0 & 5
		\end{array}\right) \right\rangle.$$
		Elliptic curves over $\Q$ such that the transpose of the image of the mod-$16$ Galois representation is conjugate to a subgroup of $H_{215c}$ correspond to non-cuspidal, $\Q$-rational points on the modular curve that appears in the list compiled in \cite{Rouse} with label $\operatorname{X}_{215c}$ which is a genus $0$ curve with infinitely many non-cuspidal, $\Q$-rational points. By Proposition \ref{HIT proposition}, there are infinitely many \textit{j}-invariants corresponding to elliptic curves over $\QQ$ with image of the mod-$16$ Galois representation conjugate to $H_{215c}$ itself (not a proper subgroup of $H_{215c}$). The image of the mod-$16$ Galois representation attached to $E$ being conjugate to $H_{215c}$ guarantees the isogeny-torsion graph associated to the $\QQ$-isogeny class of $E$ is $\mathcal{T}_{8}^{3}$. Thus, $\mathcal{T}_{8}^{3}$ corresponds to an infinite set of \textit{j}-invariants.
        
        \item $\mathcal{T}_{8}^{4}$
        
        Let $E / \QQ$ be an elliptic curve such that the image of the mod-$16$ Galois representation attached to $E$ is conjugate to
        $$H_{215l} = \left\langle \left(\begin{array}{cc}
            13 & 0 \\
            0 & 15
        \end{array}\right), \left(\begin{array}{cc}
            9 & 0 \\
            0 & 15
        \end{array}\right), \left(\begin{array}{cc}
            1 & 0 \\
            2 & 1
        \end{array}\right), \left(\begin{array}{cc}
            1 & 8 \\
            0 & 5
        \end{array}\right) \right\rangle.$$
        Note that $H_{4}$ is a quadratic twist of $H_{215c}$ (multiply the first two generators of $H_{215c}$ by $\operatorname{-Id}$). We claim that the isogeny-torsion graph attached to the $\QQ$-isogeny class of $E$ is $\mathcal{T}_{8}^{4}$. Let $\sigma_{i}$ denote the $i$-th generator of $H_{215l}$ for $1 \leq i \leq 4$. Reducing the image of the mod-$16$ Galois representation attached to $E$ by $4$, we get
        $$\overline{H_{215l}}= \left\langle \left(\begin{array}{cc}
            1 & 0 \\
            0 & 3
        \end{array}\right), \left(\begin{array}{cc}
            1 & 0 \\
            2 & 1
        \end{array}\right)\right\rangle$$
        and hence, $\sigma(P_{4}) - P_{4} \in \left\langle Q_{2} \right\rangle$ for all $\sigma \in G_{\Q}$. Let $\phi_{1} \colon E \to E / \left\langle Q_{8} \right\rangle$ be an isogeny with kernel generated by $Q_{8}$ and let $\phi_{2} \colon E \to E / \left\langle P_{2} + Q_{8} \right\rangle$ be an isogeny with kernel generated by $P_{2} + Q_{8}$. Let $\psi \colon E \to E / \left\langle P_{2} + Q_{4} \right\rangle$ be an isogeny with kernel generated by $P_{2} + Q_{4}$. Note that $E / \left\langle Q_{8} \right\rangle$, $E / \left\langle P_{2} + Q_{8} \right\rangle$, and $E / \left\langle P_{2} + Q_{4} \right\rangle$, are elliptic curves over $\Q$ with cyclic torsion subgroups of even order by Lemma \ref{lem-2torspt-all-have-2torspt} and Lemma \ref{lem-Maximality-Of-Rational-2-Power-Groups}. If we can prove that their respective torsion subgroups are all cyclic of order $4$, then we would prove that the isogeny-torsion graph attached to the $\Q$-isogeny class of $E$ is $\mathcal{T}_{8}^{4}$ (see Table \ref{T_{8} Graphs}).
        
        Let us prove that $E / \left\langle Q_{8} \right\rangle(\Q)_{\text{tors}} \cong \Z / 4 \Z$. As $\sigma(P_{4}) - P_{4} \in \left\langle Q_{2} \right\rangle$ for all $\sigma \in G_{\Q}$, $\phi_{1}(P_{4})$ is a point of order $4$ defined over $\Q$ by Lemma \ref{lem-necessity-for-point-rationality}. The point $\phi_{1}(P_{4})$ lives in the two cyclic groups of order $8$, $\left\langle \phi_{1}(P_{8}) \right\rangle$ and $\left\langle \phi_{1}(P_{8} + Q_{16}) \right\rangle$. Note that $\sigma_{1}(P_{8}) - P_{8} = [13]P_{8} - P_{8} = [12]P_{8} \notin \left\langle Q_{8} \right\rangle$. Hence, $\phi_{1}(P_{8})$ is not defined over $\Q$ by Lemma \ref{lem-necessity-for-point-rationality}. Note that $\sigma_{4}(P_{8} + Q_{16}) - (P_{8} + Q_{16}) = P_{8} + P_{2} + [5]Q_{16} - (P_{8} + Q_{16}) = P_{2} + [4]Q_{16} \notin \left\langle Q_{8} \right\rangle$. Thus, $\phi_{1}(P_{8} + Q_{16})$ is not defined over $\Q$ by Lemma \ref{lem-necessity-for-point-rationality}. Hence, $E / \left\langle Q_{8} \right\rangle(\Q)_{\text{tors}} \cong \Z / 4 \Z$. A similar computation shows that $E / \left\langle P_{2} + Q_{8} \right\rangle(\Q)_{\text{tors}} \cong \Z / 4 \Z$.
        
        Now we will prove that $E / \left\langle P_{2} + Q_{4} \right\rangle(\Q)_{\text{tors}} \cong \Z / 4 \Z$. Note that again, $\psi(P_{4})$ is a point of order $4$ defined over $\Q$ by Lemma \ref{lem-necessity-for-point-rationality}. The point $\psi(P_{4})$ lives in two cyclic groups of order $8$, $\left\langle \psi(P_{8}) \right\rangle$ and $\left\langle \psi(P_{8} + Q_{8}) \right\rangle$. As $\sigma_{1}(P_{8}) - P_{8} = [13]P_{8} - P_{8} = [12]P_{8} \notin \left\langle P_{2} + Q_{4} \right\rangle$, we have that $\phi(P_{8})$ is not defined over $\Q$ by Lemma \ref{lem-necessity-for-point-rationality}. As $\sigma_{2}(P_{8} + Q_{8}) - (P_{8} + Q_{8}) = [9]P_{8} + [15]Q_{8} - (P_{8} + Q_{8}) = [14]Q_{8} = [3]Q_{4} \notin \left\langle P_{2} + Q_{4} \right\rangle$, we have that $\psi(P_{8} + Q_{8})$ is not defined over $\Q$ by Lemma \ref{lem-necessity-for-point-rationality}. This is enough to prove that the isogeny-torsion graph attached to the $\Q$-isogeny class of $E$ is $\mathcal{T}_{8}^{4}$ (see Table \ref{T_{8} Graphs}).
        
        The set of elliptic curves over $\Q$ such that the image of the mod-$16$ Galois representation is conjugate to a subgroup of the transpose of $H_{215l}$ corresponds to the non-cuspidal, $\Q$-rational points on the modular curve found in \cite{Rouse} with label $\operatorname{X}_{215l}$. This modular curve is a genus $0$ curve with infinitely many non-cuspidal, $\Q$-rational points. By Proposition \ref{HIT proposition}, there are infinitely many \textit{j}-invariants corresponding to elliptic curves over $\Q$ such that the image of the mod-$16$ Galois representation is conjugate to $H_{215l}$ (not a proper subgroup of $H_{215l}$). The image of the mod-$16$ Galois representation being conjugate to $H_{215l}$ guarantees that the isogeny-torsion graph associated to the $\Q$-isogeny class of $E$ is $\mathcal{T}_{8}^{4}$. Thus, $\mathcal{T}_{8}^{4}$ corresponds to an infinite set of \textit{j}-invariants.
        
		\item $\mathcal{T}_{8}^{5}$
		
		Let $E / \Q$ be an elliptic curve such that the image of the mod-$16$ Galois representation attached to $E$ is conjugate to
		$$H_{215k} = \left\langle \left(\begin{array}{cc}
		    13 & 0 \\
		    0 & 1
		\end{array}\right), \left(\begin{array}{cc}
		    1 & 0 \\
		    2 & 1
		\end{array}\right), \left(\begin{array}{cc}
		    1 & 0 \\
		    0 & 15
		\end{array}\right), \left(\begin{array}{cc}
		    15 & 8 \\
		    0 & 11
		\end{array}\right) \right\rangle.$$
		The group $H_{215k}$ is a quadratic twist of $H_{215c}$ (multiply the first, second, and fourth generators of $H_{215c}$ by $\operatorname{-Id}$). We claim that the isogeny-torsion graph associated to the $\Q$-isogeny class of $E$ is $\mathcal{T}_{8}^{5}$.
		
		Let $\phi \colon E \to E / \left\langle Q_{8} \right\rangle$ be an isogeny with kernel $\left\langle Q_{8} \right\rangle$. The group $E / \left\langle Q_{8} \right\rangle(\Q)_{\text{tors}}$ is cyclic of even order by Lemma \ref{lem-Maximality-Of-Rational-2-Power-Groups} and Lemma \ref{lem-2torspt-all-have-2torspt}. The point of $E / \left\langle Q_{8} \right\rangle$ of order $2$ defined over $\Q$ is $\phi(P_{2})$ by Lemma \ref{lem-necessity-for-point-rationality}. The point $\phi(P_{2})$ is contained in two cyclic groups of order $4$, $\left\langle \phi(P_{4}) \right\rangle$ and $\left\langle \phi(P_{4} + Q_{16}) \right\rangle$. The point $\phi(P_{4} + Q_{16})$ is defined over $\Q$ if and only if $\sigma(P_{4} + Q_{16}) - (P_{4} + Q_{16}) \in \left\langle Q_{8} \right\rangle$ for all $\sigma \in G_{\Q}$ by Lemma \ref{lem-necessity-for-point-rationality}. In other words, for an arbitrary $\sigma \in G_{\Q}$
		
		\begin{center} $\sigma(P_{4}) = P_{4}$ or $P_{4} + Q_{2} \iff \sigma(Q_{16}) = [2b+1]Q_{16}$ for some $b \in \Z$ \end{center}
		\begin{center} $\sigma(P_{4}) = [3]P_{4}$ or $[3]P_{4} + Q_{2} \iff \sigma(Q_{16}) = P_{2} + [2b+1]Q_{16}$ for some $b \in \Z$. \end{center}
		Lifting $P_{4}$ to level $16$, we have satisfy the following algebraic relations:
		\begin{center} $\sigma(P_{16}) = [1+4a]P_{16} + [2c]Q_{16} \iff \sigma(Q_{16}) = [2b+1]Q_{16}$ for some $a,b,c \in \Z$ \end{center}
		\begin{center} $\sigma(P_{16}) = [3+4a]P_{16} + [2c]Q_{16} \iff \sigma(Q_{16}) = P_{2} + [2b+1]Q_{16}$ for some $a,b,c \in \Z.$ \end{center}
		That is how the elements of $H_{215k}$ behave! Hence, $\phi(P_{4} + Q_{16})$ is defined over $\Q$, making $E / \left\langle Q_{8} \right\rangle(\Q)_{\text{tors}}$ a cyclic group of order $4$ or $8$.
		
		Now let $\phi' \colon E \to E / \left\langle P_{2} + Q_{8} \right\rangle$ be an isogeny with kernel $\left\langle P_{2} + Q_{8} \right\rangle$. Then the group $E / \left\langle P_{2} + Q_{8} \right\rangle(\Q)_{\text{tors}}$ is a cyclic group of even order by Lemma \ref{lem-Maximality-Of-Rational-2-Power-Groups} and Lemma \ref{lem-2torspt-all-have-2torspt}. The point of $E / \left\langle P_{2} + Q_{8} \right\rangle$ of order $2$ defined over $\Q$ is $\phi'(P_{2})$ by Lemma \ref{lem-necessity-for-point-rationality}. The point $\phi'(P_{2})$ lives in two cyclic groups of order $4$, $\left\langle \phi'(P_{4}) \right\rangle$ and $\left\langle \phi'(Q_{16}) \right\rangle$. We have that $\phi'(P_{4})$ is defined over $\Q$ if and only if $\sigma(P_{4}) - P_{4} \in \left\langle P_{2} + Q_{8} \right\rangle$ for all $\sigma \in G_{\Q}$ by Lemma \ref{lem-necessity-for-point-rationality}. Let $\tau \in G_{\Q}$ such that $\tau(P_{16}) = [15]P_{16}$. Then $\tau(P_{4}) = [15]P_{4} = [3]P_{4}$ and hence, $\tau(P_{4}) - P_{4} = [2]P_{4} = P_{2}$. As $P_{2} \notin \left\langle P_{2} + Q_{8} \right\rangle$, $\phi'(P_{4})$ is not defined over $\Q$ by Lemma \ref{lem-necessity-for-point-rationality}. Similarly, we have that $\phi'(Q_{16})$ is defined over $\Q$ if and only if $\sigma(Q_{16}) - Q_{16} \in \left\langle P_{2} + Q_{8} \right\rangle$ for all $\sigma \in G_{\Q}$ by Lemma \ref{lem-necessity-for-point-rationality}. Let $\tau \in G_{\Q}$ such that $\tau(Q_{16}) = [15]Q_{16}$. Then $\tau(Q_{16}) - Q_{16} = [14]Q_{16} = [7]Q_{8}$. As $[7]Q_{8} \notin \left\langle P_{2} + Q_{8} \right\rangle$, we have that $\phi'(Q_{16})$ is not defined over $\Q$ by Lemma \ref{lem-necessity-for-point-rationality}. This is enough to prove that $E / \left\langle P_{2} + Q_{8} \right\rangle(\Q)_{\text{tors}}$ is cyclic of order $2$. And thus, we may conclude that the isogeny-torsion graph associated to the $\Q$-isogeny class of $E$ is $\mathcal{T}_{8}^{5}$ (see Table \ref{T_{8} Graphs}).
		
		Elliptic curves over $\Q$ such that the transpose of the image of the mod-$16$ Galois representation is conjugate to a subgroup of $H_{215k}$ correspond to non-cuspidal, $\Q$-rational points on the modular curve found in \cite{Rouse} with label $\operatorname{X}_{215k}$ which is a genus $0$ curve with infinitely many non-cuspidal, $\Q$-rational points. By Proposition \ref{HIT proposition}, there are infinitely many \textit{j}-invariants corresponding to elliptic curves over $\Q$ such that the image of the mod-$16$ Galois representation is conjugate to $H_{215k}$ itself (not a proper subgroup of $H_{215k}$). The image of the mod-$16$ Galois representation being conjugate to $H_{215k}$ guarantees that the isogeny-torsion graph attached to the $\QQ$-isogeny class of $E$ is $\mathcal{T}_{8}^{5}$. Thus, $\mathcal{T}_{8}^{5}$ corresponds to an infinite set of \textit{j}-invariants.
		
		\item $\mathcal{T}_{8}^{6}$
		
		Let $E / \Q$ be an elliptic curve such that the image of the mod-$16$ Galois representation attached to $E$ is conjugate to $H_{215} = \left\langle H_{215k}, \operatorname{-Id} \right\rangle$. Then by Lemma \ref{No Points of Order 4}, there are no elliptic curves over $\Q$ in the $\Q$-isogeny class of $E$ that have a point of order $4$ defined over $\Q$. Hence, the isogeny-torsion graph of $E$ is $\mathcal{T}_{8}^{6}$.
		
		Elliptic curves over $\Q$ such that the image of the mod-$16$ Galois representation is conjugate to a subgroup of $H_{215}$ correspond to non-cuspidal, $\Q$-rational points on the modular curve defined by the group that appears in the list compiled in \cite{SZ} with label $16G^{0}-16b$ which is a genus $0$ curve with infinitely many non-cuspidal, $\Q$-rational points. This is the same modular curve that appears in the list \cite{Rouse} with label $\operatorname{X}_{215}$. By Proposition \ref{HIT proposition}, there are infinitely many \textit{j}-invariants corresponding to elliptic curves over $\Q$ such that the image of the mod-$16$ Galois representation is conjugate to $H_{215}$ itself (not a proper subgroup of $H_{215}$). The image of the mod-$16$ Galois representation attached to $E$ being conjugate to $H_{215}$ guarantees that the isogeny-torsion graph associated to the $\QQ$-isogeny class of $E$ is $\mathcal{T}_{8}^{6}$. Thus, $\mathcal{T}_{8}^{6}$ corresponds to an infinite set of \textit{j}-invariants. 
		\end{itemize}
		
		This concludes the proof of Proposition \ref{T8 Graphs Proposition}. Also, this concludes the proof that \textit{all} isogeny-torsion graphs containing an elliptic curve over $\QQ$ with full two-torsion defined over $\Q$ correspond to infinite sets of \textit{j}-invariants.
		
    \section{Isogeny Graphs of $R_{k}$ Type}

		\subsection{Isogeny-Torsion Graphs of $R_{6}$ Type}
		
		In this subsection, we prove that the two isogeny-torsion graphs of $R_{6}$ type correspond to infinite sets of \textit{j}-invariants.
		
		\begin{proposition}\label{R6 Graphs Proposition}
		Let $\mathcal{G}$ be an isogeny-torsion graph of $R_{6}$ type (regardless of torsion configuration). Then $\mathcal{G}$ corresponds to an infinite set of \textit{j}-invariants.
		\end{proposition}
		
		Let $E / \Q$ be an elliptic curve such that $E(\Q)_{\text{tors}} \cong \Z / 6 \Z$ and the isogeny-torsion graph associated to the $\Q$-isogeny class of $E$ is $\mathcal{R}_{6}^{1}$. Let $d$ be a non-zero, square-free integer not equal to $1$ or $-3$ and let $E^{(d)}$ be the quadratic twist of $E$ by $d$. Then the isogeny-torsion graph associated to $E^{(d)}$ is $\mathcal{R}_{6}^{2}$. Hence, if we prove that $\mathcal{R}_{6}^{1}$ corresponds to an infinite set of \textit{j}-invariants, we just have to take quadratic twists to show that $\mathcal{R}_{6}^{2}$ corresponds to an infinite set of \textit{j}-invariants.

		\begin{center}
        \begin{table}[h!]
 	\renewcommand{\arraystretch}{1.3}
 	\scalebox{1.1}{
 	\begin{tabular}{ |c|c|c|c| }
 		\hline
 		Graph Type & Type & Isomorphism Types & Label \\
 		\hline
 		
 		\multirow{2}*{\includegraphics[scale=0.16]{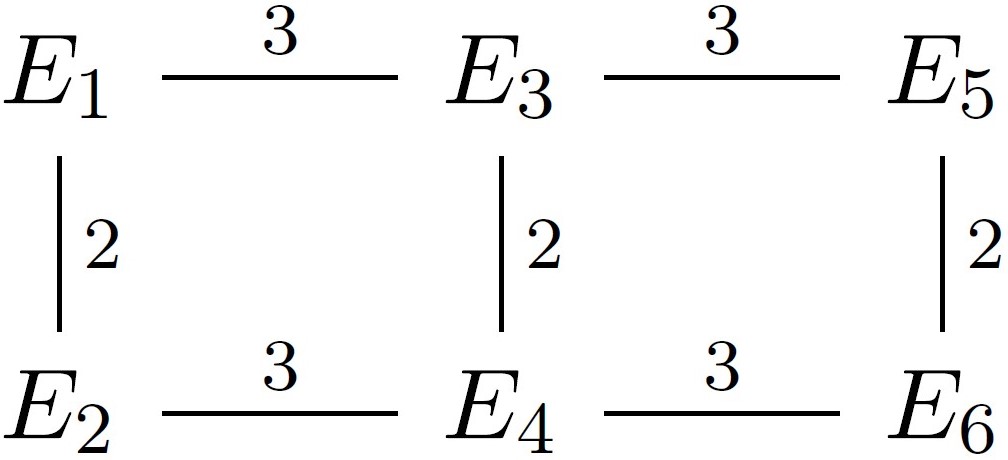}} & \multirow{2}*{$R_{6}$} & $([6],[6],[6],[6],[2],[2])$ & $\mathcal{R}_{6}^{1}$ \\
 		\cline{3-4}
 		& & $([2],[2],[2],[2],[2],[2])$ & $\mathcal{R}_{6}^{2}$ \\
 		\hline
 	\end{tabular}}
 	\caption{Isogeny Graphs of $R_{6}$ Type}
        \label{R_{6} Graphs}
 \end{table}
 		\end{center}
		
		Let $E' / \Q$ be an elliptic curve. Note that the isogeny-torsion graph associated to the $\Q$-isogeny class of $E'$ is of $R_{6}$ type if and only if $E'$ is $\Q$-isogenous to an elliptic curve $E$ over $\Q$ such that $E$ has a point of order $2$ defined over $\Q$ and has two $\QQ$-rational subgroups of order $3$.
		
		\begin{itemize}
		    \item $\mathcal{R}_{6}^{1}$
		    
		    Let $E / \Q$ be an elliptic curve. Then the isogeny-torsion graph associated to the $\Q$-isogeny class of $E$ is $\mathcal{R}_{6}^{1}$ if and only if (up to relabeling) $E$ has a point of order $6$ defined over $\Q$ and two $\Q$-rational subgroups of order $3$. In this case, the image of the mod-$2$ Galois representation attached to $E$ is conjugate to $\mathcal{B}_{2}$, the subgroup of $\operatorname{GL}(2, \Z / 2 \Z)$ consisting of upper-triangular matrices. By Lemma \ref{Split Cartan}, the image of the mod-$3$ Galois representation attached to $E$ is conjugate to $D_{3} = \left\langle \left(\begin{array}{cc}
		    1 & 0 \\
		    0 & 2
		\end{array}\right) \right\rangle$. Suppose the image of the mod-$12$ Galois representation is conjugate to $\mathcal{B}_{2} \times D_{3} \cong \Z / 2 \Z \times \Z / 2 \Z$.
		
		Work in \cite{enriquealvaro} has parametrized elliptic curves over $\Q$ such that the image of the mod-$6$ Galois representation is isomorphic to $\Z / 2 \Z \times \Z / 2 \Z$. Let $E_{t} : y^{2} = x^{3} + A(t)x + B(t)$ where
		\begin{center} $A(t) = -27t^{12}+216t^{9}-6480t^{6}+12528t^{3}-432$ and $B(t) = 54t^{18} - 648t^{15} - 25920t^{12} + 166320t^{9} - 651888t^{6} + 222912t^{3} + 3456$. \end{center}
		
		Using the infinite one-parameter family $E_{t}$, we see that $\mathcal{R}_{6}^{1}$ corresponds to an infinite set of \textit{j}-invariants.
		
		\item $\mathcal{R}_{6}^{2}$
		
		Let $E / \Q$ be an elliptic curve such that $E(\Q)_{\text{tors}} \cong \Z / 6 \Z$ and $E$ has two independent $\Q$-rational subgroups of order $3$. Then the isogeny-torsion graph associated to the $\Q$-isogeny class of $E$ is $\mathcal{R}_{6}^{1}$. Moreover, $\QQ(E[3]) = \Q(\sqrt{-3})$. Let $E' / \QQ$ be the quadratic twist of $E$ by a square-free integer not equal to $1$ or $-3$. The image of the mod-$2$ Galois representation attached to $E'$ is the same as the image of the mod-$2$ Galois representation attached to $E$ and the image of the mod-$3$ Galois representation attached to $E'$ is conjugate to $\left\{ \left(\begin{array}{cc}
		    \ast & 0 \\
		    0 & \ast
		\end{array}\right) \right\}$. By lemma \ref{Quadratic Twisting Odd Graphs}, no elliptic curve over $\Q$ that is $\Q$-isogenous to $E'$ has a point of order $3$ defined over $\Q$. Thus, the isogeny-torsion graph associated to the $\Q$-isogeny class of $E'$ is $\mathcal{R}_{6}^{2}$. As $\mathcal{R}_{6}^{1}$ corresponds to an infinite set of \textit{j}-invariants, $\mathcal{R}_{6}^{2}$ also corresponds to an infinite set of \textit{j}-invariants.
		\end{itemize}
		
		This concludes the proof of Proposition \ref{R6 Graphs Proposition}.
	    
		\subsection{Isogeny-Torsion Graphs of $R_{4}$ Type}
		
		In this subsection, we prove that each of the isogeny-torsion graphs of $R_{4}(6)$ type and $R_{4}(10)$ type correspond to infinite sets of \textit{j}-invariants.
		
		\begin{proposition}\label{R4 Graphs Proposition}
		Let $\mathcal{G}$ be an isogeny-torsion graph of $R_{4}(10)$ or $R_{4}(6)$ type (regardless of torsion configuration). Then $\mathcal{G}$ corresponds to an infinite set of \textit{j}-invariants.
		\end{proposition}
		
		Proving both of the isogeny-torsion graphs of $R_{4}(10)$ type correspond to infinite sets of \textit{j}-invariants will be relatively easy. We just need to find a one-parameter family of elliptic curves over $\Q$ with a point of order $10$ defined over $\Q$ and then use an appropriate twist. On the other hand, the proof that both of the isogeny-torsion graphs of $R_{4}(6)$ type correspond to infinite sets of \textit{j}-invariants will be relatively messy, relying completely on Hilbert's Irreducibility Theorem.
		
		\begin{table}[h!]
		\renewcommand{\arraystretch}{1.35}
 	\scalebox{1.1}{
 	\begin{tabular}{ |c|c|c|c| }
 		\hline
 		Graph Type & Type & Isomorphism Types & Label \\
 		
 		\hline
 		\multirow{4}*{\includegraphics[scale=0.3]{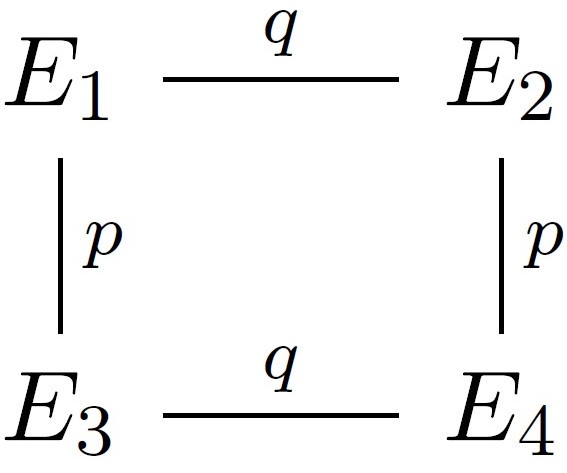}} & \multirow{2}*{$R_{4}(10)$} & $([10],[10],[2],[2])$ & $\mathcal{R}_{4}^{1}(10)$ \\
 		\cline{3-4}
 		& & $([2],[2],[2],[2])$ & $\mathcal{R}_{4}^{2}(10)$ \\
 		\cline{2-4}
 		& \multirow{2}*{$R_{4}(6)$} & $([6],[6],[2],[2])$ & $\mathcal{R}_{4}^{1}(6)$ \\
 		\cline{3-4}
 		& & $([2],[2],[2],[2])$ & $\mathcal{R}_{4}^{2}(6)$ \\
 		\hline
 		\end{tabular}}
 		\caption{Isogeny Graphs of $R_{4}$ Type}
        \label{R_{4} Graphs}
 		\end{table}

		\begin{itemize}
		
		    \item Isogeny-torsion Graphs of $R_{4}(10)$ Type
		    
		    \begin{enumerate}
		        \item $\mathcal{R}_{4}^{1}(10)$
		    
		    The isogeny-torsion graph of a $\Q$-isogeny class of elliptic curves over $\Q$ is $\mathcal{R}_{4}^{1}(10)$ if and only there is an elliptic curve over $\Q$ in the $\Q$-isogeny class with rational $10$-torsion. Let $t \in \QQ$ and let $E_{t} : y^{2} + (1-a)xy - by = x^{3}-bx^{2}$ with
		    \begin{center}
		        $a = \frac{t(t-1)(2t-1)}{t^{2}-3t+1}$ and $\frac{t^{3}(t-1)(2t-1)}{(t^{2}-3t+1)^{2}}$.
		    \end{center}
		    If $E_{t}$ is a smooth elliptic curve, then $E_{t}(\Q)_{\text{tors}} \cong \Z / 10 \Z$ (see Appendix E of \cite{alrbook}). Note that in this case, the image of the mod-$2$ Galois representation attached to $E_{t}$ is conjugate to $\left\langle \left(\begin{array}{cc}
		        1 & 1 \\
		        0 & 1
		    \end{array}\right) \right\rangle$. By Lemma \ref{Rational Points}, the image of the mod-$5$ Galois representation attached to $E_{t}$ is conjugate to $\mathcal{B}_{5}$, the subgroup of $\operatorname{GL}(2, \Z / 5 \Z)$ consisting of matrices of the form $\left(\begin{array}{cc}
		        1 & x \\
		        0 & y
		    \end{array}\right)$. Using the one-parameter family of elliptic curves $E_{t}$, we can conclude that $\mathcal{R}_{4}^{1}(10)$ corresponds to an infinite set of \textit{j}-invariants.
		    
		    \item $\mathcal{R}_{4}^{2}(10)$
		    
		    By Corollary \ref{Corollary Unique Index 2 Subgroup}, the only quadratic subfield of $\QQ(E_{t}[5])$ is $\QQ(\sqrt{5})$. Let $d$ be a non-zero, square-free integer not equal to $1$ or $5$ and let $E^{(d)}_{t} / \Q$ be the quadratic twist of $E_{t}$ by $d$. Then the image of the mod-$5$ Galois representation attached to $E^{(d)}_{t}$ is conjugate to $\left\langle \mathcal{B}_{5}, \operatorname{-Id} \right\rangle$. By Lemma \ref{Quadratic Twisting Odd Graphs}, none of the elliptic curves that are $\Q$-isogenous to $E^{(d)}_{t}$ have a point of order $5$ defined over $\Q$. The elliptic curve $E^{(d)}_{t}$ still has a point of order $2$ defined over $\Q$ and a $\Q$-rational subgroup of order $5$. The isogeny-torsion graph associated to the $\Q$-isogeny class of $E^{(d)}_{t}$ is $\mathcal{R}_{4}^{2}(10)$. As $\mathcal{R}_{4}^{1}(10)$ corresponds to an infinite set of \textit{j}-invariants, $\mathcal{R}_{4}^{2}(10)$ also corresponds to an infinite set of \textit{j}-invariants.
		    \end{enumerate}
		    
		    \item Isogeny-Torsion Graphs of $R_{4}(6)$ Type
		    
		    Let $E / \Q$ be an elliptic curve such that the isogeny-torsion graph associated to the $\Q$-isogeny class of $E$ is of $R_{4}(6)$ type. There are two possibilities for the isogeny-torsion graph, namely, $\mathcal{R}_{4}^{1}(6)$ and $\mathcal{R}_{4}^{2}(6)$. In both cases, all the elliptic curves over $\Q$ in the $\Q$-isogeny class have a point of order $2$ defined over $\Q$. In the case that an elliptic curve over $\Q$ in the $\Q$-isogeny class has a point of order $6$ defined over $\Q$, we may assume without loss of generality that $E(\Q)_{\text{tors}} \cong \Z / 6 \Z$. In this case, by Lemma \ref{Rational Points}, the image of the mod-$3$ Galois representation attached to $E$ is conjugate to $\mathfrak{B}_{3}$; the subgroup of $\operatorname{GL}(2, \Z / 3 \Z)$ consisting of matrices of the form $\left(\begin{array}{cc}
		        1 & x \\
		        0 & y
		    \end{array}\right)$. In the case that no elliptic curve over $\Q$ in the $\Q$-isogeny class has a point of order $6$ defined over $\Q$, the image of the mod-$3$ Galois representation of every elliptic curve over $\Q$ in the $\Q$-isogeny class is conjugate to $\left\langle \mathfrak{B}_{3}, \operatorname{-Id} \right\rangle$. In either case of torsion configuration, the image of the mod-$2$ Galois representation attached to the elliptic curves over $\Q$ in the $\Q$-isogeny class is conjugate to $\mathfrak{B}_{2}$; the subgroup of $\operatorname{GL}(2, \Z / 2 \Z)$ consisting of upper-triangular matrices.
		    
		    Isogeny-torsion graphs of $R_{4}(6)$ type are proper subgraphs of isogeny-torsion graphs of $S$ type and isogeny-torsion graphs of $R_{6}$ type. To prove there are infinitely many \textit{j}-invariants corresponding to isogeny-torsion graphs of $R_{4}(6)$ type, it suffices to prove that there are infinitely many \textit{j}-invariants corresponding to elliptic curves over $\Q$ with a $\Q$-rational subgroup of order $6$ and
		    \begin{enumerate}
		        \item without full two-torsion (to avoid an isogeny-torsion graph of type $S$),
		        
		        \item without a cyclic $\Q$-rational subgroup of order $12$ (to avoid an isogeny-torsion graph of type $S$),
		        
		        \item without a cyclic $\Q$-rational subgroup of order $9$ (to avoid an isogeny-torsion graph of type $R_{6}$),
		        
		        \item without two distinct $\Q$-rational subgroups of order $3$ (to avoid an isogeny-torsion graph of type $R_{6}$).
		    \end{enumerate}
		    Let $G$ be the full inverse image of the subgroup $\mathfrak{B}_{2} \times \left\langle \mathfrak{B}_{3}, \operatorname{-Id} \right\rangle$ of $\operatorname{GL}(2, \Z / 2 \Z) \times \operatorname{GL}(2, \Z / 3 \Z)$ via the mod-$6$ reduction map from $\operatorname{GL}(2, \Z / 4 \Z) \times \operatorname{GL}(2, \Z / 9 \Z)$. The group $G$ has order $2^{7} \cdot 3^{5}$ and defines a modular curve of genus $0$ with infinitely many non-cuspidal, $\Q$-rational points. Elliptic curves over $\Q$ with a $\Q$-rational subgroup of order $6$ correspond to non-cuspidal, $\Q$-rational points on the modular curve $\operatorname{X}_{0}(6)$. Note that the modular curve defined by $G$ is precisely the modular curve defined by $\mathfrak{B}_{2} \times \left\langle \mathfrak{B}_{3}, \operatorname{-Id} \right\rangle$, which is precisely the modular curve $\operatorname{X}_{0}(6)$.
		    
		    \begin{enumerate}
		    
		    \item Let $G_{1}$ be the full inverse image of the subgroup $$K_{1} = \{I\} \times \left\langle \mathfrak{B}_{3}, \operatorname{-Id} \right\rangle$$
		    of $\operatorname{GL}(2, \Z / 2 \Z) \times \operatorname{GL}(2, \Z / 3 \Z)$ via the mod-$6$ reduction map from the group $\operatorname{GL}(2, \Z / 4 \Z) \times \operatorname{GL}(2, \Z / 9 \Z)$. Elliptic curves over $\Q$ with full two-torsion defined over $\Q$ and a $\Q$-rational subgroup of order $3$ correspond to non-cuspidal $\Q$-rational points on the modular curve defined by $K_{1}$ which is precisely the modular curve defined by $G_{1}$. The group $G_{1}$ defines a modular curve of genus $0$ with infinitely many non-cuspidal, $\Q$-rational points and $G_{1}$ is an index-$2$ subgroup of $G$.
		    
		    \item Let $G_{2}$ be the full inverse image of the subgroup
		    $$K_{2} := \left\langle \left(\begin{array}{cc}
		        1 & 0 \\
		        0 & 3
		    \end{array} \right), \left(\begin{array}{cc}
		        1 & 1 \\
		        0 & 1
		    \end{array} \right), \operatorname{-Id} \right\rangle \times \left\langle \mathfrak{B}_{3}, \operatorname{-Id} \right\rangle$$
		    of $\operatorname{GL}(2, \Z / 4 \Z) \times \operatorname{GL}(2, \Z / 3 \Z)$ via the mod-$3$ reduction map from the matrix group $\operatorname{GL}(2, \Z / 4 \Z) \times \operatorname{GL}(2, \Z / 9 \Z)$. Elliptic curves over $\Q$ with a $\Q$-rational subgroup of order $12$ correspond to non-cuspidal, $\Q$-rational points on the modular curve defined by $K_{2}$ which is precisely the modular curve defined by $G_{2}$. The group $G_{2}$ is an index-$2$ subgroup of $G$ and $G_{2}$ defines a modular curve of genus $0$ with infinitely many non-cuspidal, $\Q$-rational points.
		    
		    \item Let $G_{3}$ be the full inverse image of the subgroup
		    $$K_{3} := \mathfrak{B}_{2} \times \left\langle \left(\begin{array}{cc}
		        2 & 0 \\
		        0 & 1
		    \end{array}\right), \left(\begin{array}{cc}
		        1 & 1 \\
		        0 & 1
		    \end{array}\right), \left(\begin{array}{cc}
		        1 & 0 \\
		        0 & 2
		    \end{array}\right) \right\rangle$$
		    of $\operatorname{GL}(2, \Z / 2 \Z) \times \operatorname{GL}(2, \Z / 9 \Z)$ via the mod-$2$ reduction map from the matrix group $\operatorname{GL}(2, \Z / 4 \Z) \times \operatorname{GL}(2, \Z / 9 \Z)$. Elliptic curves over $\Q$ with a $\Q$-rational subgroup of order $18$ correspond to non-cuspidal $\Q$-rational points on the modular curve defined by $K_{3}$ which is precisely the modular curve $\operatorname{X}_{0}(18)$. The group $G_{3}$ is an index-$3$ subgroup of $G$ and $G_{3}$ defines a modular curve of genus $0$ with infinitely many non-cuspidal, $\Q$-rational points.
		    
		    \item Finally, let $G_{4}$ be the full inverse image of the subgroup $\mathfrak{B}_{2} \times \left\langle \left(\begin{array}{cc}
		        1 & 0 \\
		        0 & 2
		    \end{array}\right), \operatorname{-Id} \right\rangle$ of $\operatorname{GL}(2, \Z / 2 \Z) \times \operatorname{GL}(2, \Z / 3 \Z)$ via the mod-$6$ reduction map from the group $\operatorname{GL}(2, \Z / 4 \Z) \times \operatorname{GL}(2, \Z / 9 \Z)$. Elliptic curves over $\Q$ with a point of order $2$ defined over $\Q$ and two $\Q$-rational subgroups of order $3$ correspond to non-cuspidal, $\Q$-rational points on the modular curve defined by $G_{4}$. The group $G_{4}$ is an index-$3$ subgroup of $G$ and $G_{4}$ defines a modular curve of genus $0$ with infinitely many non-cuspidal, $\Q$-rational points.
		    \end{enumerate}
		    
		    Now let
		    $$ S_{G} := \bigcup_{i=1}^{4} \pi_{G_{i}, G}(X_{G_{i}}(\Q))$$
		    where $\pi_{G_{i},G} \colon \operatorname{X}_{G_{i}} \to \operatorname{X}_{G}$ is the natural morphism induced by the inclusion $G_{i} \subseteq G$. The degree of $\pi_{G_{i}, G}$ is equal to $[G:G_{i}] = 2$ or $3$. Note that $\operatorname{X}_{G} \cong \PP^{1}$ and that $S_{G}$ is a \textit{thin} subset in the language of Serre because all the degrees of the maps $\pi_{G_{i},G}$ are at least $2$. The field $\Q$ is Hilbertian and $\PP^{1}_{\Q} \cong \operatorname{X}_{G}(\Q)$ is not thin. This implies that the complement $\operatorname{X}_{G}(\Q) \setminus S_{G}$ is not thin and must be infinite.
		    
		    This proves that there are infinitely many \textit{j}-invariants corresponding to elliptic curves over $\Q$ such that the image of mod-$2$ Galois representation is conjugate to $\mathfrak{B}_{2}$ and the image of the mod-$3$ Galois representation is conjugate to one of three subgroups of $\operatorname{GL}(2, \Z / 3 \Z)$, namely, $\mathcal{B}_{3}$, $\left\langle \left(\begin{array}{cc}
		        1 & 1 \\
		        0 & 1
		    \end{array}\right), \left(\begin{array}{cc}
		        2 & 0 \\
		        0 & 1
		    \end{array}\right) \right\rangle$, or $\left\langle \mathcal{B}_{3}, \operatorname{-Id} \right\rangle$. For the first two groups, the isogeny-torsion graph is $\mathcal{R}_{4}^{1}(6)$. For the third group, the isogeny-torsion graph is $\mathcal{R}_{4}^{2}(6)$. Thus, we have proven that $\mathcal{R}_{4}^{1}(6)$ and $\mathcal{R}_{4}^{2}(6)$ together correspond to infinitely many \textit{j}-invariants. We must prove that both correspond to infinitely many \textit{j}-invariants individually.
		    
		    Let $E / \QQ$ be an elliptic curve such that $E(\Q)_{\text{tors}} \cong \Z / 6 \Z$ and the isogeny-torsion graph associated to the $\Q$-isogeny class of $E$ is $\mathcal{R}_{4}^{1}(6)$. Let $E^{(r)}$ be the quadratic twist of $E$ by a non-zero, square-free integer $r$ not equal to $1$ or $-3$. The only quadratic subfield of $\Q(E[3])$ is $\Q(\sqrt{-3})$ and thus, $\Q(\sqrt{-3})$ does not contain $\sqrt{r}$. The image of the mod-$3$ Galois representation attached to $E^{(r)}$ is conjugate to $\left\langle B_{3}, \operatorname{-Id} \right\rangle$ and thus, the isogeny-torsion graph associated to the $\Q$-isogeny class of $E^{(r)}$ is $\mathcal{R}_{4}^{2}(6)$.
		    
		    On the other hand, let $E / \QQ$ be an elliptic curve such that the isogeny-torsion graph associated to the $\Q$-isogeny class of $E$ is $\mathcal{R}_{4}^{2}(6)$. Then the image of the mod-$3$ Galois representation attached to $E$ is conjugate to $\left\langle \mathcal{B}_{3}, \operatorname{-Id} \right\rangle$. By Lemma \ref{Three Subgroups of Index 2}, there are three subgroups of $\left\langle \mathcal{B}_{3}, \operatorname{-Id} \right\rangle$ of index $2$ and by Corollary \ref{Corollary Three Subgroups of Index 2}, these three subgroups of index $2$ correspond to three subfields of $\QQ(E[3])$ of degree $2$, one totally real and two totally imaginary. The three quadratic subfields of $\QQ(E[3])$ are $\QQ(\sqrt{-3})$, $\QQ(\sqrt{d})$, and $\QQ(\sqrt{-3d})$ for some positive, square-free integer $d$ not equal to $1$. Let $E^{(d)}$ be the quadratic twist of $E$ by $d$ and let $E^{(-3d)}$ be the quadratic twist of $E$ by $-3d$. Then the image of the mod-$3$ Galois representation attached to $E^{(d)}$ or the image of the mod-$3$ Galois representation attached to $E^{(-3d)}$ is conjugate to $\mathcal{B}_{3}$. Hence, either the isogeny-torsion graph associated to the $\Q$-isogeny class of $E^{(d)}$ is $\mathcal{R}_{4}^{1}(6)$ or the isogeny-torsion graph associated to the $\Q$-isogeny class of $E^{(-3d)}$ is $\mathcal{R}_{4}^{1}(6)$.
		    
		    From the fact that one of $\mathcal{R}_{4}^{1}(6)$ or $\mathcal{R}_{4}^{2}(6)$ corresponds to an infinite set of \textit{j}-invariants, and the graphs are interchangeable by quadratic twists, both $\mathcal{R}_{4}^{1}(6)$ and $\mathcal{R}_{4}^{2}(6)$ correspond to infinite sets of \textit{j}-invariants.
		\end{itemize}
		
		This concludes the proof of Proposition \ref{R4 Graphs Proposition}.
		
		\section{Linear Graphs}
		\subsection{Isogeny-Torsion Graphs of $L_{3}$ Type}
		
		In this subsection, we prove that the isogeny-torsion graphs of $L_{3}(25)$ and $L_{3}(9)$ type correspond to infinite sets of \textit{j}-invariants.
		
		\begin{proposition}\label{L3 Graphs Proposition}
		Let $\mathcal{G}$ be an isogeny-torsion graph of $L_{3}(9)$ type or $L_{3}(25)$ (regardless of torsion configuration). Then $\mathcal{G}$ corresponds to an infinite set of \textit{j}-invariants.
		\end{proposition}
		
		To prove that the isogeny-torsion graphs of $L_{3}(25)$ type correspond to infinite sets of \textit{j}-invariants, we start by proving one of the isogeny-torsion graphs correspond to an infinite set of \textit{j}-invariants. Then we use a quadratic twist to prove the other isogeny-torsion graph corresponds to an infinite set of \textit{j}-invariants. Proving the isogeny-torsion graphs of $L_{3}(9)$ correspond to infinite sets of \textit{j}-invariants is done the same way.
		
		\begin{table}[h!]
	\renewcommand{\arraystretch}{1.25}
	\begin{tabular}{ |c|c|c|c| }
		\hline
		Isogeny Graph & Type & Torsion Configuration & Label \\
		\hline
		
		\multirow{5}*{$E_{1}\myiso E_{2}\myiso E_{3}$} & \multirow{2}*{$L_{3}(25)$} & $([5],[5],[1])$ & $\mathcal{L}_{3}^{1}(25)$ \\
		\cline{3-4}
		& & $([1],[1],[1])$ & $\mathcal{L}_{3}^{2}(25)$ \\
		\cline{2-4}
		& \multirow{3}*{$L_{3}(9)$} & $([9],[3],[1])$ & $\mathcal{L}_{3}^{1}(9)$ \\
		\cline{3-4}
		& & $([3],[3],[1])$ & $\mathcal{L}_{3}^{2}(9)$\\
		\cline{3-4}
		& & $([1],[1],[1])$ & $\mathcal{L}_{3}^{3}(9)$ \\
		\hline
		\end{tabular}
		\caption{Isogeny Graphs of $L_{3}$ Type}
        \label{L_{3} Graphs}
		\end{table}
		
		\begin{itemize}
		
		    \item Isogeny-Torsion Graphs of $L_{3}(25)$ Type
		    
		    \begin{enumerate}
		    
		    \item $\mathcal{L}_{3}^{1}(25)$
		    
		    Let $E / \Q$ be an elliptic curve such that $E$ contains two $\Q$-rational subgroups of order $5$. Then the isogeny-torsion graph associated to the $\Q$-isogeny class of $E$ is $\mathcal{L}_{3}^{1}(25)$ if and only if $E(\Q)_{\text{tors}} \cong \Z / 5 \Z$ or in other words, if and only if the image of the mod-$5$ Galois representation attached to $E$ is conjugate to
		    $$H = \left\langle \left(\begin{array}{cc}
		        1 & 0 \\
		        0 & 2
		    \end{array}\right) \right\rangle.$$
		    In this case, the image of the mod-$5$ Galois representation attached to $E$ is an abelian extension of $\QQ$. There is a parametrization of elliptic curves over $\Q$ such that the image of the mod-$5$ Galois representation is conjugate to $H$ in \cite{enriquealvaro}; $E_{t} : y^{2} = x^{3} + A(t)x + B(t)$ where
		    \begin{center}
		        $A(t) = -\frac{t^{20}-228t^{15}+494t^{10}+228t^{5}+1}{48}$ and $B(t) = \frac{t^{30} + 522t^{25} - 10005t^{20} - 1005t^{10} - 522t + 1}{864}$.
		    \end{center}
		    Conversely, let $t \in \QQ$ such that $E_{t}$ is an elliptic curve. Then the isogeny-torsion graph associated to the $\QQ$-isogeny class of $E_{t}$ is $\mathcal{L}_{3}^{1}(25)$. Using the infinite one-parameter family of elliptic curves over $\Q$, $E_{t}$, we prove that $\mathcal{L}_{3}^{1}(25)$ corresponds to an infinite set of \textit{j}-invariants.
		    
		    \item $\mathcal{L}_{3}^{2}(25)$
		    
		    Let $E / \Q$ be an elliptic curve such that $E$ contains two $\Q$-rational subgroups of order $5$. Then the isogeny-torsion graph associated to the $\Q$-isogeny class of $E$ is $\mathcal{L}_{3}^{2}(25)$ if and only if $E$ has trivial rational torsion. As $H$ is a cyclic group of order $4$, $E_{t}$ has a single subgroup of index $2$. Thus, $\QQ(E_{t}[5])$ has a single quadratic subfield, namely $\QQ(\sqrt{5})$. Let $E_{t}^{(d)}$ be a quadratic twist of $E_{t}$ by $d$ where $d$ is any non-zero, square-free integer not equal to $1$ or $5$. Then the image of the mod-$5$ Galois representation attached to $E_{t}^{d}$ is conjugate to $\left\langle \left(\begin{array}{cc}
		        1 & 0 \\
		        0 & 2
		    \end{array}\right), \operatorname{-Id} \right\rangle$ and the isogeny-torsion graph associated to the $\Q$-isogeny class of $E_{t}^{d}$ is $\mathcal{L}_{3}^{2}(25)$ by Lemma \ref{Quadratic Twisting Odd Graphs}. As $\mathcal{L}_{3}^{1}(25)$ corresponds to an infinite set of \textit{j}-invariants, $\mathcal{L}_{3}^{2}(25)$ also corresponds to an infinite set of \textit{j}-invariants.
		    
		    \end{enumerate}
		    
		    \item Isogeny-Torsion Graphs of $L_{3}(9)$ Type
		    
		    \begin{enumerate}
		    
		    \item $\mathcal{L}_{3}^{1}(9)$
		    
		    Let $E / \Q$ be an elliptic curve. Then the isogeny-torsion graph associated to the $\Q$-isogeny class of $E$ is $\mathcal{L}_{3}^{1}(9)$ if and only if $E$ is $\Q$-isogenous to an elliptic curve over $\Q$ with rational $9$-torsion (see Table \ref{L_{3} Graphs}). Let us say that $E(\Q)_{\text{tors}} \cong \Z / 9 \Z$. A Magma computation reveals that the image of the mod-$9$ Galois representation attached to $E$ is conjugate to
		    $$H = \left\langle \left(\begin{array}{cc}
		        1 & 0 \\
		        0 & 2
		    \end{array}\right), \left(\begin{array}{cc}
		        1 & 1 \\
		        0 & 1
		    \end{array}\right) \right\rangle.$$
		    There are infinitely many \textit{j}-invariants corresponding to elliptic curves over $\Q$ with rational $9$-torsion. One such one-parameter family of elliptic curves is
		    $$E_{a,b}(t) : y^{2} + (1-a)xy - by = x^{3} - bx^{2}$$
		    where $a = t^{2}(t-1)$ and $b = t^{2}(t-1)(t^{2}-t+1)$ (see appendix E of \cite{alrbook}). Thus, $\mathcal{L}_{3}^{1}(9)$ corresponds to an infinite set of \textit{j}-invariants.
		    
		    \item $\mathcal{L}_{3}^{2}(9)$
		    
		    Let $E_{a,b}^{(-3)}(t)$ be the quadratic twist of $E_{a,b}(t)$ by $-3$. Then the image of the mod-$3$ Galois representation attached to $E_{a,b}^{(-3)}(t)$ is conjugate to
		    $$\overline{H'} = \left\langle -\left(\begin{array}{cc}
		        1 & 0 \\
		        0 & 2
		    \end{array}\right), \left(\begin{array}{cc}
		        1 & 1 \\
		        0 & 1
		    \end{array}\right) \right\rangle = \left\langle \left(\begin{array}{cc}
		        2 & 0 \\
		        0 & 1
		    \end{array}\right), \left(\begin{array}{cc}
		        1 & 1 \\
		        0 & 1
		    \end{array}\right) \right\rangle.$$
		    Lifting $\overline{H'}$ to level $9$, we see that the image of the mod-$9$ Galois representation attached to $E_{a,b}^{(-3)}(t)$ is conjugate to
		    $$H' = \left\langle -\left(\begin{array}{cc}
		        1 & 0 \\
		        0 & 2
		    \end{array}\right), \left(\begin{array}{cc}
		        1 & 1 \\
		        0 & 1
		    \end{array}\right) \right\rangle = \left\langle \left(\begin{array}{cc}
		        8 & 0 \\
		        0 & 7
		    \end{array}\right), \left(\begin{array}{cc}
		        1 & 1 \\
		        0 & 1
		    \end{array}\right) \right\rangle.$$
		    We will prove that the isogeny-torsion graph attached to the $\Q$-isogeny class of $E = E_{a,b}^{(-3)}(t)$ is $\mathcal{L}_{3}^{2}(9)$. Let $P_{9}$ and $Q_{9}$ be points on $E$ such that $\left\langle P_{9} \right\rangle \bigcap \left\langle Q_{9} \right\rangle = \{\mathcal{O}\}$ (equivalently, $P_{9}, Q_{9}$ form a basis of $E[9]$) and the image of the mod-$9$ Galois representation attached to $E$ is $H'$. Let $\sigma_{1}$ and $\sigma_{2}$ be a pair of Galois automorphisms such that $\sigma_{1}(P_{9}) = [8]P_{9}$ and $\sigma_{2}$ fixes $P_{9}$ and $\sigma_{2}(Q_{9}) = P_{9} + Q_{9}$ and $\sigma_{1}(Q_{9}) = [7]Q_{9}$. Then $\sigma_{1}$ and $\sigma_{2}$ generate the image of the mod-$9$ Galois representation attached to $E$. Note that the group $\left\langle P_{9} \right\rangle$ is $\Q$-rational. Let $\phi \colon E \to E / \left\langle P_{9} \right\rangle$ be an isogeny with kernel $\left\langle P_{9} \right\rangle$. Denote $Q_{3} = [3]Q_{9}$. We claim that $\phi(Q_{3})$ is a point of order $3$ defined over $\Q$. Note that $\sigma_{1}$ does not fix $Q_{9}$ but it does fix $Q_{3}$ as
		    $$\sigma_{1}(Q_{3}) = \sigma_{1}([3]Q_{9}) = [3]\sigma_{1}(Q_{9}) = [3]([7]Q_{9}) = [21]Q_{9} = [3]Q_{9} = Q_{3}.$$
		    Also, $\sigma_{2}(Q_{9}) = P_{9} + Q_{9}$ and hence,
		    $$\sigma_{2}(Q_{3}) = \sigma_{2}([3]Q_{9}) = [3]\sigma_{2}(Q_{9}) = [3](P_{9} + Q_{9}) = [3]P_{9} + Q_{3}.$$
		    Thus, $\sigma_{2}(Q_{3}) - Q_{3} = [3]P_{9}$. By Lemma \ref{lem-necessity-for-point-rationality}, $\phi(Q_{3})$ is a point of order $3$ defined over $\Q$.
		    
		    There are three cyclic subgroups of $E / \left\langle P_{9} \right\rangle$ of order $9$ that contain $\phi(Q_{3})$, namely, $\left\langle \phi(Q_{9}) \right\rangle$, $\left\langle \phi(P_{9} + Q_{9}) \right\rangle$, and $\left\langle \phi([2]P_{9} + Q_{9}) \right\rangle$. We will test to see if $E / \left\langle P_{9} \right\rangle$ has a point of order $9$ defined over $\Q$ by seeing whether or not $\phi(Q_{9})$, $\phi(P_{9} + Q_{9})$, or $\phi([2]P_{9} + Q_{9})$ is defined over $\Q$. Note that $\sigma_{1}(Q_{9}) - Q_{9} = [7]Q_{9} - Q_{9} = [6]Q_{9} \notin \left\langle P_{9} \right\rangle$. By Lemma \ref{lem-necessity-for-point-rationality}, $\phi(Q_{9})$ is not defined over $\Q$. A similar computation proves that $\sigma_{1}([c]P_{9} + Q_{9}) - (P_{9} + Q_{9}) \notin \left\langle P_{9} \right\rangle$ for $c = 1$ and $2$. By Lemma \ref{lem-necessity-for-point-rationality}, $\phi(P_{9} + Q_{9})$ and $\phi([2]P_{9} + Q_{9})$ are not defined over $\Q$. Hence, $E / \left\langle P_{9} \right\rangle(\Q)_{\text{tors}} \cong \Z / 3 \Z$ and we can conclude that the isogeny-torsion graph associated to the $\Q$-isogeny class of $E_{a,b}^{(-3)}(t)$ is $\mathcal{L}_{3}^{2}(9)$ (see Table \ref{L_{3} Graphs}). As $E_{a,b}(t)$ and $E_{a,b}^{(-3)}(t)$ are quadratic twists and $\mathcal{L}_{3}^{1}(9)$ corresponds to an infinite set of \textit{j}-invariants, so does $\mathcal{L}_{3}^{2}(9)$.
		    
		    \item $\mathcal{L}_{3}^{3}(9)$
		    
		    Let $E_{a,b}^{(d)}(t)$ be the quadratic twist of $E_{a,b}(t)$ by a non-zero, square-free integer $d$ not equal to $1$ or $-3$. Then the image of the mod-$9$ Galois representation attached to $E_{a,b}^{(d)}(t)$ is conjugate to $\left\langle H, \operatorname{-Id} \right\rangle$.
		    
		    The group $\left\langle H, \operatorname{-Id} \right\rangle$ is conjugate to the group found in \cite{SZ} with label $9I^{0}-9c$. By Lemma \ref{Quadratic Twisting Odd Graphs}, $E_{a,b}^{(d)}(t)$ is not $\Q$-isogenous to an elliptic curve over $\Q$ with a point of order $3$ defined over $\Q$. Thus, we can conclude that the isogeny-torsion graph associated to the $\Q$-isogeny class of $E_{a,b}^{(d)}(t)$ is $\mathcal{L}_{3}^{3}(9)$. As $E_{a,b}(t)$ and $E_{a,b}^{(d)}(t)$ are quadratic twists and $\mathcal{L}_{3}^{1}(9)$ corresponds to an infinite set of \textit{j}-invariants, so does $\mathcal{L}_{3}^{3}(9)$.
		    
		    \end{enumerate}
		    \end{itemize}
		    
		    This concludes the proof of Proposition \ref{L3 Graphs Proposition}.
		    
		    \subsection{Isogeny-Torsion Graphs of $L_{2}$ Type}
		    In this subsection, we prove that the isogeny-torsion graphs of $L_{2}(13)$, $L_{2}(7)$, $L_{2}(5), L_{2}(3)$, and $L_{2}(2)$ type each correspond to infinite sets of \textit{j}-invariants.
		    
		    \begin{proposition}\label{L2 Graphs Proposition}
		    Let $\mathcal{G}$ be an isogeny-torsion graph of $L_{2}(13), L_{2}(7), L_{2}(5), L_{2}(3),$ or $L_{2}(2)$ type (regardless of torsion configuration). Then $\mathcal{G}$ corresponds to an infinite set of \textit{j}-invariants.
		    \end{proposition}
		  
		    We break the proof down by cases. The proof that the isogeny-torsion graph of $L_{2}(13)$ type corresponds to an infinite set of \textit{j}-invariants is a matter of looking up groups in \cite{SZ}. The proof that the isogeny-torsion graphs of $L_{2}(7)$ type correspond to infinitely many \textit{j}-invariants will be done by proving one such isogeny-torsion graph of $L_{2}(7)$ type corresponds to an infinite set of \textit{j}-invariants and then using an appropriate quadratic twist to prove the other isogeny-torsion graph of $L_{2}(7)$ type corrresponds to an infinite set of \textit{j}-invariants. The proof that the isogeny-torsion graphs of $L_{2}(2), L_{2}(3),$ and $L_{2}(5)$ type each correspond to infinite sets of \textit{j}-invariants is an application of Hilbert's Irreduciblity Theorem.
		    
		    \begin{center}
		    \begin{table}[h!]
	\renewcommand{\arraystretch}{1.25}
	\begin{tabular}{|c|c|c|c|}
	\hline
		Isogeny Graph & Type & Isomorphism Types & Label \\
		\hline
		    \multirow{8}*{$E_1 \myiso E_2$} & {$L_{2}(13)$} & $([1],[1])$ & $\mathcal{L}_{2}(13)$ \\
		\cline{2-4} 
		& \multirow{2}*{$L_{2}(7)$} & $([7],[1])$ & $\mathcal{L}_{2}^{1}(7)$ \\
		\cline{3-4}
		& & $([1],[1])$ & $\mathcal{L}_{2}^{2}(7)$ \\
		\cline{2-4} 
		& \multirow{2}*{$L_{2}(5)$} & $([5],[1])$ & $\mathcal{L}_{2}^{1}(5)$ \\
		\cline{3-4}
		& & $([1],[1])$ & $\mathcal{L}_{2}^{2}(5)$ \\
		\cline{2-4} 
		& \multirow{2}*{$L_{2}(3)$} & $([3],[1])$ & $\mathcal{L}_{2}^{1}(3)$ \\
		\cline{3-4}
		& & $([1],[1])$ & $\mathcal{L}_{2}^{2}(3)$ \\
		\cline{2-4}
		& $L_{2}(2)$ & $([2],[2])$ & $\mathcal{L}_{2}(2)$ \\
		\cline{2-4}
		
		\hline
	\end{tabular}
	\caption{Isogeny Graphs of $L_{2}$ Type}
        \label{L_{2} Graphs}
	\end{table}
	\end{center}
		    
		    \begin{itemize} \item $\mathcal{L}_{2}(13)$
		    
		    Let $E / \Q$ be an elliptic curve. Then $E$ has a $\Q$-rational subgroup of order $13$ if and only if the image of the mod-$13$ Galois representation attached to $E$ is conjugate to a subgroup of
		    $$H = \left\{ \left(\begin{array}{cc}
		        \ast & \ast \\
		        0 & \ast
		    \end{array}\right)\right\}.$$
		    The group $H$ appears in the list compiled in \cite{SZ} as $13A^{0}-13a$. Note that the modular curve defined by $13A^{0}-13a$ is isomorphic to $\operatorname{X}_{0}(13)$, which is a genus $0$ curve with infinitely many non-cuspidal, $\Q$-rational points (see Theorem \ref{thm-ratnoncusps}). Elliptic curves over $\Q$ with a $\Q$-rational subgroup of order $13$ correspond to non-cuspidal, $\Q$-rational points on the modular curve defined by $H$. As there is a single isogeny-torsion graph of $L_{2}(13)$ type, it must correspond to an infinite set of \textit{j}-invariants.
		    
		    \item Isogeny-Torsion Graphs of $L_{2}(7)$ Type
		    
		    \begin{enumerate}
		    \item $\mathcal{L}_{2}^{1}(7)$
		    
		    Let $E / \Q$ be an elliptic curve such that the isogeny-torsion graph of $E$ is of $L_{2}(7)$ type. There are two possible torsion configurations. The isogeny-torsion graph associated to the $\Q$-isogeny class of $E$ is $\mathcal{L}_{2}^{1}(7)$ if and only if $E$ is $\Q$-isogenous to an elliptic curve over $\Q$ with rational $7$-torsion.
		    
		    Let $E_{a,b}(t) : y^{2} + (1-a)xy - by = x^{3}-bx^{2}$ where
		    \begin{center}
		        $a = t^{2} -t$ and $b = t^{3}-t^{2}$.
		    \end{center}
		    
		    If $E_{a,b}(t)$ is an elliptic curve, then $E_{a,b}(t)(\Q)_{\text{tors}} \cong \Z / 7 \Z$ (see appendix E of \cite{alrbook}). Using this one-parameter family of elliptic curves over $\Q$ with rational $7$-torsion, we can conclude that $\mathcal{L}_{2}^{1}(7)$ corresponds to an infinite set of \textit{j}-invariants.
		    
		    \item $\mathcal{L}_{2}^{2}(7)$
		    
		    By Lemma \ref{Rational Points} the image of the mod-$7$ Galois representation attached to $E_{a,b}(t)$ is conjugate to $\mathcal{B}_{7}$, the subgroup of $\operatorname{GL}(2, \Z / 7 \Z)$ of matrices of the form $\left(\begin{array}{cc}
		        1 & x \\
		        0 & y
		    \end{array}\right)$. By Corollary \ref{Corollary Unique Index 2 Subgroup}, $\QQ(E[7])$ contains a single quadratic subfield, $\QQ\left(\sqrt{-7}\right)$. Let $E_{a,b}^{(d)}(t)$ be a quadratic twist of $E_{a,b}(t)$ by a non-zero, square-free integer $d$ not equal to $1$ or $-7$. Then the image of the mod-$7$ Galois representation attached to $E_{a,b}^{(d)}(t)$ is conjugate to $\left\langle \mathcal{B}_{7}, \operatorname{-Id} \right\rangle$. By Corollary \ref{Corollary Quadratic Twisting Odd Graphs}, none of the elliptic curves over $\Q$ that are $\Q$-isogenous to $E_{a,b}^{(d)}(t)$ have a point of order $7$ defined over $\Q$. Hence, the isogeny-torsion graph associated to the $\Q$-isogeny class of $E_{a,b}^{(d)}(t)$ is $\mathcal{L}_{2}^{2}(7)$. As $E_{a,b}(t)$ and $E_{a,b}^{(d)}(t)$ are quadratic twists and $\mathcal{L}_{2}^{1}(7)$ corresponds to an infinite set of \textit{j}-invariants, $\mathcal{L}_{2}^{2}(7)$ also corresponds to an infinite set of \textit{j}-invariants.
		    
		    \end{enumerate}
		    
		    \item Isogeny-Torsion Graphs of $L_{2}(5)$ Type
		    
		    An isogeny-torsion graph of $L_{2}(5)$ type is a proper subgraph of an isogeny-torsion graph of $L_{3}(25)$ type, an isogeny graph of $R_{4}(15)$ type, and an isogeny-torsion graph of $R_{4}(10)$ type. As there are finitely many \textit{j}-invariants corresponding to elliptic curves over $\Q$ with isogeny graph of $R_{4}(15)$ type (see Proposition \ref{Finite Graphs Proposition} and Theorem \ref{thm-kenku}), we can just focus on isogeny-torsion graphs of $L_{2}(5)$ type, $L_{3}(25)$ type, and $R_{4}(10)$ type.

		    After appropriate modifications, a very similar analysis that we did in subsection 9.2 to determine the two isogeny-torsion graphs of $R_{4}(6)$ type correspond to infinite sets of \textit{j}-invariants can be used to prove that one of the two isogeny-torsion graphs of $L_{2}(5)$ type, $\mathcal{L}_{2}^{1}(5)$ or $\mathcal{L}_{2}^{2}(5)$, correspond to infinite sets of \textit{j}-invariants. Now we must prove that both $\mathcal{L}_{2}^{1}(5)$ and $\mathcal{L}_{2}^{2}(5)$ correspond to infinitely many \textit{j}-invariants using quadratic twists.
		    
		    Let $E / \Q$ be an elliptic curve such that $E(\Q)_{\text{tors}} \cong \Z / 5 \Z$ and $C_{5}(E) = C(E) = 2$. Then the isogeny-torsion graph associated to the $\QQ$-isogeny class of $E$ is $\mathcal{L}_{2}^{1}(5)$. By Lemma \ref{Rational Points}, the image of the mod-$5$ Galois representation is conjugate to $\mathcal{B}_{5}$, the subgroup of $\operatorname{GL}(2, \Z / 5 \Z)$ consisting of matrices of the form $\left(\begin{array}{cc}
		        1 & x \\
		        0 & y
		    \end{array}\right)$ and by Lemma \ref{Unique Index 2 Subgroup}, $\mathcal{B}_{5}$ contains a single subgroup of index $2$. Using an application of Hilbert's irreducibility theorem, we can show that there are infinitely many \textit{j}-invariants that correspond to elliptic curves $E / \QQ$ such that the image of the mod-$5$ Galois representation attached to $E$ is conjugate to $\mathcal{B}_{5}$ and $C_{5}(E) = C(E) = 2$. Thus, $\mathcal{L}_{2}^{1}(5)$ corresponds to an infinite set of \textit{j}-invariants. Now let $d$ be a non-zero, square-free integer not equal to $1$ or $5$ and let $E^{(d)}$ be the quadratic twist of $E$ by $d$. Then, the image of the mod-$5$ Galois representation attached to $E^{(d)}$ is conjugate to $\left\langle \mathcal{B}_{5}, \operatorname{-Id} \right\rangle$ and by Corollary \ref{Corollary Quadratic Twisting Odd Graphs}, the isogeny-torsion graph associated to $E^{(d)}$ is $\mathcal{L}_{2}^{2}(5)$.
		    
		    We proved we can toggle amongst the two isogeny-torsion graphs of $L_{2}(5)$ type. Thus, both correspond to an infinite set of \textit{j}-invariants.
		    
		    \item Isogeny-Torsion Graphs of $L_{2}(3)$ Type
		    
		    An isogeny-torsion graph of $L_{2}(3)$ type is a proper subgraph of an isogeny-torsion graph of $L_{3}(9)$ type, an isogeny-torsion graph of $L_{4}$ type, an isogeny-torsion graph of $R_{4}(6)$ type, an isogeny-torsion graph of $R_{4}(15)$ type, an isogeny-torsion graph of $R_{4}(21)$ type, an isogeny-torsion graph of $R_{6}$ type, and an isogeny-torsion graph of $S$ type. As there are finitely many \textit{j}-invariants corresponding to isogeny-torsion graphs of $L_{4}$, $R_{4}(15)$, and $R_{4}(21)$ type (see Proposition \ref{Finite Graphs Proposition} and Theorem \ref{thm-kenku}), we can just focus on the isogeny-torsion graphs of $L_{2}(3)$ type, isogeny-torsion graphs of $L_{3}(9)$ type, isogeny-torsion graphs of $R_{4}(6)$ type, isogeny-torsion graphs of $R_{6}$ type, and isogeny-torsion graphs of $S$ type.
		    
		    After appropriate modifications, a very similar analysis that we did earlier to prove that both $\mathcal{L}_{2}^{1}(5)$ and $\mathcal{L}_{2}^{2}(5)$ correspond to infinite sets of \textit{j}-invariants can be used to prove that both $\mathcal{L}_{2}^{1}(3)$ and $\mathcal{L}_{2}^{2}(3)$ correspond to infinite sets of \textit{j}-invariants.

    \item $\mathcal{L}_{2}(2)$
    
    The isogeny-torsion graph of $L_{2}(2)$ type is a proper subgraph of an isogeny-torsion graph of $R_{4}(6)$ type, an isogeny-torsion graph of $R_{4}(10)$ type, the isogeny-torsion graph of $R_{4}(14)$ type, an isogeny-torsion graph of $T_{4}$ type, an isogeny-torsion graph of $T_{6}$ type, an isogeny-torsion graph of $T_{8}$ type, and an isogeny-torsion graph of $S$ type. As there are finitely many \textit{j}-invariants corresponding to the isogeny-torsion graphs of $R_{4}(14)$ type (see Proposition \ref{Finite Graphs Proposition} and Theorem \ref{thm-kenku}), we only need to focus on the other isogeny-torsion graphs that properly contain the isogeny-torsion graph of $L_{2}(2)$ type.
    
    After appropriate modifications, a very similar analysis that we did to prove that both isogeny-torsion graphs of $L_{2}(5)$ type correspond to infinite sets of \textit{j}-invariants can be used to determine that the isogeny-torsion graph of $L_{2}(2)$ type corresponds to an infinite set of \textit{j}-invariants.
		\end{itemize}
		
		 This concludes the proof of Proposition \ref{L2 Graphs Proposition}.
		
		\subsection{The Isogeny-Torsion Graph of $L_{1}$ Type}
		
		\begin{proposition}\label{L1 Graphs Proposition}
		The isogeny-torsion graph of $L_{1}$ type corresponds to an infinite set of \textit{j}-invariants.
		\end{proposition}
		
		Let $E / \Q$ be an elliptic curve. Then the isogeny-torsion graph associated to the $\Q$-isogeny class of $E$ is of $L_{1}$ type if and only if the $\Q$-isogeny class of $E$ consists only of $E$. In other words, $E$ does not have any non-trivial, finite, cyclic, $\Q$-rational subgroups. We need to prove there are infinitely many \textit{j}-invariants corresponding to elliptic curves over $\Q$ with no non-trivial, finite, cyclic, $\Q$-rational subgroups. By Theorem \ref{thm-kenku} and Proposition \ref{Finite Graphs Proposition}, it suffices to consider the cases when $E$ has a $\Q$-rational subgroup of order $2$, $3$, $5$, $7$, or $13$.
		
		For each $p = 2$, $3$, $5$, $7$, $13$, let $\textit{j}_{p}$ be the $\textit{j}$-invariant of some elliptic curve over $\Q$ such that the image of the mod-$p$ Galois representation is conjugate to a subgroup of the group of upper triangular matrices of $\operatorname{GL}(2, \Z / p \Z)$. For each prime $p = 2, 3, 5, 7, 13$, we have $\textit{j}_{p} = \frac{F_{p}(x)}{G_{p}(x)}$ for some $F_{p}(x), G_{p}(x) \in \Z[x]$ with $\text{deg}(F) - \text{deg}(G) \geq 2$. For each prime $p = 2, 3, 5, 7, 13$, let $H_{p}$ denote the set of all rational numbers $t$, such that $G_{p}(x)t - F_{p}(x)$ is irreducible. If the \textit{j}-invariant of an elliptic curve $E / \Q$ is an element of $H_{p}$, then $E$ does not have a $\Q$-rational subgroup of order $p$. For each $p = 2, 3, 5, 7, 13$, $H_{p}$ is a basic Hilbert set and thus, $H = \bigcap_{2 \leq p \leq 13} H_{p}$ is a Hilbert set and is dense in $\Q$. Thus, there are infinitely many \textit{j}-invariants (elements of $H$) that correspond to elliptic curves over $\Q$ which do not have $\Q$-rational subgroups of order $2$, $3$, $5$, $7$, or $13$. Hence, the isogeny-torsion graph of $L_{1}$ type corresponds to an infinite set of \textit{j}-invariants. This concludes the proofs of Proposition \ref{L1 Graphs Proposition} and Theorem \ref{main thm}.

\bibliography{bibliography}{}
\bibliographystyle{plain}

\end{document}